\documentclass[12pt,letter]{amsart}
\pdfoutput=1

\usepackage[centering,text={16.5cm,24cm}]{geometry}

\usepackage[T1]{fontenc}
\usepackage[latin1]{inputenc}
\usepackage{lmodern}
\usepackage[english]{babel}

\usepackage{amsmath}
\usepackage{amssymb}
\usepackage{amsthm}
\usepackage{mathrsfs}
\usepackage{mathtools}
\usepackage{stmaryrd}

\usepackage{hhline}
\usepackage{paralist}
\usepackage{setspace}
\usepackage{hyperref}
\usepackage{braket}
\usepackage[all]{xy}
\usepackage{wrapfig}
\usepackage{bbm}

\usepackage{tikz}
\usepackage{tikz-3dplot}
\usepackage{pgfplots}
\usepackage{rotating}
\usepackage{tikz}
\usepackage[outline]{contour}
\usetikzlibrary{matrix,shapes,arrows,arrows.meta,calc,topaths,intersections,hobby,positioning,decorations.pathreplacing,decorations.pathmorphing,fit,patterns}
\tikzset{cross/.style={cross out, draw=black, fill=none, minimum size=2*(#1-\pgflinewidth), inner sep=0pt, outer sep=0pt}, cross/.default={2pt}}
\tdplotsetmaincoords{70}{110}
\tikzset{>=latex}
\usepackage{tkz-euclide}
\usetikzlibrary{shapes.misc}
\tikzset{cross/.style={cross out, draw=black, fill=none, minimum size=2*(#1-\pgflinewidth), inner sep=0pt, outer sep=0pt}, cross/.default={2pt}}

\newcommand{\rotateRPY}[3]% roll, pitch, yaw
{   \pgfmathsetmacro{\rollangle}{#1}
    \pgfmathsetmacro{\pitchangle}{#2}
    \pgfmathsetmacro{\yawangle}{#3}

    \pgfmathsetmacro{\newxx}{cos(\yawangle)*cos(\pitchangle)}
    \pgfmathsetmacro{\newxy}{sin(\yawangle)*cos(\pitchangle)}
    \pgfmathsetmacro{\newxz}{-sin(\pitchangle)}
    \path (\newxx,\newxy,\newxz);
    \pgfgetlastxy{\nxx}{\nxy};

    \pgfmathsetmacro{\newyx}{cos(\yawangle)*sin(\pitchangle)*sin(\rollangle)-sin(\yawangle)*cos(\rollangle)}
    \pgfmathsetmacro{\newyy}{sin(\yawangle)*sin(\pitchangle)*sin(\rollangle)+ cos(\yawangle)*cos(\rollangle)}
    \pgfmathsetmacro{\newyz}{cos(\pitchangle)*sin(\rollangle)}
    \path (\newyx,\newyy,\newyz);
    \pgfgetlastxy{\nyx}{\nyy};

    \pgfmathsetmacro{\newzx}{cos(\yawangle)*sin(\pitchangle)*cos(\rollangle)+ sin(\yawangle)*sin(\rollangle)}
    \pgfmathsetmacro{\newzy}{sin(\yawangle)*sin(\pitchangle)*cos(\rollangle)-cos(\yawangle)*sin(\rollangle)}
    \pgfmathsetmacro{\newzz}{cos(\pitchangle)*cos(\rollangle)}
    \path (\newzx,\newzy,\newzz);
    \pgfgetlastxy{\nzx}{\nzy};
}
\tikzset{RPY/.style={x={(\nxx,\nxy)},y={(\nyx,\nyy)},z={(\nzx,\nzy)}}}

\newtheorem{prop}{Proposition}[section]

\newtheorem{cor}[prop]{Corollary}
\newtheorem{lem}[prop]{Lemma}
\newtheorem{conj}[prop]{Conjecture}

\newtheorem{thmintro}{Theorem}
\newtheorem{conjintro}{Conjecture}

\theoremstyle{definition}
\newtheorem{defi}[prop]{Definition}
\newtheorem{con}[prop]{Construction}
\newtheorem{expl}[prop]{Example}

\newtheorem*{defiintro}{Definition}
\newtheorem*{explintro}{Example}

\theoremstyle{remark}
\newtheorem{rem}[prop]{Remark}

\newcommand{\bigslant}[2]{{\raisebox{.2em}{$#1$}\left/\raisebox{-.2em}{$#2$}\right.}}

\newcommand{\iso}{\xrightarrow{\raisebox{-0.7ex}[0ex][0ex]{$\sim$}}}

\begin{document}

\title[Smoothings from log resolutions and divisorial extractions]{Smoothings from zero mutable Laurent polynomials via log resolutions and divisorial extractions}

\author{Tim Gr\"afnitz}
\address{Leibniz-Universit\"at Hannover \\ Institut f\"ur Algebraische Geometrie \\ Welfengarten 1, 30167 Hannover \\ Germany} 
\email{graefnitz@math.uni-hannover.de}

\begin{abstract}
A conjecture by Corti, Filip and Petracci \cite{CFP}, inspired by mirror symmetry, states that smoothing types of affine Gorenstein toric $3$-folds correspond to zero mutable Laurent polynomials. We propose a method to prove this conjecture via log crepant log resolutions constructed from compatible collections of divisorial extractions. For affine cones over weighted projective planes $\mathbb{P}(1,a,b)$ we prove for several infinite families of zero mutable Laurent polynomials that they indeed describe curves that admit a compatible collection of divisorial extractions. The construction of log crepant log resolutions and smoothings will be worked out in joint work with Alessio Corti and Helge Ruddat, starting with \cite{CGR}.
\end{abstract}

\maketitle

\thispagestyle{empty}

\setcounter{tocdepth}{1}

\renewcommand{\baselinestretch}{1.2}\normalsize
\tableofcontents
\renewcommand{\baselinestretch}{1.25}\normalsize

\section*{Introduction}														%%%

There is a deep connection between $\mathbb{Q}$-Gorenstein deformations of toric varieties and mutations of the associated polytopes and Laurent polynomials. 
This relationship manifests in several ways, which we will briefly outline before presenting a new approach to these problems and introducing the main conjectures and results of this paper.

A projective Gorenstein toric Fano variety is defined by the face fan of a reflexive polytope. If two reflexive polytopes are related by mutations, their corresponding toric varieties are $\mathbb{Q}$-Gorenstein deformation equivalent (\cite{Ilt}, Theorem 1.3). It is conjectured that $\mathbb{Q}$-Gorenstein equivalence classes of projective Fano varieties with mild singularities correspond to mutation equivalence classes of rigid maximally mutable Laurent polynomials (\cite{CKPT}, Conjecture 5.1).

In this paper, we focus on the affine case. The affine charts of a projective Gorenstein toric Fano variety correspond to the cones over the facets of the associated reflexive polytope. More generally, an affine Gorenstein toric variety $X_0$ is defined by the cone over a polytope $\triangle$ placed at height 1. A key aspect of the definition of maximally mutable Laurent polynomials is that their restrictions to each facet can be mutated to the trivial Laurent polynomial $1$, supported on a single point. This naturally leads to the concept of zero mutable Laurent polynomials (Definition \ref{defi:zeromut}) and the following conjecture.

Let $X_0$ be a affine Gorenstein toric variety defined by the cone over the polytope $\triangle$.

\begin{defiintro}
Let $\textup{Def}(X_0)$ denote the versal $\mathbb{Q}$-Gorenstein deformation space of $X_0$. A connected component of $\textup{Def}(X_0)$ is called a \emph{smoothing component} if it contains an element corresponding to a variety $X$ with only terminal singularities. Such a variety $X$ is called a \emph{smoothing} of $X_0$. Note that in the $3$-dimensional case there are no Gorenstein terminal quotient singularities and $X$ is smooth.
\end{defiintro}

\begin{conjintro}[\cite{CFP}, Conjecture A]
\label{conj:affine}
Smoothing components of $\textup{Def}(X_0)$ are in bijection with zero mutable Laurent polynomials $f$ with Newton polytope $\triangle$.
\end{conjintro}

If $X_0$ has only an isolated singularity, then all edges of $\triangle$ have length $1$ and zero mutable Laurent polynomials correspond to maximal Minkowski decompositions of $\triangle$ via part (iii) of Definition \ref{defi:zeromut}. This special case of Conjecture \ref{conj:affine} was proved in \cite{Alt}.

\subsection*{The proposal: smoothings from log resolutions}

Let $X_0$ be an affine Gorenstein toric $3$-fold defined by the cone $\sigma$ over a polygon $\triangle$ (placed at height $1$).
\begin{enumerate}[(1)]
\item A zero mutable Laurent polynomial $f$ with Netwon polytope $\triangle$ defines, for each edge $e$ of $\triangle$, a decreasing partition $\mathbf{divstep}_e(f)$ of the lattice length $\ell(e)$ by its \emph{divisibility steps} (see Definition \ref{defi:divstep} and Proposition \ref{prop:zeromut}). Denote the dual partitions by $\mathbf{divstep}_e(f)^\vee$.
\item For any collection $(\mathbf{p}_e)_e$ of partitions $\mathbf{p}_e$ of $\ell(e)$, we build a simple normal crossings space $\cup X_i$ together with a log structure on $\cup X_i \setminus C$, where $C$ has codimension $2$, as follows.
\begin{enumerate}[(i)]
\item The central subdivision of the dual cone $\sigma^\vee$ induces a degeneration of $X_0$ into a union of toric varieties $\cup X_i$ (Construction \ref{con:toricdeg}). The intersections of irreducible components $X_i \cap X_j$ correspond to edges $e$ of $\triangle$.
\item For each edge $e$, let $\mathbf{p}_e=(d_1,\ldots,d_r)$ be a partition of the lattice length $\ell(e)$. Let $C_{ij} \subset X_i\cap X_j \simeq \mathbb{A}^2$ be a curve of degree $\ell(e)$ that intersects $\cap_k X_k  \simeq \mathbb{A}^1 \subset X_i \cap X_j$ at a single point $P$ such that locally at $P$ it has $r$ branches whose intersection multiplicities with $\cap_k X_k$ at $P$ are given by $d_1,\ldots,d_r$ (Construction \ref{con:singlog}). 
\item Following \cite{GSdata}, we can construct a log structure on $\cup X_i \setminus \cup C_{ij}$ which is log smooth over the standard log point (Proposition \ref{prop:joint}). 
\end{enumerate}
\item Under a compatibility condition, we construct a smoothing of $X_0$ as follows.
\begin{compactenum}[(i)]
\item Construct a log crepant log resolution of $\cup X_i \setminus \cup C_{ij}$ by a compatible collection of divisorial extractions $Y_i\rightarrow X_i$ of the log singular locus $C=\cup C_{ij}$. 
\item Use \cite{CR} to glue the divisorial extractions to a toroidal crossing space $\cup Y_i$ with log structure that is (away from $P=\cap C_{ij}$) smooth over the standard log point. 
\item Generalize \cite{FFR} to obtain a smoothing of $\cup Y_i$, hence a smoothing of $X_0$.
\end{compactenum}
That steps (i)-(iii) above work is the content of the conjectures below.
\end{enumerate}

\begin{conjintro}
\label{conj:main}
If the collection of partitions $(\mathbf{p}_e)_e$ in part (2) of the proposal above is given by $\mathbf{divstep}_e(f)^\vee$ for a ZMLP $f$ as in part (1), then there exists a compatible collection $Y_i\rightarrow X_i$ of divisorial extractions of $\cup C_{ij}\subset\cup X_i$ (Definition \ref{defi:compatible}).
\end{conjintro}

\begin{conjintro}[c.f. \cite{CGR}, Conjecture 18]
\label{conj:log}
Let $\cup X_i$ and $\cup C_{ij}$ be as in part (2) of the proposal. 
\begin{compactenum}[(1)]
\item A compatible collection $Y_i\rightarrow X_i$ of divisorial extractions of $\cup C_{ij}$ in $\cup X_i$ gives a generically toroidal crossing space $\cup Y_i$ that is log smooth over the standard log point away from a single point $P$, where \'etale locally it is isomorphic to $(xy=0)\subset\tfrac{1}{r}(1,-1,a,-a)$.
\item $\cup Y_i$ can be smoothed to an orbifold with terminal singularities $X$.
\end{compactenum}
\end{conjintro}

\subsection*{Results for rectangular triangles}

In this paper we will focus on the case when $X_0$ is the affine cone over a weighted projective plane $\mathbb{P}(1,a,b)$. Then $\triangle$ is a rectangular triangle.

\begin{defiintro}
A lattice polygon $\triangle$ is a \emph{rectangular triangle} if it can be mapped to the standard triangle $\triangle(a,b)=\textup{Conv}\{(0,0),(a,0),(0,b)\}$ for some $a,b\ge1$ under a transformation in the integral affine transformation group $\textup{GL}_2(\mathbb{Z})\ltimes\mathbb{Z}^2$. Equivalently, $\triangle$ is a lattice triangle with a vertex whose primitive integral outgoing edge direction vectors have determinant $1$. We call a rectangular triangle $\triangle(a,b)$ \emph{primitive} if $\gcd(a,b)=1$.
\end{defiintro}

\begin{explintro}[Tom and Jerry]
The central subdivision of  $\triangle(2,3)$ gives a toric degeneration with three components in the central fiber identified as $X_1=\mathbb{A}^3$, $X_2=A_1\times\mathbb{A}^1$ and $X_3=A_2\times\mathbb{A}^2$.
There exist precisely two distinct zero mutable Laurent polynomials for $\triangle(2,3)$ that we refer to as Tom and Jerry. They are displayed together with their coefficient decorated Newton polygons in 
Figure \ref{fig:intro1}. 

\begin{figure}[h!]
\centering
\begin{tikzpicture}[scale=.9]
\draw (1.5,2.5) node{Tom};
\draw (0,0) -- (0,2) -- (3,0) -- cycle;
\fill (0,0) circle (2pt) node[below]{$1$};
\fill (1,0) circle (2pt) node[below]{$3$};
\fill (2,0) circle (2pt) node[below]{$3$};
\fill (3,0) circle (2pt) node[below]{$1$};
\fill (0,1) circle (2pt) node[left]{$2$};
\fill (1,1) circle (2pt) node[below]{$2$};
\fill (0,2) circle (2pt) node[left]{$1$};
\draw (1.5,-1.3) node{$(1+x)^3+2y(1+x)+y^2$};
\end{tikzpicture}
\hspace{1cm}
\begin{tikzpicture}[scale=.9]
\draw (1.5,2.5) node{Jerry};
\draw (0,0) -- (0,2) -- (3,0) -- cycle;
\fill (0,0) circle (2pt) node[below]{$1$};
\fill (1,0) circle (2pt) node[below]{$3$};
\fill (2,0) circle (2pt) node[below]{$3$};
\fill (3,0) circle (2pt) node[below]{$1$};
\fill (0,1) circle (2pt) node[left]{$2$};
\fill (1,1) circle (2pt) node[below]{$3$};
\fill (0,2) circle (2pt) node[left]{$1$};
\draw (1.5,-1.3) node{$(1+y)^2+3x(1+y)+3x^2+x^3$};
\end{tikzpicture}
\caption{Zero-mutable Laurent polynomials for Tom and Jerry.}
\label{fig:intro1}
\end{figure}
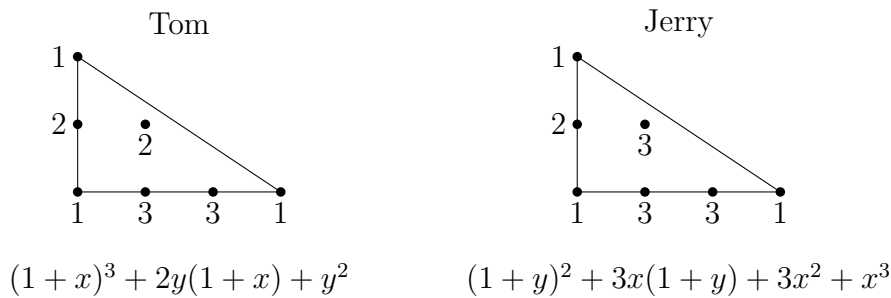

The rows of Tom from bottom to top give the polynomials  $(1+x)^3$, $2y(1+x)$ and  $y^2$ which are divisible by $(1+x)$ to order $3,1,0$ respectively. 
Taking differences, we obtain a partition of the integer length $3$ of the bottom edge as $2+1=3$. 
Transposing the corresponding partition diagram gives the dual (or conjugate) partition which is also $1+2=3$.
Doing a similar analysis on the columns of Tom yields divisibilities $2,0,0$, so a partition $2=2$ whose dual is $1+1=2$. The third edge of Tom has length $1$ so, following a similar analysis, we always just get the partition $1=1$. 
The full set of dual partitions for Tom and Jerry are
\begin{itemize}
    \item[Tom:] $(\mathbf{divstep}_e^\vee)_e=((1),(1,1),(2,1))$,
    \item[Jerry:] $(\mathbf{divstep}_e^\vee)_e=((1),(2),(1,1,1))$.
\end{itemize}     
    For both of Tom and Jerry, the degree of the curves $C_{ij}\subset X_i\cap X_j \simeq \mathbb{A}^2$ matched the corresponding edge length of the edge in $\triangle(2,3)$, so we have degrees $1$, $2$ and $3$. The actual geometry of these curves is governed by the dual partitions, i.e., for 
\begin{itemize}
    \item[Tom:] $C_{23}$ is a line, $C_{13}$ is a union of two lines and $C_{12}$ is a union of a conic and a line;
    \item[Jerry:] $C_{23}$ is a line, $C_{13}$ is a smooth conic and $C_{12}$ is a union of three lines.
\end{itemize} 

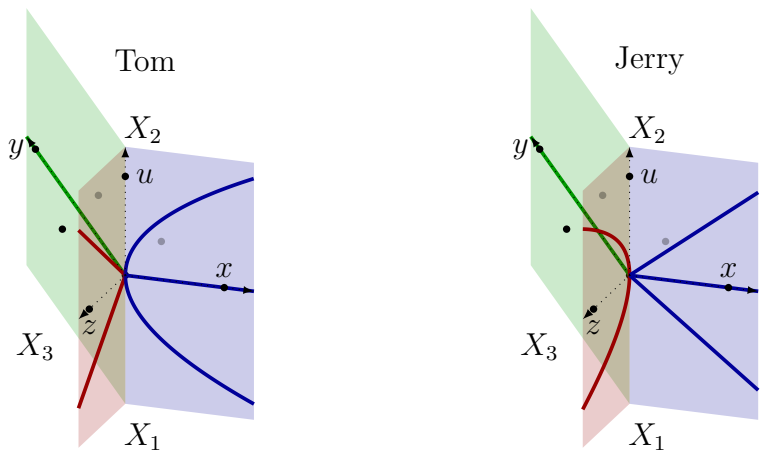
\begin{figure}[h!]
\centering
\begin{tikzpicture}[scale=1.4,tdplot_main_coords]
\draw (0,.2,2.2) node{Tom};
\fill (0,0,0) circle (1pt);
\fill[black!40!green,opacity=0.2] (0,0,-1.3) -- (-3.3,-2.2,-1.3) -- (-3.3,-2.2,1.3) -- (0,0,1.3);
\fill[black!40!blue,opacity=0.2] (0,0,-1.3) -- (0,1.3,-1.3) -- (0,1.3,1.3) -- (0,0,1.3);
\draw[thick,black!40!green,line width=1.4pt] (0,0,0) -- (-3.3,-2.2,0);
\draw[thick,black!40!red,line width=1.4pt] (0,0,0) -- (1.3,0,.9);
\draw[thick,black!40!red,line width=1.4pt] (0,0,0) -- (1.3,0,-.9);
\draw[thick,black!40!blue,line width=1.4pt,smooth,samples=100,domain=0:1.14] plot (0,\x^2,\x);
\draw[thick,black!40!blue,line width=1.4pt,smooth,samples=100,domain=-1.14:0] plot (0,-\x^2,\x);
\draw[thick,black!40!blue,line width=1.4pt] (0,0,0) -- (0,1.3,0);
\fill[black!40!red,opacity=0.2] (0,0,-1.3) -- (1.3,0,-1.3) -- (1.3,0,1.3) -- (0,0,1.3);
\draw[dotted,->] (0,0,0) -- (0,0,1.3);
\draw[dotted,->] (0,0,0) -- (-3.3,-2.2,0);
\draw[dotted,->] (0,0,0) -- (1.3,0,0);
\draw[dotted,->] (0,0,0) -- (0,1.3,0);
\fill (-1,-1,0) circle (1pt);
\fill[opacity=.3] (-1,0,0) circle (1pt);
\fill[opacity=.3] (-2,-1,0) circle (1pt);
\fill (0,0,1) circle (1pt) node[right]{$u$};
\fill (-3,-2,0) circle (1pt) node[left]{$y$};
\fill (1,0,0) circle (1pt) node[below]{$z$};
\fill (0,1,0) circle (1pt) node[above]{$x$};
\draw (-.5,0,-1.8) node{$X_1$};
\draw (-.5,0,1.3) node{$X_2$};
\draw (-3,-2,-2) node{$X_3$};
\end{tikzpicture}
\hspace{3cm}
\begin{tikzpicture}[scale=1.4,tdplot_main_coords]
\draw (0,.2,2.2) node{Jerry};
\fill (0,0,0) circle (1pt);
\fill[black!40!green,opacity=0.2] (0,0,-1.3) -- (-3.3,-2.2,-1.3) -- (-3.3,-2.2,1.3) -- (0,0,1.3);
\fill[black!40!blue,opacity=0.2] (0,0,-1.3) -- (0,1.3,-1.3) -- (0,1.3,1.3) -- (0,0,1.3);
\draw[thick,black!40!green,line width=1.4pt] (0,0,0) -- (-3.3,-2.2,0);
\draw[thick,black!40!red,line width=1.4pt,smooth,samples=100,domain=0:1.14] plot (\x^2,0,.8*\x);
\draw[thick,black!40!red,line width=1.4pt,smooth,samples=100,domain=-1.14:0] plot (-\x^2,0,.8*\x);
\draw[thick,black!40!blue,line width=1.4pt] (0,0,0) -- (0,1.3,1);
\draw[thick,black!40!blue,line width=1.4pt] (0,0,0) -- (0,1.3,0);
\draw[thick,black!40!blue,line width=1.4pt] (0,0,0) -- (0,1.3,-1);
\fill[black!40!red,opacity=0.2] (0,0,-1.3) -- (1.3,0,-1.3) -- (1.3,0,1.3) -- (0,0,1.3);
\draw[dotted,->] (0,0,0) -- (0,0,1.3);
\draw[dotted,->] (0,0,0) -- (-3.3,-2.2,0);
\draw[dotted,->] (0,0,0) -- (1.3,0,0);
\draw[dotted,->] (0,0,0) -- (0,1.3,0);
\fill (-1,-1,0) circle (1pt);
\fill[opacity=.3] (-1,0,0) circle (1pt);
\fill[opacity=.3] (-2,-1,0) circle (1pt);
\fill (0,0,1) circle (1pt) node[right]{$u$};
\fill (-3,-2,0) circle (1pt) node[left]{$y$};
\fill (1,0,0) circle (1pt) node[below]{$z$};
\fill (0,1,0) circle (1pt) node[above]{$x$};
\draw (-.5,0,-1.8) node{$X_1$};
\draw (-.5,0,1.3) node{$X_2$};
\draw (-3,-2,-2) node{$X_3$};
\end{tikzpicture}
\caption{Log singular loci for Tom and Jerry.}
\label{fig:intro2}
\end{figure}

We show that for both cases we obtain a compatible collection of divisorial extractions by both of the following sequences, see Construction \ref{con:procedure}
\begin{enumerate}[(1)]
\item 
\begin{enumerate}[(a)]
\item The blow up $Y_2\rightarrow X_2$ of $C_{23}\subset D_{23}\subset X_2=A_{a-1}\times\mathbb{A}^1$.
\item The blow up $Y\rightarrow X_1$ of $C_{12}\subset D_{12}\subset X_1=\mathbb{A}^3$.
\item A divisorial extraction $Z\rightarrow Y$ of the proper transform $\tilde{C}_{13}\subset\tilde{D}_{13}\subset Y$.
\end{enumerate}
\item 
\begin{enumerate}[(a)]
\item The blow up $Y_3\rightarrow X_3$ of $C_{23}\subset D_{23}\subset X_3=A_{b-1}\times\mathbb{A}^1$.
\item The blow up $Y\rightarrow X_1$ of $C_{13}\subset D_{13}\subset X_1=\mathbb{A}^3$.
\item A divisorial extraction $Z\rightarrow Y$ of the proper transform $\tilde{C}_{12}\subset\tilde{D}_{12}\subset Y$.
\end{enumerate}
\end{enumerate}
\end{explintro}

We classify zero mutable Laurent polynomials (ZMLPs) on several families of rectangular triangles and use this to prove Conjecture \ref{conj:main} in some cases.

\begin{thmintro}[Propositions \ref{prop:a}, \ref{prop:b}, \ref{prop:d}, \ref{prop:c}]\
\label{thm:classification}
\begin{enumerate}[(a)]
\item On $\triangle(1,k)$, i.e. an $A_{k-1}$-triangle, there is exactly $1$ ZMLP, called Tom.
\item On $\triangle(2,k)$, for $k \notin 2\mathbb{Z}$, there are exactly $2$ ZMLPs, called Tom and Jerry.
\item On $\triangle(k,k+1)$, for $k\geq 3$, there are exactly $3$ ZMLPs, called Tom, Jerry and Spike.
\item On $\triangle(3,k)$, for $k\geq 7$, $k\notin 3\mathbb{Z}$, there are exactly $4$ ZMLPs, Tom, Jerry, Spike and Tyke.
\end{enumerate}
In all cases above, there is a chain of mutations $f \mapsto f_r \mapsto \ldots \mapsto f_1 \mapsto f_0=1$ such that all steps $f_i$ are zero mutable Laurent polynomials on primitive rectangular triangles.
\end{thmintro}

\begin{thmintro}[Propositions \ref{prop:maindivex}, \ref{prop:An}, \ref{prop:tom1}, \ref{prop:tom2}, \ref{prop:jerrytyke}]
\label{thm:main}
Conjecture \ref{conj:main} holds for all cases in Theorem \ref{thm:classification} called Tom, Jerry or Tyke.
\end{thmintro}

\begin{conjintro}
\label{conj:spike}
Conjecture \ref{conj:main} holds for all cases in Theorem \ref{thm:classification} called Spike.
\end{conjintro}

\subsection*{Acknowledgements}
This paper can be seen as part of a project with Alessio Corti and Helge Ruddat to understand smoothings of toric varieties via log geometry \cite{CGR}. I am grateful to them for many ideas and discussions that helped to shape this paper.

\section{Zero mutable Laurent polynomials}											%%%

Let $M\simeq \mathbb{Z}^n$ be an affine lattice and consider a Laurent polynomial
\[ f = \sum_{m\in M} c_m z^m \in \mathbb{C}[M]\simeq\mathbb{C}[x_1^{\pm 1},\ldots,x_n^{\pm 1}]. \] 

\begin{defi}
The \emph{Newton polytope} of $f$ is the convex hull of the exponents with nonzero coefficients,
\[ \textup{Newt}(f) = \textup{Conv}(\{m \in M \ | \ c_m \neq 0 \}). \]
If $\triangle=\textup{Newt}(f)$ we say that $f$ is \emph{supported on} the polytope $\triangle$ or simply \emph{on} $\triangle$.
\end{defi}

\begin{defi}
\label{defi:mut}
Let $\varphi : M \rightarrow \mathbb{Z}$ be an integral affine linear function on $M$, let $\varphi_0$ be its linear part, and let $h \in \mathbb{C}[\textup{ker }\varphi_0]\setminus\mathbb{C}$ be a Laurent polynomial defined on the zero locus of $\varphi_0$. Write
\[ f = \sum_{k\in\mathbb{Z}} f_k \quad \textup{where} \quad f_k\in \mathbb{C}[(\varphi=k)\cap M]. \]
We call $f$ \emph{mutable} with respect to $(\varphi,h)$ if for all $k<0$ the Laurent polynomial $f_k$ is divisible by $h^{-k}$. In this case we define the \emph{mutation} of $f$ with respect to $(\varphi,h)$ to be
\[ \textup{mut}_{(\varphi,h)}(f) = \sum_{k\in\mathbb{Z}} h^kf_k. \]
This induces a notion of mutations of polytopes via the Newton polytope construction.
\end{defi}

\begin{defi}
\label{defi:zeromut}
We define \emph{zero-mutable Laurent polynomials (ZMLPs)} recursively as follows.
\begin{enumerate}[(i)]
\item A Laurent polynomial in one variable $f \in \mathbb{C}[x^{\pm1}]$ is zero mutable if it is of the form $f = (1+x^{\pm1})^k x^l$ for $k\in\mathbb{N}$ and $l\in\mathbb{Z}$. 
\item A Laurent polynomial $f\in\mathbb{C}[M]$ supported on a saturated affine sublattice $M'\subset M$ is zero mutable if it is zero mutable as a Laurent polynomial $f\in\mathbb{C}[M']$.
\item A reducible Laurent polynomial $f=f_1f_2$ is zero mutable if both factors $f_1$, $f_2$ are zero mutable.
\item An irreducible Laurent polynomial $f$ is zero mutable if it is mapped to the Laurent polynomial $1$ via a sequence of mutations $\textup{mut}_{(\varphi,h)}$ involving only zero mutable $h$.
\end{enumerate}
\end{defi}

\begin{expl}
Figure \ref{fig:graph} was produced by a \texttt{SageMath} code that can be found on the second authors webpage \href{https://timgraefnitz.com/}{\texttt{timgraefnitz.com}}. It zero mutable Laurent polynomials on polygons of size $\leq 3$. Edges indicate mutations between them. In the third column, row $2$ and $3$ from above, we see the zero mutable Laurent polynomials corresponding to Tom and Jerry.

We did not consider products, so some reducible zero mutable Laurent polynomials are missing, e.g. the two for the affine cone over $dP_6$. Also other zero mutable Laurent polynomials are missing, like $f=(1+x)^2+xy^2$, since this would occur with its isomorphic representative $f=(1+xy^2)^2+x$, which has size $4$. The picture shows that even for small polytopes, a description of all zero mutable Laurent polynomials can be very complicated.
\end{expl}

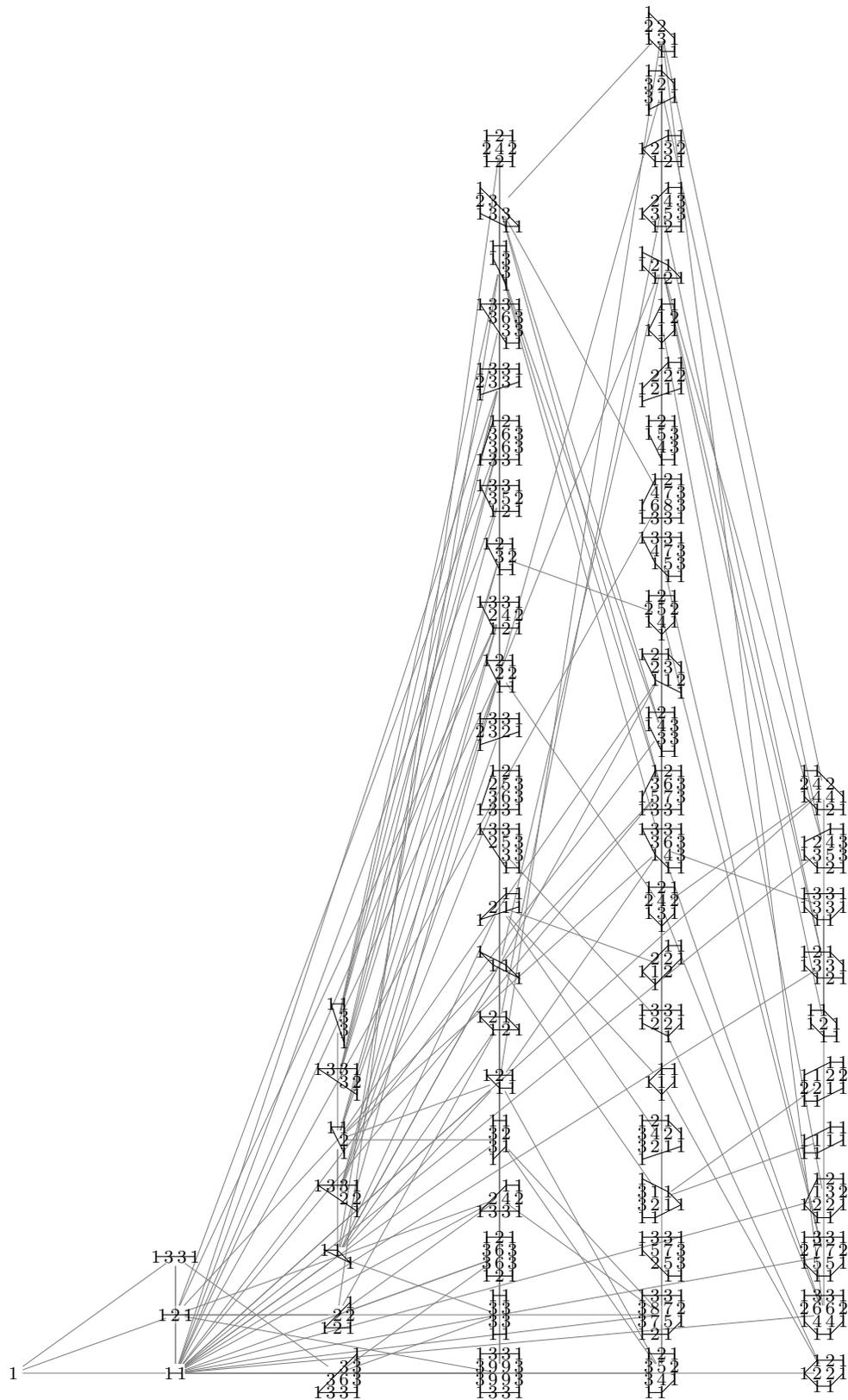
\begin{figure}[h!]
\centering
\begin{tikzpicture}[xscale=2.5,yscale=.9]
\node (0) at (1,0) {}; \node (1) at (2,0) {}; \node (2) at (2,1) {}; \node (3) at (2,2) {}; \node (4) at (3,0) {}; \node (5) at (4,0) {}; \node (6) at (4,1) {}; \node (7) at (4,2) {}; \node (8) at (3,1) {}; \node (9) at (4,3) {}; \node (10) at (4,4) {}; \node (11) at (5,0) {}; \node (12) at (5,1) {}; \node (13) at (5,2) {}; \node (14) at (4,5) {}; \node (15) at (4,6) {}; \node (16) at (3,2) {}; \node (17) at (4,7) {}; \node (18) at (5,3) {}; \node (19) at (5,4) {}; \node (20) at (4,8) {}; \node (21) at (5,5) {}; \node (22) at (6,0) {}; \node (23) at (5,6) {}; \node (24) at (4,9) {}; \node (25) at (4,10) {}; \node (26) at (5,7) {}; \node (27) at (4,11) {}; \node (28) at (3,3) {}; \node (29) at (3,4) {}; \node (30) at (4,12) {}; \node (31) at (4,13) {}; \node (32) at (5,8) {}; \node (33) at (6,1) {}; \node (34) at (5,9) {}; \node (35) at (5,10) {}; \node (36) at (5,11) {}; \node (37) at (5,12) {}; \node (38) at (4,14) {}; \node (39) at (4,15) {}; \node (40) at (3,5) {}; \node (41) at (4,16) {}; \node (42) at (5,13) {}; \node (43) at (6,2) {}; \node (44) at (5,14) {}; \node (45) at (5,15) {}; \node (46) at (5,16) {}; \node (47) at (3,6) {}; \node (48) at (4,17) {}; \node (49) at (4,18) {}; \node (50) at (4,19) {}; \node (51) at (5,17) {}; \node (52) at (5,18) {}; \node (53) at (6,3) {}; \node (54) at (6,4) {}; \node (55) at (6,5) {}; \node (56) at (5,19) {}; \node (57) at (6,6) {}; \node (58) at (6,7) {}; \node (59) at (6,8) {}; \node (60) at (5,20) {}; \node (61) at (5,21) {}; \node (62) at (6,9) {}; \node (63) at (5,22) {}; \node (64) at (5,23) {}; \node (65) at (6,10) {}; \node (66) at (4,20) {}; \node (67) at (4,21) {};
\draw[gray] (0) -- (1); \draw[gray] (0) -- (2); \draw[gray] (0) -- (3); \draw[gray] (1) -- (2); \draw[gray] (1) -- (3); \draw[gray] (1) -- (5); \draw[gray] (1) -- (7); \draw[gray] (1) -- (9); \draw[gray] (1) -- (10); \draw[gray] (1) -- (11); \draw[gray] (1) -- (12); \draw[gray] (1) -- (14); \draw[gray] (1) -- (16); \draw[gray] (1) -- (22); \draw[gray] (1) -- (30); \draw[gray] (1) -- (31); \draw[gray] (1) -- (33); \draw[gray] (1) -- (35); \draw[gray] (1) -- (36); \draw[gray] (1) -- (39); \draw[gray] (1) -- (43); \draw[gray] (1) -- (45); \draw[gray] (1) -- (48); \draw[gray] (1) -- (53); \draw[gray] (1) -- (58); \draw[gray] (1) -- (62); \draw[gray] (1) -- (65); \draw[gray] (2) -- (3); \draw[gray] (2) -- (5); \draw[gray] (2) -- (8); \draw[gray] (2) -- (9); \draw[gray] (2) -- (12); \draw[gray] (2) -- (31); \draw[gray] (2) -- (35); \draw[gray] (2) -- (41); \draw[gray] (3) -- (4); \draw[gray] (4) -- (5); \draw[gray] (4) -- (6); \draw[gray] (4) -- (7); \draw[gray] (5) -- (6); \draw[gray] (5) -- (7); \draw[gray] (6) -- (7); \draw[gray] (6) -- (10); \draw[gray] (6) -- (14); \draw[gray] (6) -- (16); \draw[gray] (7) -- (8); \draw[gray] (7) -- (9); \draw[gray] (7) -- (67); \draw[gray] (8) -- (9); \draw[gray] (8) -- (15); \draw[gray] (8) -- (67); \draw[gray] (9) -- (10); \draw[gray] (9) -- (12); \draw[gray] (9) -- (15); \draw[gray] (9) -- (67); \draw[gray] (10) -- (11); \draw[gray] (10) -- (12); \draw[gray] (10) -- (14); \draw[gray] (10) -- (15); \draw[gray] (10) -- (17); \draw[gray] (10) -- (29); \draw[gray] (10) -- (50); \draw[gray] (11) -- (12); \draw[gray] (11) -- (14); \draw[gray] (11) -- (64); \draw[gray] (12) -- (13); \draw[gray] (14) -- (15); \draw[gray] (14) -- (16); \draw[gray] (14) -- (17); \draw[gray] (14) -- (29); \draw[gray] (14) -- (34); \draw[gray] (14) -- (50); \draw[gray] (14) -- (64); \draw[gray] (14) -- (66); \draw[gray] (15) -- (16); \draw[gray] (15) -- (60); \draw[gray] (16) -- (17); \draw[gray] (16) -- (20); \draw[gray] (16) -- (29); \draw[gray] (16) -- (37); \draw[gray] (16) -- (50); \draw[gray] (16) -- (63); \draw[gray] (17) -- (18); \draw[gray] (17) -- (56); \draw[gray] (18) -- (19); \draw[gray] (18) -- (54); \draw[gray] (18) -- (55); \draw[gray] (19) -- (20); \draw[gray] (19) -- (51); \draw[gray] (20) -- (21); \draw[gray] (20) -- (26); \draw[gray] (20) -- (27); \draw[gray] (21) -- (22); \draw[gray] (21) -- (23); \draw[gray] (21) -- (63); \draw[gray] (22) -- (23); \draw[gray] (23) -- (24); \draw[gray] (24) -- (25); \draw[gray] (24) -- (27); \draw[gray] (25) -- (28); \draw[gray] (26) -- (63); \draw[gray] (27) -- (28); \draw[gray] (28) -- (29); \draw[gray] (28) -- (30); \draw[gray] (28) -- (31); \draw[gray] (29) -- (30); \draw[gray] (29) -- (31); \draw[gray] (29) -- (34); \draw[gray] (29) -- (37); \draw[gray] (29) -- (38); \draw[gray] (29) -- (47); \draw[gray] (30) -- (31); \draw[gray] (30) -- (32); \draw[gray] (30) -- (56); \draw[gray] (32) -- (33); \draw[gray] (32) -- (36); \draw[gray] (32) -- (65); \draw[gray] (33) -- (34); \draw[gray] (33) -- (64); \draw[gray] (33) -- (65); \draw[gray] (34) -- (35); \draw[gray] (34) -- (50); \draw[gray] (34) -- (59); \draw[gray] (35) -- (36); \draw[gray] (35) -- (50); \draw[gray] (35) -- (66); \draw[gray] (36) -- (50); \draw[gray] (36) -- (66); \draw[gray] (38) -- (39); \draw[gray] (38) -- (40); \draw[gray] (38) -- (42); \draw[gray] (39) -- (40); \draw[gray] (39) -- (66); \draw[gray] (40) -- (41); \draw[gray] (40) -- (48); \draw[gray] (41) -- (47); \draw[gray] (41) -- (49); \draw[gray] (42) -- (43); \draw[gray] (42) -- (46); \draw[gray] (43) -- (44); \draw[gray] (44) -- (45); \draw[gray] (45) -- (46); \draw[gray] (45) -- (66); \draw[gray] (47) -- (48); \draw[gray] (47) -- (49); \draw[gray] (47) -- (50); \draw[gray] (48) -- (49); \draw[gray] (50) -- (66); \draw[gray] (51) -- (52); \draw[gray] (52) -- (53); \draw[gray] (56) -- (57); \draw[gray] (56) -- (58); \draw[gray] (56) -- (61); \draw[gray] (56) -- (62); \draw[gray] (57) -- (58); \draw[gray] (57) -- (59); \draw[gray] (58) -- (59); \draw[gray] (58) -- (60); \draw[gray] (62) -- (63); \draw[gray] (64) -- (65); \draw[gray] (64) -- (66);
\draw (0) node {\begin{tikzpicture}[scale=1/5] \fill (0,0) circle (1pt); \fill (0, 0) node{\tiny$1$} circle (.5pt);\end{tikzpicture}};
\draw (1) node {\begin{tikzpicture}[scale=1/5] \draw (1, 0) -- (0, 0); \fill (1, 0) node{\tiny$1$} circle (.5pt);\fill (0, 0) node{\tiny$1$} circle (.5pt);\end{tikzpicture}};
\draw (2) node {\begin{tikzpicture}[scale=1/5] \draw (0, 0) -- (2, 0); \fill (0, 0) node{\tiny$1$} circle (.5pt);\fill (2, 0) node{\tiny$1$} circle (.5pt);\fill (1, 0) node{\tiny$2$} circle (.5pt);\end{tikzpicture}};
\draw (3) node {\begin{tikzpicture}[scale=1/5] \draw (0, 0) -- (3, 0); \fill (0, 0) node{\tiny$1$} circle (.5pt);\fill (3, 0) node{\tiny$1$} circle (.5pt);\fill (1, 0) node{\tiny$3$} circle (.5pt);\fill (2, 0) node{\tiny$3$} circle (.5pt);\end{tikzpicture}};
\draw (4) node {\begin{tikzpicture}[scale=1/5] \draw (3, 0) -- (0, 0) -- (3, 3) -- cycle; \fill (0, 0) node{\tiny$1$} circle (.5pt);\fill (3, 3) node{\tiny$1$} circle (.5pt);\fill (3, 0) node{\tiny$1$} circle (.5pt);\fill (1, 1) node{\tiny$3$} circle (.5pt);\fill (2, 2) node{\tiny$3$} circle (.5pt);\fill (1, 0) node{\tiny$3$} circle (.5pt);\fill (2, 1) node{\tiny$6$} circle (.5pt);\fill (3, 2) node{\tiny$3$} circle (.5pt);\fill (2, 0) node{\tiny$3$} circle (.5pt);\fill (3, 1) node{\tiny$3$} circle (.5pt);\end{tikzpicture}};
\draw (5) node {\begin{tikzpicture}[scale=1/5] \draw (0, 3) -- (0, 0) -- (3, 0) -- (3, 3) -- cycle; \fill (0, 0) node{\tiny$1$} circle (.5pt);\fill (3, 3) node{\tiny$1$} circle (.5pt);\fill (3, 0) node{\tiny$1$} circle (.5pt);\fill (0, 3) node{\tiny$1$} circle (.5pt);\fill (0, 1) node{\tiny$3$} circle (.5pt);\fill (0, 2) node{\tiny$3$} circle (.5pt);\fill (1, 0) node{\tiny$3$} circle (.5pt);\fill (1, 1) node{\tiny$9$} circle (.5pt);\fill (1, 2) node{\tiny$9$} circle (.5pt);\fill (1, 3) node{\tiny$3$} circle (.5pt);\fill (2, 0) node{\tiny$3$} circle (.5pt);\fill (2, 1) node{\tiny$9$} circle (.5pt);\fill (2, 2) node{\tiny$9$} circle (.5pt);\fill (2, 3) node{\tiny$3$} circle (.5pt);\fill (3, 1) node{\tiny$3$} circle (.5pt);\fill (3, 2) node{\tiny$3$} circle (.5pt);\end{tikzpicture}};
\draw (6) node {\begin{tikzpicture}[scale=1/5] \draw (2, 3) -- (2, 0) -- (3, 0) -- (3, 3) -- cycle; \fill (2, 0) node{\tiny$1$} circle (.5pt);\fill (3, 3) node{\tiny$1$} circle (.5pt);\fill (3, 0) node{\tiny$1$} circle (.5pt);\fill (2, 3) node{\tiny$1$} circle (.5pt);\fill (2, 1) node{\tiny$3$} circle (.5pt);\fill (2, 2) node{\tiny$3$} circle (.5pt);\fill (3, 1) node{\tiny$3$} circle (.5pt);\fill (3, 2) node{\tiny$3$} circle (.5pt);\end{tikzpicture}};
\draw (7) node {\begin{tikzpicture}[scale=1/5] \draw (2, 3) -- (2, 0) -- (4, 0) -- (4, 3) -- cycle; \fill (2, 0) node{\tiny$1$} circle (.5pt);\fill (4, 3) node{\tiny$1$} circle (.5pt);\fill (4, 0) node{\tiny$1$} circle (.5pt);\fill (2, 3) node{\tiny$1$} circle (.5pt);\fill (2, 1) node{\tiny$3$} circle (.5pt);\fill (2, 2) node{\tiny$3$} circle (.5pt);\fill (3, 0) node{\tiny$2$} circle (.5pt);\fill (3, 1) node{\tiny$6$} circle (.5pt);\fill (3, 2) node{\tiny$6$} circle (.5pt);\fill (3, 3) node{\tiny$2$} circle (.5pt);\fill (4, 1) node{\tiny$3$} circle (.5pt);\fill (4, 2) node{\tiny$3$} circle (.5pt);\end{tikzpicture}};
\draw (8) node {\begin{tikzpicture}[scale=1/5] \draw (4, 0) -- (2, 0) -- (4, 2) -- cycle; \fill (2, 0) node{\tiny$1$} circle (.5pt);\fill (4, 2) node{\tiny$1$} circle (.5pt);\fill (4, 0) node{\tiny$1$} circle (.5pt);\fill (3, 1) node{\tiny$2$} circle (.5pt);\fill (3, 0) node{\tiny$2$} circle (.5pt);\fill (4, 1) node{\tiny$2$} circle (.5pt);\end{tikzpicture}};
\draw (9) node {\begin{tikzpicture}[scale=1/5] \draw (4, 2) -- (2, 0) -- (5, 0) -- (5, 2) -- cycle; \fill (2, 0) node{\tiny$1$} circle (.5pt);\fill (5, 2) node{\tiny$1$} circle (.5pt);\fill (5, 0) node{\tiny$1$} circle (.5pt);\fill (4, 2) node{\tiny$1$} circle (.5pt);\fill (3, 0) node{\tiny$3$} circle (.5pt);\fill (4, 0) node{\tiny$3$} circle (.5pt);\fill (3, 1) node{\tiny$2$} circle (.5pt);\fill (4, 1) node{\tiny$4$} circle (.5pt);\fill (5, 1) node{\tiny$2$} circle (.5pt);\end{tikzpicture}};
\draw (10) node {\begin{tikzpicture}[scale=1/5] \draw (4, 2) -- (4, -1) -- (5, 0) -- (5, 2) -- cycle; \fill (4, -1) node{\tiny$1$} circle (.5pt);\fill (5, 2) node{\tiny$1$} circle (.5pt);\fill (5, 0) node{\tiny$1$} circle (.5pt);\fill (4, 2) node{\tiny$1$} circle (.5pt);\fill (4, 0) node{\tiny$3$} circle (.5pt);\fill (4, 1) node{\tiny$3$} circle (.5pt);\fill (5, 1) node{\tiny$2$} circle (.5pt);\end{tikzpicture}};
\draw (11) node {\begin{tikzpicture}[scale=1/5] \draw (4, 2) -- (4, -1) -- (5, -1) -- (6, 0) -- (6, 2) -- cycle; \fill (4, -1) node{\tiny$1$} circle (.5pt);\fill (6, 2) node{\tiny$1$} circle (.5pt);\fill (6, 0) node{\tiny$1$} circle (.5pt);\fill (4, 2) node{\tiny$1$} circle (.5pt);\fill (5, -1) node{\tiny$1$} circle (.5pt);\fill (4, 0) node{\tiny$3$} circle (.5pt);\fill (4, 1) node{\tiny$3$} circle (.5pt);\fill (5, 0) node{\tiny$4$} circle (.5pt);\fill (5, 1) node{\tiny$5$} circle (.5pt);\fill (5, 2) node{\tiny$2$} circle (.5pt);\fill (6, 1) node{\tiny$2$} circle (.5pt);\end{tikzpicture}};
\draw (12) node {\begin{tikzpicture}[scale=1/5] \draw (4, 2) -- (4, -1) -- (6, -1) -- (7, 0) -- (7, 2) -- cycle; \fill (4, -1) node{\tiny$1$} circle (.5pt);\fill (7, 2) node{\tiny$1$} circle (.5pt);\fill (7, 0) node{\tiny$1$} circle (.5pt);\fill (4, 2) node{\tiny$1$} circle (.5pt);\fill (6, -1) node{\tiny$1$} circle (.5pt);\fill (4, 0) node{\tiny$3$} circle (.5pt);\fill (4, 1) node{\tiny$3$} circle (.5pt);\fill (5, -1) node{\tiny$2$} circle (.5pt);\fill (5, 0) node{\tiny$7$} circle (.5pt);\fill (5, 1) node{\tiny$8$} circle (.5pt);\fill (5, 2) node{\tiny$3$} circle (.5pt);\fill (6, 0) node{\tiny$5$} circle (.5pt);\fill (6, 1) node{\tiny$7$} circle (.5pt);\fill (6, 2) node{\tiny$3$} circle (.5pt);\fill (7, 1) node{\tiny$2$} circle (.5pt);\end{tikzpicture}};
\draw (13) node {\begin{tikzpicture}[scale=1/5] \draw (4, 2) -- (4, 1) -- (6, -1) -- (7, -1) -- (7, 2) -- cycle; \fill (4, 1) node{\tiny$1$} circle (.5pt);\fill (4, 2) node{\tiny$1$} circle (.5pt);\fill (7, 2) node{\tiny$1$} circle (.5pt);\fill (7, -1) node{\tiny$1$} circle (.5pt);\fill (6, -1) node{\tiny$1$} circle (.5pt);\fill (7, 0) node{\tiny$3$} circle (.5pt);\fill (6, 0) node{\tiny$5$} circle (.5pt);\fill (5, 0) node{\tiny$2$} circle (.5pt);\fill (7, 1) node{\tiny$3$} circle (.5pt);\fill (6, 1) node{\tiny$7$} circle (.5pt);\fill (5, 1) node{\tiny$5$} circle (.5pt);\fill (6, 2) node{\tiny$3$} circle (.5pt);\fill (5, 2) node{\tiny$3$} circle (.5pt);\end{tikzpicture}};
\draw (14) node {\begin{tikzpicture}[scale=1/5] \draw (5, 1) -- (4, 2) -- (6, 2) -- (6, 1) -- cycle; \fill (4, 2) node{\tiny$1$} circle (.5pt);\fill (5, 1) node{\tiny$1$} circle (.5pt);\fill (6, 2) node{\tiny$1$} circle (.5pt);\fill (6, 1) node{\tiny$1$} circle (.5pt);\fill (5, 2) node{\tiny$2$} circle (.5pt);\end{tikzpicture}};
\draw (15) node {\begin{tikzpicture}[scale=1/5] \draw (5, 1) -- (4, 2) -- (6, 2) -- (7, 1) -- cycle; \fill (4, 2) node{\tiny$1$} circle (.5pt);\fill (5, 1) node{\tiny$1$} circle (.5pt);\fill (7, 1) node{\tiny$1$} circle (.5pt);\fill (6, 2) node{\tiny$1$} circle (.5pt);\fill (6, 1) node{\tiny$2$} circle (.5pt);\fill (5, 2) node{\tiny$2$} circle (.5pt);\end{tikzpicture}};
\draw (16) node {\begin{tikzpicture}[scale=1/5] \draw (5, 2) -- (7, 1) -- (6, 2) -- cycle; \fill (7, 1) node{\tiny$1$} circle (.5pt);\fill (6, 2) node{\tiny$1$} circle (.5pt);\fill (5, 2) node{\tiny$1$} circle (.5pt);\end{tikzpicture}};
\draw (17) node {\begin{tikzpicture}[scale=1/5] \draw (5, 2) -- (7, 1) -- (6, 2) -- (4, 3) -- cycle; \fill (7, 1) node{\tiny$1$} circle (.5pt);\fill (6, 2) node{\tiny$1$} circle (.5pt);\fill (5, 2) node{\tiny$1$} circle (.5pt);\fill (4, 3) node{\tiny$1$} circle (.5pt);\end{tikzpicture}};
\draw (18) node {\begin{tikzpicture}[scale=1/5] \draw (4, 3) -- (4, 0) -- (5, 0) -- (7, 1) -- (6, 2) -- cycle; \fill (4, 0) node{\tiny$1$} circle (.5pt);\fill (7, 1) node{\tiny$1$} circle (.5pt);\fill (6, 2) node{\tiny$1$} circle (.5pt);\fill (4, 3) node{\tiny$1$} circle (.5pt);\fill (5, 0) node{\tiny$1$} circle (.5pt);\fill (4, 1) node{\tiny$3$} circle (.5pt);\fill (4, 2) node{\tiny$3$} circle (.5pt);\fill (5, 1) node{\tiny$2$} circle (.5pt);\fill (5, 2) node{\tiny$1$} circle (.5pt);\fill (6, 1) node{\tiny$1$} circle (.5pt);\end{tikzpicture}};
\draw (19) node {\begin{tikzpicture}[scale=1/5] \draw (4, 3) -- (4, 0) -- (7, 1) -- (7, 2) -- (6, 3) -- cycle; \fill (4, 0) node{\tiny$1$} circle (.5pt);\fill (7, 2) node{\tiny$1$} circle (.5pt);\fill (7, 1) node{\tiny$1$} circle (.5pt);\fill (4, 3) node{\tiny$1$} circle (.5pt);\fill (6, 3) node{\tiny$1$} circle (.5pt);\fill (4, 2) node{\tiny$3$} circle (.5pt);\fill (4, 1) node{\tiny$3$} circle (.5pt);\fill (5, 3) node{\tiny$2$} circle (.5pt);\fill (5, 2) node{\tiny$4$} circle (.5pt);\fill (5, 1) node{\tiny$2$} circle (.5pt);\fill (6, 2) node{\tiny$2$} circle (.5pt);\fill (6, 1) node{\tiny$1$} circle (.5pt);\end{tikzpicture}};
\draw (20) node {\begin{tikzpicture}[scale=1/5] \draw (6, 2) -- (4, 0) -- (7, 1) -- (7, 2) -- cycle; \fill (4, 0) node{\tiny$1$} circle (.5pt);\fill (7, 2) node{\tiny$1$} circle (.5pt);\fill (7, 1) node{\tiny$1$} circle (.5pt);\fill (6, 2) node{\tiny$1$} circle (.5pt);\fill (5, 1) node{\tiny$2$} circle (.5pt);\fill (6, 1) node{\tiny$1$} circle (.5pt);\end{tikzpicture}};
\draw (21) node {\begin{tikzpicture}[scale=1/5] \draw (6, 0) -- (5, 1) -- (6, 2) -- (7, 2) -- (7, 1) -- cycle; \fill (5, 1) node{\tiny$1$} circle (.5pt);\fill (6, 0) node{\tiny$1$} circle (.5pt);\fill (7, 2) node{\tiny$1$} circle (.5pt);\fill (6, 2) node{\tiny$1$} circle (.5pt);\fill (7, 1) node{\tiny$1$} circle (.5pt);\fill (6, 1) node{\tiny$1$} circle (.5pt);\end{tikzpicture}};
\draw (22) node {\begin{tikzpicture}[scale=1/5] \draw (6, 0) -- (5, 1) -- (6, 2) -- (8, 2) -- (8, 1) -- (7, 0) -- cycle; \fill (5, 1) node{\tiny$1$} circle (.5pt);\fill (6, 0) node{\tiny$1$} circle (.5pt);\fill (8, 2) node{\tiny$1$} circle (.5pt);\fill (6, 2) node{\tiny$1$} circle (.5pt);\fill (7, 0) node{\tiny$1$} circle (.5pt);\fill (8, 1) node{\tiny$1$} circle (.5pt);\fill (7, 1) node{\tiny$2$} circle (.5pt);\fill (6, 1) node{\tiny$2$} circle (.5pt);\fill (7, 2) node{\tiny$2$} circle (.5pt);\end{tikzpicture}};
\draw (23) node {\begin{tikzpicture}[scale=1/5] \draw (4, 2) -- (4, 1) -- (6, 0) -- (7, 1) -- (7, 2) -- cycle; \fill (4, 1) node{\tiny$1$} circle (.5pt);\fill (4, 2) node{\tiny$1$} circle (.5pt);\fill (7, 2) node{\tiny$1$} circle (.5pt);\fill (7, 1) node{\tiny$1$} circle (.5pt);\fill (6, 0) node{\tiny$1$} circle (.5pt);\fill (6, 2) node{\tiny$3$} circle (.5pt);\fill (5, 2) node{\tiny$3$} circle (.5pt);\fill (6, 1) node{\tiny$2$} circle (.5pt);\fill (5, 1) node{\tiny$2$} circle (.5pt);\end{tikzpicture}};
\draw (24) node {\begin{tikzpicture}[scale=1/5] \draw (6, -1) -- (4, 2) -- (7, 2) -- (7, -1) -- cycle; \fill (4, 2) node{\tiny$1$} circle (.5pt);\fill (7, 2) node{\tiny$1$} circle (.5pt);\fill (7, -1) node{\tiny$1$} circle (.5pt);\fill (6, -1) node{\tiny$1$} circle (.5pt);\fill (7, 0) node{\tiny$3$} circle (.5pt);\fill (6, 0) node{\tiny$3$} circle (.5pt);\fill (7, 1) node{\tiny$3$} circle (.5pt);\fill (6, 1) node{\tiny$5$} circle (.5pt);\fill (5, 1) node{\tiny$2$} circle (.5pt);\fill (6, 2) node{\tiny$3$} circle (.5pt);\fill (5, 2) node{\tiny$3$} circle (.5pt);\end{tikzpicture}};
\draw (25) node {\begin{tikzpicture}[scale=1/5] \draw (5, 2) -- (4, -1) -- (7, -1) -- (7, 2) -- cycle; \fill (4, -1) node{\tiny$1$} circle (.5pt);\fill (7, 2) node{\tiny$1$} circle (.5pt);\fill (7, -1) node{\tiny$1$} circle (.5pt);\fill (5, 2) node{\tiny$1$} circle (.5pt);\fill (6, -1) node{\tiny$3$} circle (.5pt);\fill (5, -1) node{\tiny$3$} circle (.5pt);\fill (7, 0) node{\tiny$3$} circle (.5pt);\fill (6, 0) node{\tiny$6$} circle (.5pt);\fill (5, 0) node{\tiny$3$} circle (.5pt);\fill (7, 1) node{\tiny$3$} circle (.5pt);\fill (6, 1) node{\tiny$5$} circle (.5pt);\fill (5, 1) node{\tiny$2$} circle (.5pt);\fill (6, 2) node{\tiny$2$} circle (.5pt);\end{tikzpicture}};
\draw (26) node {\begin{tikzpicture}[scale=1/5] \draw (5, -1) -- (4, 0) -- (6, 2) -- (7, 2) -- (7, 1) -- cycle; \fill (4, 0) node{\tiny$1$} circle (.5pt);\fill (5, -1) node{\tiny$1$} circle (.5pt);\fill (7, 2) node{\tiny$1$} circle (.5pt);\fill (7, 1) node{\tiny$1$} circle (.5pt);\fill (6, 2) node{\tiny$1$} circle (.5pt);\fill (6, 0) node{\tiny$2$} circle (.5pt);\fill (6, 1) node{\tiny$2$} circle (.5pt);\fill (5, 0) node{\tiny$1$} circle (.5pt);\fill (5, 1) node{\tiny$2$} circle (.5pt);\end{tikzpicture}};
\draw (27) node {\begin{tikzpicture}[scale=1/5] \draw (4, 2) -- (4, 0) -- (7, 1) -- (7, 2) -- cycle; \fill (4, 0) node{\tiny$1$} circle (.5pt);\fill (7, 2) node{\tiny$1$} circle (.5pt);\fill (4, 2) node{\tiny$1$} circle (.5pt);\fill (7, 1) node{\tiny$1$} circle (.5pt);\fill (5, 2) node{\tiny$3$} circle (.5pt);\fill (6, 2) node{\tiny$3$} circle (.5pt);\fill (4, 1) node{\tiny$2$} circle (.5pt);\fill (5, 1) node{\tiny$3$} circle (.5pt);\fill (6, 1) node{\tiny$2$} circle (.5pt);\end{tikzpicture}};
\draw (28) node {\begin{tikzpicture}[scale=1/5] \draw (7, 0) -- (4, 2) -- (7, 2) -- cycle; \fill (4, 2) node{\tiny$1$} circle (.5pt);\fill (7, 2) node{\tiny$1$} circle (.5pt);\fill (7, 0) node{\tiny$1$} circle (.5pt);\fill (6, 2) node{\tiny$3$} circle (.5pt);\fill (5, 2) node{\tiny$3$} circle (.5pt);\fill (7, 1) node{\tiny$2$} circle (.5pt);\fill (6, 1) node{\tiny$2$} circle (.5pt);\end{tikzpicture}};
\draw (29) node {\begin{tikzpicture}[scale=1/5] \draw (7, 0) -- (6, 2) -- (7, 2) -- cycle; \fill (6, 2) node{\tiny$1$} circle (.5pt);\fill (7, 0) node{\tiny$1$} circle (.5pt);\fill (7, 2) node{\tiny$1$} circle (.5pt);\fill (7, 1) node{\tiny$2$} circle (.5pt);\end{tikzpicture}};
\draw (30) node {\begin{tikzpicture}[scale=1/5] \draw (7, 0) -- (6, 2) -- (8, 2) -- (8, 0) -- cycle; \fill (6, 2) node{\tiny$1$} circle (.5pt);\fill (7, 0) node{\tiny$1$} circle (.5pt);\fill (8, 2) node{\tiny$1$} circle (.5pt);\fill (8, 0) node{\tiny$1$} circle (.5pt);\fill (8, 1) node{\tiny$2$} circle (.5pt);\fill (7, 1) node{\tiny$2$} circle (.5pt);\fill (7, 2) node{\tiny$2$} circle (.5pt);\end{tikzpicture}};
\draw (31) node {\begin{tikzpicture}[scale=1/5] \draw (7, 0) -- (6, 2) -- (9, 2) -- (9, 0) -- cycle; \fill (6, 2) node{\tiny$1$} circle (.5pt);\fill (7, 0) node{\tiny$1$} circle (.5pt);\fill (9, 2) node{\tiny$1$} circle (.5pt);\fill (9, 0) node{\tiny$1$} circle (.5pt);\fill (8, 0) node{\tiny$2$} circle (.5pt);\fill (9, 1) node{\tiny$2$} circle (.5pt);\fill (8, 1) node{\tiny$4$} circle (.5pt);\fill (7, 1) node{\tiny$2$} circle (.5pt);\fill (8, 2) node{\tiny$3$} circle (.5pt);\fill (7, 2) node{\tiny$3$} circle (.5pt);\end{tikzpicture}};
\draw (32) node {\begin{tikzpicture}[scale=1/5] \draw (6, 2) -- (6, 0) -- (7, -1) -- (8, 0) -- (8, 2) -- cycle; \fill (6, 0) node{\tiny$1$} circle (.5pt);\fill (8, 2) node{\tiny$1$} circle (.5pt);\fill (6, 2) node{\tiny$1$} circle (.5pt);\fill (7, -1) node{\tiny$1$} circle (.5pt);\fill (8, 0) node{\tiny$1$} circle (.5pt);\fill (8, 1) node{\tiny$2$} circle (.5pt);\fill (7, 2) node{\tiny$2$} circle (.5pt);\fill (7, 1) node{\tiny$4$} circle (.5pt);\fill (7, 0) node{\tiny$3$} circle (.5pt);\fill (6, 1) node{\tiny$2$} circle (.5pt);\end{tikzpicture}};
\draw (33) node {\begin{tikzpicture}[scale=1/5] \draw (6, 2) -- (6, 0) -- (7, -1) -- (8, -1) -- (9, 0) -- (9, 2) -- cycle; \fill (6, 0) node{\tiny$1$} circle (.5pt);\fill (9, 2) node{\tiny$1$} circle (.5pt);\fill (6, 2) node{\tiny$1$} circle (.5pt);\fill (7, -1) node{\tiny$1$} circle (.5pt);\fill (9, 0) node{\tiny$1$} circle (.5pt);\fill (8, -1) node{\tiny$1$} circle (.5pt);\fill (8, 0) node{\tiny$4$} circle (.5pt);\fill (7, 0) node{\tiny$4$} circle (.5pt);\fill (9, 1) node{\tiny$2$} circle (.5pt);\fill (8, 1) node{\tiny$6$} circle (.5pt);\fill (7, 1) node{\tiny$6$} circle (.5pt);\fill (6, 1) node{\tiny$2$} circle (.5pt);\fill (8, 2) node{\tiny$3$} circle (.5pt);\fill (7, 2) node{\tiny$3$} circle (.5pt);\end{tikzpicture}};
\draw (34) node {\begin{tikzpicture}[scale=1/5] \draw (7, 0) -- (6, 2) -- (9, 2) -- (9, -1) -- (8, -1) -- cycle; \fill (6, 2) node{\tiny$1$} circle (.5pt);\fill (7, 0) node{\tiny$1$} circle (.5pt);\fill (9, 2) node{\tiny$1$} circle (.5pt);\fill (9, -1) node{\tiny$1$} circle (.5pt);\fill (8, -1) node{\tiny$1$} circle (.5pt);\fill (9, 0) node{\tiny$3$} circle (.5pt);\fill (8, 0) node{\tiny$4$} circle (.5pt);\fill (9, 1) node{\tiny$3$} circle (.5pt);\fill (8, 1) node{\tiny$6$} circle (.5pt);\fill (7, 1) node{\tiny$3$} circle (.5pt);\fill (8, 2) node{\tiny$3$} circle (.5pt);\fill (7, 2) node{\tiny$3$} circle (.5pt);\end{tikzpicture}};
\draw (35) node {\begin{tikzpicture}[scale=1/5] \draw (6, 0) -- (6, -1) -- (9, -1) -- (9, 2) -- (7, 2) -- cycle; \fill (6, -1) node{\tiny$1$} circle (.5pt);\fill (6, 0) node{\tiny$1$} circle (.5pt);\fill (9, 2) node{\tiny$1$} circle (.5pt);\fill (9, -1) node{\tiny$1$} circle (.5pt);\fill (7, 2) node{\tiny$1$} circle (.5pt);\fill (7, -1) node{\tiny$3$} circle (.5pt);\fill (8, -1) node{\tiny$3$} circle (.5pt);\fill (7, 0) node{\tiny$5$} circle (.5pt);\fill (8, 0) node{\tiny$7$} circle (.5pt);\fill (9, 0) node{\tiny$3$} circle (.5pt);\fill (7, 1) node{\tiny$3$} circle (.5pt);\fill (8, 1) node{\tiny$6$} circle (.5pt);\fill (9, 1) node{\tiny$3$} circle (.5pt);\fill (8, 2) node{\tiny$2$} circle (.5pt);\end{tikzpicture}};
\draw (36) node {\begin{tikzpicture}[scale=1/5] \draw (6, 2) -- (6, 1) -- (7, -1) -- (8, -1) -- (8, 2) -- cycle; \fill (6, 1) node{\tiny$1$} circle (.5pt);\fill (6, 2) node{\tiny$1$} circle (.5pt);\fill (7, -1) node{\tiny$1$} circle (.5pt);\fill (8, 2) node{\tiny$1$} circle (.5pt);\fill (8, -1) node{\tiny$1$} circle (.5pt);\fill (8, 0) node{\tiny$3$} circle (.5pt);\fill (8, 1) node{\tiny$3$} circle (.5pt);\fill (7, 0) node{\tiny$3$} circle (.5pt);\fill (7, 1) node{\tiny$4$} circle (.5pt);\fill (7, 2) node{\tiny$2$} circle (.5pt);\end{tikzpicture}};
\draw (37) node {\begin{tikzpicture}[scale=1/5] \draw (12, 2) -- (11, 4) -- (13, 4) -- (14, 3) -- (14, 1) -- cycle; \fill (11, 4) node{\tiny$1$} circle (.5pt);\fill (12, 2) node{\tiny$1$} circle (.5pt);\fill (14, 3) node{\tiny$1$} circle (.5pt);\fill (14, 1) node{\tiny$1$} circle (.5pt);\fill (13, 4) node{\tiny$1$} circle (.5pt);\fill (12, 4) node{\tiny$2$} circle (.5pt);\fill (13, 3) node{\tiny$3$} circle (.5pt);\fill (14, 2) node{\tiny$2$} circle (.5pt);\fill (12, 3) node{\tiny$2$} circle (.5pt);\fill (13, 2) node{\tiny$1$} circle (.5pt);\end{tikzpicture}};
\draw (38) node {\begin{tikzpicture}[scale=1/5] \draw (7, 0) -- (6, 2) -- (8, 2) -- (8, 0) -- cycle; \fill (6, 2) node{\tiny$1$} circle (.5pt);\fill (7, 0) node{\tiny$1$} circle (.5pt);\fill (8, 2) node{\tiny$1$} circle (.5pt);\fill (8, 0) node{\tiny$1$} circle (.5pt);\fill (8, 1) node{\tiny$2$} circle (.5pt);\fill (7, 1) node{\tiny$3$} circle (.5pt);\fill (7, 2) node{\tiny$2$} circle (.5pt);\end{tikzpicture}};
\draw (39) node {\begin{tikzpicture}[scale=1/5] \draw (7, 0) -- (6, 2) -- (9, 2) -- (9, 0) -- cycle; \fill (6, 2) node{\tiny$1$} circle (.5pt);\fill (7, 0) node{\tiny$1$} circle (.5pt);\fill (9, 2) node{\tiny$1$} circle (.5pt);\fill (9, 0) node{\tiny$1$} circle (.5pt);\fill (8, 0) node{\tiny$2$} circle (.5pt);\fill (9, 1) node{\tiny$2$} circle (.5pt);\fill (8, 1) node{\tiny$5$} circle (.5pt);\fill (7, 1) node{\tiny$3$} circle (.5pt);\fill (8, 2) node{\tiny$3$} circle (.5pt);\fill (7, 2) node{\tiny$3$} circle (.5pt);\end{tikzpicture}};
\draw (40) node {\begin{tikzpicture}[scale=1/5] \draw (7, 0) -- (4, 2) -- (7, 2) -- cycle; \fill (4, 2) node{\tiny$1$} circle (.5pt);\fill (7, 2) node{\tiny$1$} circle (.5pt);\fill (7, 0) node{\tiny$1$} circle (.5pt);\fill (6, 2) node{\tiny$3$} circle (.5pt);\fill (5, 2) node{\tiny$3$} circle (.5pt);\fill (7, 1) node{\tiny$2$} circle (.5pt);\fill (6, 1) node{\tiny$3$} circle (.5pt);\end{tikzpicture}};
\draw (41) node {\begin{tikzpicture}[scale=1/5] \draw (5, 2) -- (4, -1) -- (7, -1) -- (7, 2) -- cycle; \fill (4, -1) node{\tiny$1$} circle (.5pt);\fill (7, 2) node{\tiny$1$} circle (.5pt);\fill (7, -1) node{\tiny$1$} circle (.5pt);\fill (5, 2) node{\tiny$1$} circle (.5pt);\fill (6, -1) node{\tiny$3$} circle (.5pt);\fill (5, -1) node{\tiny$3$} circle (.5pt);\fill (7, 0) node{\tiny$3$} circle (.5pt);\fill (6, 0) node{\tiny$6$} circle (.5pt);\fill (5, 0) node{\tiny$3$} circle (.5pt);\fill (7, 1) node{\tiny$3$} circle (.5pt);\fill (6, 1) node{\tiny$6$} circle (.5pt);\fill (5, 1) node{\tiny$3$} circle (.5pt);\fill (6, 2) node{\tiny$2$} circle (.5pt);\end{tikzpicture}};
\draw (42) node {\begin{tikzpicture}[scale=1/5] \draw (6, 2) -- (6, 0) -- (7, -1) -- (8, 0) -- (8, 2) -- cycle; \fill (6, 0) node{\tiny$1$} circle (.5pt);\fill (8, 2) node{\tiny$1$} circle (.5pt);\fill (6, 2) node{\tiny$1$} circle (.5pt);\fill (7, -1) node{\tiny$1$} circle (.5pt);\fill (8, 0) node{\tiny$1$} circle (.5pt);\fill (8, 1) node{\tiny$2$} circle (.5pt);\fill (7, 2) node{\tiny$2$} circle (.5pt);\fill (7, 1) node{\tiny$5$} circle (.5pt);\fill (7, 0) node{\tiny$4$} circle (.5pt);\fill (6, 1) node{\tiny$2$} circle (.5pt);\end{tikzpicture}};
\draw (43) node {\begin{tikzpicture}[scale=1/5] \draw (6, 2) -- (6, 0) -- (7, -1) -- (8, -1) -- (9, 0) -- (9, 2) -- cycle; \fill (6, 0) node{\tiny$1$} circle (.5pt);\fill (9, 2) node{\tiny$1$} circle (.5pt);\fill (6, 2) node{\tiny$1$} circle (.5pt);\fill (7, -1) node{\tiny$1$} circle (.5pt);\fill (9, 0) node{\tiny$1$} circle (.5pt);\fill (8, -1) node{\tiny$1$} circle (.5pt);\fill (8, 0) node{\tiny$5$} circle (.5pt);\fill (7, 0) node{\tiny$5$} circle (.5pt);\fill (9, 1) node{\tiny$2$} circle (.5pt);\fill (8, 1) node{\tiny$7$} circle (.5pt);\fill (7, 1) node{\tiny$7$} circle (.5pt);\fill (6, 1) node{\tiny$2$} circle (.5pt);\fill (8, 2) node{\tiny$3$} circle (.5pt);\fill (7, 2) node{\tiny$3$} circle (.5pt);\end{tikzpicture}};
\draw (44) node {\begin{tikzpicture}[scale=1/5] \draw (7, 0) -- (6, 2) -- (9, 2) -- (9, -1) -- (8, -1) -- cycle; \fill (6, 2) node{\tiny$1$} circle (.5pt);\fill (7, 0) node{\tiny$1$} circle (.5pt);\fill (9, 2) node{\tiny$1$} circle (.5pt);\fill (9, -1) node{\tiny$1$} circle (.5pt);\fill (8, -1) node{\tiny$1$} circle (.5pt);\fill (9, 0) node{\tiny$3$} circle (.5pt);\fill (8, 0) node{\tiny$5$} circle (.5pt);\fill (9, 1) node{\tiny$3$} circle (.5pt);\fill (8, 1) node{\tiny$7$} circle (.5pt);\fill (7, 1) node{\tiny$4$} circle (.5pt);\fill (8, 2) node{\tiny$3$} circle (.5pt);\fill (7, 2) node{\tiny$3$} circle (.5pt);\end{tikzpicture}};
\draw (45) node {\begin{tikzpicture}[scale=1/5] \draw (6, 0) -- (6, -1) -- (9, -1) -- (9, 2) -- (7, 2) -- cycle; \fill (6, -1) node{\tiny$1$} circle (.5pt);\fill (6, 0) node{\tiny$1$} circle (.5pt);\fill (9, 2) node{\tiny$1$} circle (.5pt);\fill (9, -1) node{\tiny$1$} circle (.5pt);\fill (7, 2) node{\tiny$1$} circle (.5pt);\fill (7, -1) node{\tiny$3$} circle (.5pt);\fill (8, -1) node{\tiny$3$} circle (.5pt);\fill (7, 0) node{\tiny$6$} circle (.5pt);\fill (8, 0) node{\tiny$8$} circle (.5pt);\fill (9, 0) node{\tiny$3$} circle (.5pt);\fill (7, 1) node{\tiny$4$} circle (.5pt);\fill (8, 1) node{\tiny$7$} circle (.5pt);\fill (9, 1) node{\tiny$3$} circle (.5pt);\fill (8, 2) node{\tiny$2$} circle (.5pt);\end{tikzpicture}};
\draw (46) node {\begin{tikzpicture}[scale=1/5] \draw (6, 2) -- (6, 1) -- (7, -1) -- (8, -1) -- (8, 2) -- cycle; \fill (6, 1) node{\tiny$1$} circle (.5pt);\fill (6, 2) node{\tiny$1$} circle (.5pt);\fill (7, -1) node{\tiny$1$} circle (.5pt);\fill (8, 2) node{\tiny$1$} circle (.5pt);\fill (8, -1) node{\tiny$1$} circle (.5pt);\fill (8, 0) node{\tiny$3$} circle (.5pt);\fill (8, 1) node{\tiny$3$} circle (.5pt);\fill (7, 0) node{\tiny$4$} circle (.5pt);\fill (7, 1) node{\tiny$5$} circle (.5pt);\fill (7, 2) node{\tiny$2$} circle (.5pt);\end{tikzpicture}};
\draw (47) node {\begin{tikzpicture}[scale=1/5] \draw (7, -1) -- (6, 2) -- (7, 2) -- cycle; \fill (6, 2) node{\tiny$1$} circle (.5pt);\fill (7, -1) node{\tiny$1$} circle (.5pt);\fill (7, 2) node{\tiny$1$} circle (.5pt);\fill (7, 0) node{\tiny$3$} circle (.5pt);\fill (7, 1) node{\tiny$3$} circle (.5pt);\end{tikzpicture}};
\draw (48) node {\begin{tikzpicture}[scale=1/5] \draw (4, 2) -- (4, 0) -- (7, 1) -- (7, 2) -- cycle; \fill (4, 0) node{\tiny$1$} circle (.5pt);\fill (7, 2) node{\tiny$1$} circle (.5pt);\fill (4, 2) node{\tiny$1$} circle (.5pt);\fill (7, 1) node{\tiny$1$} circle (.5pt);\fill (5, 2) node{\tiny$3$} circle (.5pt);\fill (6, 2) node{\tiny$3$} circle (.5pt);\fill (4, 1) node{\tiny$2$} circle (.5pt);\fill (5, 1) node{\tiny$3$} circle (.5pt);\fill (6, 1) node{\tiny$3$} circle (.5pt);\end{tikzpicture}};
\draw (49) node {\begin{tikzpicture}[scale=1/5] \draw (6, -1) -- (4, 2) -- (7, 2) -- (7, -1) -- cycle; \fill (4, 2) node{\tiny$1$} circle (.5pt);\fill (7, 2) node{\tiny$1$} circle (.5pt);\fill (7, -1) node{\tiny$1$} circle (.5pt);\fill (6, -1) node{\tiny$1$} circle (.5pt);\fill (7, 0) node{\tiny$3$} circle (.5pt);\fill (6, 0) node{\tiny$3$} circle (.5pt);\fill (7, 1) node{\tiny$3$} circle (.5pt);\fill (6, 1) node{\tiny$6$} circle (.5pt);\fill (5, 1) node{\tiny$3$} circle (.5pt);\fill (6, 2) node{\tiny$3$} circle (.5pt);\fill (5, 2) node{\tiny$3$} circle (.5pt);\end{tikzpicture}};
\draw (50) node {\begin{tikzpicture}[scale=1/5] \draw (6, 2) -- (6, 1) -- (7, -1) -- (7, 2) -- cycle; \fill (6, 1) node{\tiny$1$} circle (.5pt);\fill (6, 2) node{\tiny$1$} circle (.5pt);\fill (7, -1) node{\tiny$1$} circle (.5pt);\fill (7, 2) node{\tiny$1$} circle (.5pt);\fill (7, 0) node{\tiny$3$} circle (.5pt);\fill (7, 1) node{\tiny$3$} circle (.5pt);\end{tikzpicture}};
\draw (51) node {\begin{tikzpicture}[scale=1/5] \draw (4, 1) -- (4, 0) -- (7, 1) -- (7, 3) -- (6, 3) -- cycle; \fill (4, 0) node{\tiny$1$} circle (.5pt);\fill (4, 1) node{\tiny$1$} circle (.5pt);\fill (7, 3) node{\tiny$1$} circle (.5pt);\fill (7, 1) node{\tiny$1$} circle (.5pt);\fill (6, 3) node{\tiny$1$} circle (.5pt);\fill (5, 2) node{\tiny$2$} circle (.5pt);\fill (5, 1) node{\tiny$2$} circle (.5pt);\fill (6, 2) node{\tiny$2$} circle (.5pt);\fill (6, 1) node{\tiny$1$} circle (.5pt);\fill (7, 2) node{\tiny$2$} circle (.5pt);\end{tikzpicture}};
\draw (52) node {\begin{tikzpicture}[scale=1/5] \draw (6, 0) -- (5, 1) -- (6, 3) -- (7, 3) -- (7, 1) -- cycle; \fill (5, 1) node{\tiny$1$} circle (.5pt);\fill (6, 0) node{\tiny$1$} circle (.5pt);\fill (7, 3) node{\tiny$1$} circle (.5pt);\fill (7, 1) node{\tiny$1$} circle (.5pt);\fill (6, 3) node{\tiny$1$} circle (.5pt);\fill (7, 2) node{\tiny$2$} circle (.5pt);\fill (6, 1) node{\tiny$1$} circle (.5pt);\fill (6, 2) node{\tiny$1$} circle (.5pt);\end{tikzpicture}};
\draw (53) node {\begin{tikzpicture}[scale=1/5] \draw (6, 0) -- (5, 1) -- (6, 3) -- (8, 3) -- (8, 1) -- (7, 0) -- cycle; \fill (5, 1) node{\tiny$1$} circle (.5pt);\fill (6, 0) node{\tiny$1$} circle (.5pt);\fill (8, 3) node{\tiny$1$} circle (.5pt);\fill (8, 1) node{\tiny$1$} circle (.5pt);\fill (6, 3) node{\tiny$1$} circle (.5pt);\fill (7, 0) node{\tiny$1$} circle (.5pt);\fill (7, 1) node{\tiny$2$} circle (.5pt);\fill (6, 1) node{\tiny$2$} circle (.5pt);\fill (8, 2) node{\tiny$2$} circle (.5pt);\fill (7, 2) node{\tiny$3$} circle (.5pt);\fill (6, 2) node{\tiny$1$} circle (.5pt);\fill (7, 3) node{\tiny$2$} circle (.5pt);\end{tikzpicture}};
\draw (54) node {\begin{tikzpicture}[scale=1/5] \draw (4, 1) -- (4, 0) -- (5, 0) -- (7, 1) -- (7, 2) -- (6, 2) -- cycle; \fill (4, 0) node{\tiny$1$} circle (.5pt);\fill (4, 1) node{\tiny$1$} circle (.5pt);\fill (5, 0) node{\tiny$1$} circle (.5pt);\fill (7, 2) node{\tiny$1$} circle (.5pt);\fill (7, 1) node{\tiny$1$} circle (.5pt);\fill (6, 2) node{\tiny$1$} circle (.5pt);\fill (5, 1) node{\tiny$1$} circle (.5pt);\fill (6, 1) node{\tiny$1$} circle (.5pt);\end{tikzpicture}};
\draw (55) node {\begin{tikzpicture}[scale=1/5] \draw (4, 2) -- (4, 0) -- (5, 0) -- (7, 1) -- (7, 3) -- (6, 3) -- cycle; \fill (4, 0) node{\tiny$1$} circle (.5pt);\fill (7, 3) node{\tiny$1$} circle (.5pt);\fill (4, 2) node{\tiny$1$} circle (.5pt);\fill (5, 0) node{\tiny$1$} circle (.5pt);\fill (7, 1) node{\tiny$1$} circle (.5pt);\fill (6, 3) node{\tiny$1$} circle (.5pt);\fill (4, 1) node{\tiny$2$} circle (.5pt);\fill (5, 1) node{\tiny$2$} circle (.5pt);\fill (5, 2) node{\tiny$1$} circle (.5pt);\fill (6, 1) node{\tiny$1$} circle (.5pt);\fill (6, 2) node{\tiny$2$} circle (.5pt);\fill (7, 2) node{\tiny$2$} circle (.5pt);\end{tikzpicture}};
\draw (56) node {\begin{tikzpicture}[scale=1/5] \draw (4, 3) -- (4, 2) -- (5, 1) -- (7, 1) -- (6, 2) -- cycle; \fill (4, 2) node{\tiny$1$} circle (.5pt);\fill (4, 3) node{\tiny$1$} circle (.5pt);\fill (5, 1) node{\tiny$1$} circle (.5pt);\fill (7, 1) node{\tiny$1$} circle (.5pt);\fill (6, 2) node{\tiny$1$} circle (.5pt);\fill (6, 1) node{\tiny$2$} circle (.5pt);\fill (5, 2) node{\tiny$2$} circle (.5pt);\end{tikzpicture}};
\draw (57) node {\begin{tikzpicture}[scale=1/5] \draw (4, 3) -- (4, 2) -- (5, 1) -- (6, 1) -- (6, 2) -- (5, 3) -- cycle; \fill (4, 2) node{\tiny$1$} circle (.5pt);\fill (4, 3) node{\tiny$1$} circle (.5pt);\fill (5, 1) node{\tiny$1$} circle (.5pt);\fill (6, 2) node{\tiny$1$} circle (.5pt);\fill (5, 3) node{\tiny$1$} circle (.5pt);\fill (6, 1) node{\tiny$1$} circle (.5pt);\fill (5, 2) node{\tiny$2$} circle (.5pt);\end{tikzpicture}};
\draw (58) node {\begin{tikzpicture}[scale=1/5] \draw (4, 3) -- (4, 2) -- (5, 1) -- (7, 1) -- (7, 2) -- (6, 3) -- cycle; \fill (4, 2) node{\tiny$1$} circle (.5pt);\fill (4, 3) node{\tiny$1$} circle (.5pt);\fill (5, 1) node{\tiny$1$} circle (.5pt);\fill (7, 2) node{\tiny$1$} circle (.5pt);\fill (7, 1) node{\tiny$1$} circle (.5pt);\fill (6, 3) node{\tiny$1$} circle (.5pt);\fill (6, 1) node{\tiny$2$} circle (.5pt);\fill (5, 2) node{\tiny$3$} circle (.5pt);\fill (6, 2) node{\tiny$3$} circle (.5pt);\fill (5, 3) node{\tiny$2$} circle (.5pt);\end{tikzpicture}};
\draw (59) node {\begin{tikzpicture}[scale=1/5] \draw (4, 3) -- (4, 2) -- (5, 1) -- (6, 1) -- (7, 2) -- (7, 3) -- cycle; \fill (4, 2) node{\tiny$1$} circle (.5pt);\fill (4, 3) node{\tiny$1$} circle (.5pt);\fill (5, 1) node{\tiny$1$} circle (.5pt);\fill (7, 3) node{\tiny$1$} circle (.5pt);\fill (7, 2) node{\tiny$1$} circle (.5pt);\fill (6, 1) node{\tiny$1$} circle (.5pt);\fill (6, 2) node{\tiny$3$} circle (.5pt);\fill (5, 2) node{\tiny$3$} circle (.5pt);\fill (6, 3) node{\tiny$3$} circle (.5pt);\fill (5, 3) node{\tiny$3$} circle (.5pt);\end{tikzpicture}};
\draw (60) node {\begin{tikzpicture}[scale=1/5] \draw (5, 1) -- (4, 2) -- (6, 4) -- (7, 4) -- (7, 1) -- cycle; \fill (4, 2) node{\tiny$1$} circle (.5pt);\fill (5, 1) node{\tiny$1$} circle (.5pt);\fill (7, 4) node{\tiny$1$} circle (.5pt);\fill (7, 1) node{\tiny$1$} circle (.5pt);\fill (6, 4) node{\tiny$1$} circle (.5pt);\fill (6, 1) node{\tiny$2$} circle (.5pt);\fill (7, 2) node{\tiny$3$} circle (.5pt);\fill (6, 2) node{\tiny$5$} circle (.5pt);\fill (5, 2) node{\tiny$3$} circle (.5pt);\fill (7, 3) node{\tiny$3$} circle (.5pt);\fill (6, 3) node{\tiny$4$} circle (.5pt);\fill (5, 3) node{\tiny$2$} circle (.5pt);\end{tikzpicture}};
\draw (61) node {\begin{tikzpicture}[scale=1/5] \draw (5, 1) -- (4, 2) -- (6, 3) -- (7, 3) -- (7, 1) -- cycle; \fill (4, 2) node{\tiny$1$} circle (.5pt);\fill (5, 1) node{\tiny$1$} circle (.5pt);\fill (7, 3) node{\tiny$1$} circle (.5pt);\fill (7, 1) node{\tiny$1$} circle (.5pt);\fill (6, 3) node{\tiny$1$} circle (.5pt);\fill (6, 1) node{\tiny$2$} circle (.5pt);\fill (7, 2) node{\tiny$2$} circle (.5pt);\fill (6, 2) node{\tiny$3$} circle (.5pt);\fill (5, 2) node{\tiny$2$} circle (.5pt);\end{tikzpicture}};
\draw (62) node {\begin{tikzpicture}[scale=1/5] \draw (4, 3) -- (4, 2) -- (5, 1) -- (7, 1) -- (7, 4) -- (6, 4) -- cycle; \fill (4, 2) node{\tiny$1$} circle (.5pt);\fill (4, 3) node{\tiny$1$} circle (.5pt);\fill (5, 1) node{\tiny$1$} circle (.5pt);\fill (7, 4) node{\tiny$1$} circle (.5pt);\fill (7, 1) node{\tiny$1$} circle (.5pt);\fill (6, 4) node{\tiny$1$} circle (.5pt);\fill (6, 1) node{\tiny$2$} circle (.5pt);\fill (7, 2) node{\tiny$3$} circle (.5pt);\fill (6, 2) node{\tiny$5$} circle (.5pt);\fill (5, 2) node{\tiny$3$} circle (.5pt);\fill (7, 3) node{\tiny$3$} circle (.5pt);\fill (6, 3) node{\tiny$4$} circle (.5pt);\fill (5, 3) node{\tiny$2$} circle (.5pt);\end{tikzpicture}};
\draw (63) node {\begin{tikzpicture}[scale=1/5] \draw (5, 2) -- (5, -1) -- (7, 0) -- (7, 1) -- (6, 2) -- cycle; \fill (5, -1) node{\tiny$1$} circle (.5pt);\fill (7, 1) node{\tiny$1$} circle (.5pt);\fill (7, 0) node{\tiny$1$} circle (.5pt);\fill (5, 2) node{\tiny$1$} circle (.5pt);\fill (6, 2) node{\tiny$1$} circle (.5pt);\fill (5, 1) node{\tiny$3$} circle (.5pt);\fill (5, 0) node{\tiny$3$} circle (.5pt);\fill (6, 1) node{\tiny$2$} circle (.5pt);\fill (6, 0) node{\tiny$1$} circle (.5pt);\end{tikzpicture}};
\draw (64) node {\begin{tikzpicture}[scale=1/5] \draw (4, 4) -- (4, 2) -- (5, 1) -- (6, 1) -- (6, 2) -- cycle; \fill (4, 2) node{\tiny$1$} circle (.5pt);\fill (6, 2) node{\tiny$1$} circle (.5pt);\fill (4, 4) node{\tiny$1$} circle (.5pt);\fill (5, 1) node{\tiny$1$} circle (.5pt);\fill (6, 1) node{\tiny$1$} circle (.5pt);\fill (4, 3) node{\tiny$2$} circle (.5pt);\fill (5, 3) node{\tiny$2$} circle (.5pt);\fill (5, 2) node{\tiny$3$} circle (.5pt);\end{tikzpicture}};
\draw (65) node {\begin{tikzpicture}[scale=1/5] \draw (4, 4) -- (4, 2) -- (5, 1) -- (7, 1) -- (7, 2) -- (5, 4) -- cycle; \fill (4, 2) node{\tiny$1$} circle (.5pt);\fill (7, 2) node{\tiny$1$} circle (.5pt);\fill (4, 4) node{\tiny$1$} circle (.5pt);\fill (5, 1) node{\tiny$1$} circle (.5pt);\fill (7, 1) node{\tiny$1$} circle (.5pt);\fill (5, 4) node{\tiny$1$} circle (.5pt);\fill (6, 1) node{\tiny$2$} circle (.5pt);\fill (5, 2) node{\tiny$4$} circle (.5pt);\fill (4, 3) node{\tiny$2$} circle (.5pt);\fill (6, 2) node{\tiny$4$} circle (.5pt);\fill (5, 3) node{\tiny$4$} circle (.5pt);\fill (6, 3) node{\tiny$2$} circle (.5pt);\end{tikzpicture}};
\draw (66) node {\begin{tikzpicture}[scale=1/5] \draw (4, 4) -- (4, 2) -- (6, 1) -- (7, 1) -- cycle; \fill (4, 2) node{\tiny$1$} circle (.5pt);\fill (7, 1) node{\tiny$1$} circle (.5pt);\fill (4, 4) node{\tiny$1$} circle (.5pt);\fill (6, 1) node{\tiny$1$} circle (.5pt);\fill (5, 3) node{\tiny$3$} circle (.5pt);\fill (6, 2) node{\tiny$3$} circle (.5pt);\fill (4, 3) node{\tiny$2$} circle (.5pt);\fill (5, 2) node{\tiny$3$} circle (.5pt);\end{tikzpicture}};
\draw (67) node {\begin{tikzpicture}[scale=1/5] \draw (2, 2) -- (2, 0) -- (4, 0) -- (4, 2) -- cycle; \fill (2, 0) node{\tiny$1$} circle (.5pt);\fill (4, 2) node{\tiny$1$} circle (.5pt);\fill (2, 2) node{\tiny$1$} circle (.5pt);\fill (4, 0) node{\tiny$1$} circle (.5pt);\fill (2, 1) node{\tiny$2$} circle (.5pt);\fill (3, 0) node{\tiny$2$} circle (.5pt);\fill (3, 1) node{\tiny$4$} circle (.5pt);\fill (3, 2) node{\tiny$2$} circle (.5pt);\fill (4, 1) node{\tiny$2$} circle (.5pt);\end{tikzpicture}};
\end{tikzpicture}
\caption{A mutation graph showing mutation relationships between zero mutable Laurent polynomials (ZMLPs) with small Netwon polytopes.}
\label{fig:graph}
\end{figure}

\begin{expl}
\label{expl:tomjerrymut}
Consider the triangle $\triangle=\textup{Conv}\{(0,0),(0,2),(3,0)\}$. The corresponding toric variety is the affine cone over the weighted projective plane $\mathbb{P}(1,2,3)$. This has two smoothing components, called Tom and Jerry in the literature \cite{BKR}\cite{BRS}. Figure \ref{fig:tomjerrymut} shows the Laurent polynomials corresponding to Tom and Jerry and mutations to $f=1$, showing that they are zero mutable. The mutations are presented in blue, and $\varphi$ is given by its level sets.
\end{expl}

\begin{figure}[h!]
\centering
\begin{tikzpicture}[scale=.95]
\draw (1.5,3) node{Tom};
\draw[blue] (-.5,0) -- (3.5,0) node[right]{$\varphi=-3$};
\draw[blue] (-.5,1) -- (2,1) node[right]{$\varphi=-1$};
\draw[blue] (-.5,2) -- (.5,2) node[right]{$\varphi=1$};
\draw[blue] (2,2) node[right]{$h=1+x$};
\draw (0,0) -- (0,2) -- (3,0) -- cycle;
\fill (0,0) circle (2pt) node[below]{$1$};
\fill (1,0) circle (2pt) node[below]{$3$};
\fill (2,0) circle (2pt) node[below]{$3$};
\fill (3,0) circle (2pt) node[below]{$1$};
\fill (0,1) circle (2pt) node[left]{$2$};
\fill (1,1) circle (2pt) node[below]{$2$};
\fill (0,2) circle (2pt) node[left]{$1$};
\draw (1.5,-1.5) node{$(1+x)^3+2y(1+x)+y^2$};
\draw[->] (4,1) -- (5,1);
\draw[|->] (4,-1.5) -- (4.5,-1.5);
\draw[blue] (6,2.5) -- (6,-.5) node[below]{$\varphi=-2$};
\draw[blue] (7,2.5) -- (7,1) node[below]{$\varphi=0$};
\draw[blue] (6.5,0) node[right]{$h=1+y^{-1}$};
\draw (6,0) -- (6,2) -- (7,2) -- cycle;
\fill (6,0) circle (2pt) node[left]{$1$};
\fill (6,1) circle (2pt) node[left]{$2$};
\fill (6,2) circle (2pt) node[left]{$1$};
\fill (7,2) circle (2pt) node[right]{$1$};
\draw (6.5,-1.5) node{$y^2(1+y^{-1})^2+xy^2$};
\draw[->] (8,1) -- (9,1);
\draw[|->] (8.5,-1.5) -- (9,-1.5);
\draw[blue] (9.5,2) -- (11.5,2) node[right]{$\varphi=-1$};
\draw[blue] (10.5,0) node{$h=y^2(1+x)$};
\draw (10,2) -- (11,2);
\fill (10,2) circle (2pt) node[below]{$1$};
\fill (11,2) circle (2pt) node[below]{$1$};
\draw (10.5,-1.5) node{$y^2(1+x)$};
\draw[->] (12,1) -- (13,1);
\draw[|->] (12.5,-1.5) -- (13,-1.5);
\fill(14,0) circle (2pt) node[left]{$1$};
\draw (14,-1.5) node{$1$};
\end{tikzpicture}
\begin{tikzpicture}[scale=.95]
\draw (1.9,3) node{Jerry};
\draw[blue] (0,-.5) -- (0,2.5) node[above]{$\varphi=-2$};
\draw[blue] (1,-.5) -- (1,1.75) node[above]{$\varphi=-1$};
\draw[blue] (2,-.5) -- (2,1) node[above]{$\varphi=-0$};
\draw[blue] (3,-.5) -- (3,.25) node[above]{$\varphi=1$};
\draw[blue] (2,2) node[right]{$h=1+y$};
\draw (0,0) -- (0,2) -- (3,0) -- cycle;
\fill (0,0) circle (2pt) node[below]{$1$};
\fill (1,0) circle (2pt) node[below]{$3$};
\fill (2,0) circle (2pt) node[below]{$3$};
\fill (3,0) circle (2pt) node[below]{$1$};
\fill (0,1) circle (2pt) node[left]{$2$};
\fill (1,1) circle (2pt) node[below]{$3$};
\fill (0,2) circle (2pt) node[left]{$1$};
\draw (1.5,-1.5) node{$(1+y)^2+3x(1+y)+3x^2+x^3$};
\draw[->] (4,1) -- (5,1);
\draw[|->] (4.5,-1.5) -- (5,-1.5);
\draw[blue] (5.5,0) node[left]{$\varphi=-3$} -- (9.5,0);
\draw[blue] (7.5,1) node[left]{$\varphi=0$} -- (9.5,1);
\draw[blue] (7.5,2) node[right]{$h=1+x^{-1}$};
\draw (6,0) -- (9,0) -- (9,1) -- cycle;
\fill (6,0) circle (2pt) node[below]{$1$};
\fill (7,0) circle (2pt) node[below]{$3$};
\fill (8,0) circle (2pt) node[below]{$3$};
\fill (9,0) circle (2pt) node[below]{$1$};
\fill (9,1) circle (2pt) node[right]{$1$};
\draw (7.5,-1.5) node{$x^3(1+x^{-1})^3+x^3y$};
\draw[->] (10,1) -- (11,1);
\draw[|->] (10,-1.5) -- (10.5,-1.5);
\draw[blue] (12,-.5) -- (12,1.5) node[above]{$\varphi=-1$};
\draw[blue] (12,2.5) node{$h=x^3(1+y)$};
\draw (12,0) -- (12,1);
\fill (12,0) circle (2pt) node[right]{$1$};
\fill (12,1) circle (2pt) node[right]{$1$};
\draw (12,-1.5) node{$x^3(1+y)$};
\draw[->] (13,1) -- (14,1);
\draw[|->] (13.5,-1.5) -- (14,-1.5);
\fill(15,0) circle (2pt) node[left]{$1$};
\draw (15,-1.5) node{$1$};
\end{tikzpicture}
\caption{Sequences of mutations showing that Tom and Jerry is zero mutable.}
\label{fig:tomjerrymut}
\end{figure}
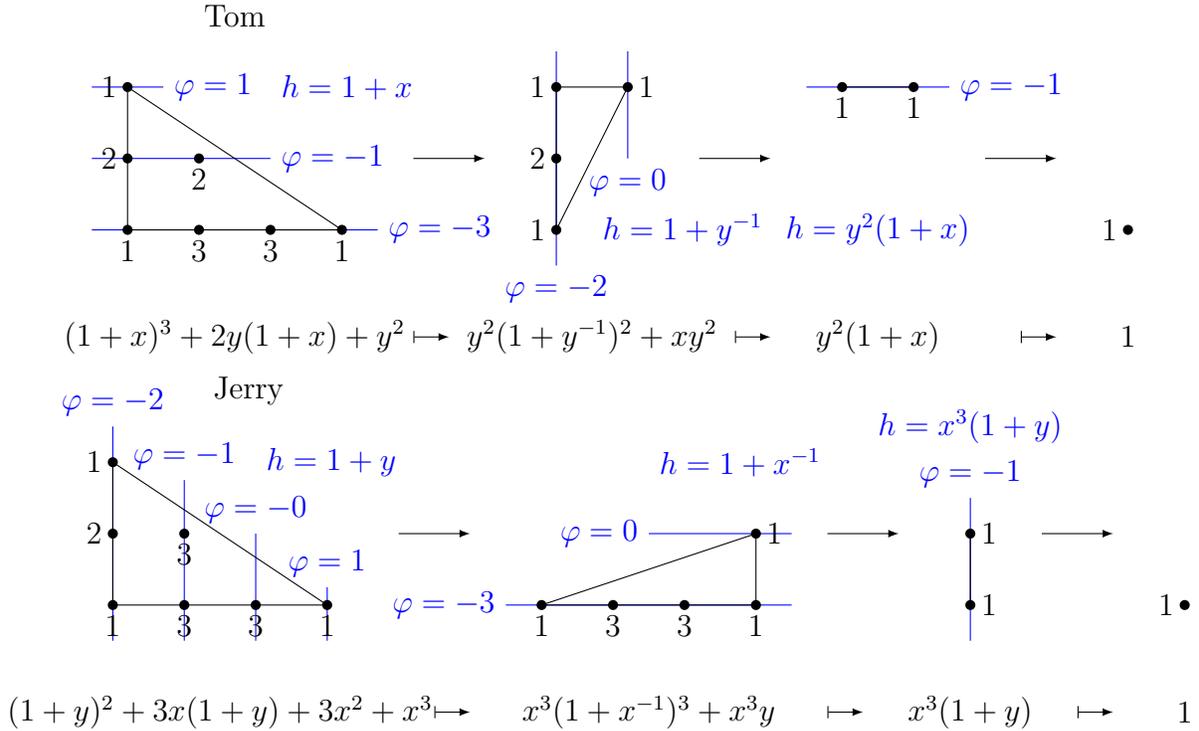

\begin{rem}
\begin{enumerate}[(1)]
\item The property of being zero mutable is preserved under automorphisms of $\mathbb{C}[M]$ and multiplication with monomials. If we represent $f$ by the coefficients on its Newton polytope, then zero mutability only depends on the isomorphisms class under transformations in the integral affine transformation group $\textup{GL}_n(\mathbb{Z})\ltimes\mathbb{Z}^n$ of $M\simeq\mathbb{Z}^n$.
\item In the projective case one considers mutations of Laurent polynomials supported on reflexive polytopes. A reflexive polytope has a unique interior lattice point that can be taken to be the origin. Then only mutations with respect to linear functions $\varphi$ (i.e., $\varphi(0)=0$) are considered, as in \cite{CKPT}.
\item Mutation acts on monomials by $\textup{mut}_{(\varphi,h)}(z^m) = z^m f^{\varphi(m)}$. If $\varphi$ is linear and $n_\varphi$ is the normal vector to the level sets of $\varphi$, then $\varphi(m)=\braket{m,n_\varphi}$ and $\textup{mut}_{(\varphi,h)}(z^m) = z^m f^{\braket{m,n_\varphi}}$ resembles a wall crossing morphism of a scattering diagram, see e.g. \cite{GS}.
\item Under mirror symmetry, Fano varieties correspond to Laurent polynomials \cite{CCG+}\cite{ACGK}. Hence, Conjecture \ref{conj:affine} is somewhat motivated by mirror symmetry. See \cite{CFP}, \S4.2, for more details. Together with the remark (3) above and the relation between scattering diagrams and toric degenerations \cite{GS} this motivated our proposal outlined in the introduction.
\end{enumerate}
\end{rem}

\begin{prop}
If $f$ is zero mutable, then its coefficients are positive integers.
\end{prop}

\begin{proof}
$f=1$ has positive integer coefficients and all products and mutations of Laurent polynomials with positive  integer coefficients have positive integer coefficients.
\end{proof}

\begin{prop}
If $f$ is zero mutable, then its restrictions to all facets are zero mutable.
\end{prop}

\begin{proof}
Let $f$ be zero mutable. Restriction to facets commutes with taking products, so we can assume that $f$ is irreducible. Let $1 = f_0 \mapsto f_1 \mapsto \ldots f_{n-1} \mapsto f_n = f$ be a sequence of mutations from $1$ to $f$. We show that the statement is true for all $f_i$. For $f_0=1$ there is nothing to show. Assume that the statement holds for $f_{i-1}$. A new facet in $f_i$ is obtained from $f_{i-1}$ by multiplying a monomial $z^m$ with a zero mutable $h$. Hence, the coefficients on the facet are zero mutable.
\end{proof}

\begin{cor}
\label{cor:binomial}
Let $f$ be zero mutable and let $m$ be a lattice point on an edge $e$ of the Newton polytope of $f$. The the $z^m$-coefficient of $f$ is $\binom{\ell(e)}{k}$, where $\ell(e)$ is the lattice length of $e$ and $k$ is the lattice distance of $m$ from a vertex of $e$.
\end{cor}

\subsection{ZMLPs on polygons: divisibility steps}												%%

We explain how zero mutable Laurent polynomials on polygons induce tuples of decreasing partitions that we call divisibility steps.

\begin{defi}
\label{defi:div}
Let $f$ be a Laurent polynomial with Newton polygon $\triangle=\textup{Newt}(f)$. For each edge $e$ of $\triangle$ let $m_e$ be a primitive tangent vector of $e$ and let $n_e$ be the primitive inner normal vector. We can write
\[ f = \sum_{k=k_{\textup{min}}}^{k_{\textup{max}}} f_k, \]
where $f_k$ is supported on $\triangle_k = \{m \in \triangle \cap M \ | \ \braket{m,n_e} = k \}$, $k_{\textup{min}} = \min\{k \ | \ \triangle_k \neq \emptyset\}$ and $k_{\textup{max}} = \max\{k \ | \ \triangle_k \neq \emptyset\}$. Let $d_k$ be the multiplicity of $1+z^{m_e}$ in $f_k$, i.e. we can write
\[ f_k = (1+z^{m_e})^{d_k} r_k \]
for a Laurent polynomial $r_k$ not divisible by $1+z^{m_e}$. Note that $d_k$ does not depend on the choice of primitive tangent vector $m_e$, since we can multiply $1+z^{m_e}$ by $z^{-m_e}$. If $f_k=0$, we define $d_k=\infty$, see Example \ref{expl:inf}. Let $d=\min(\{d_k\})$.  Define the \emph{divisibility tuple} of $f$ with respect to $e$ by
\[ \mathbf{div}_e(f) = \left(d_k\right)_{k=k_{\textup{min}}}^{k_{\textup{max}}}. \]
\end{defi}

The divisibility tuples give conditions on the coefficients of a Laurent polynomial. It might happen that these conditions are linearly dependent, so that a smaller tuple gives the same conditions. This leads us to the following definition.

We write $\mathbf{d}\leq \mathbf{d}'$ for tuples $\mathbf{d}=(d_1,\ldots,d_r)$ and $\mathbf{d}'=(d'_1,\ldots,d'_r)$ of the same length if $d_i\leq d'_i$ for all $i=1,\ldots,r$, and hence can talk about the smallest tuple with a given length.

\begin{defi}
\label{defi:reqdiv}
Let $f$ be a Laurent polynomial with Netwon polygon $\triangle=\textup{Newt}(f)$. The \emph{required divisibility tuples} of $f$ are the collection (over the edges $e$ of $\triangle$) of smallest tuples $\mathbf{reqdiv}_e(f)$ of length equal to the length of $\mathbf{div}_e(f)$ that satisfy the following property. If $f'$ is a Laurent polynomial with $\textup{Newt}(f')=\triangle$ and for all edges $e$ of $\triangle$ we have $\mathbf{reqdiv}_e(f)\leq \mathbf{div}_e(f')$, then $\mathbf{div}_e(f)=\mathbf{div}_e(f')$. This is well-defined up to automorphisms of $\triangle=\textup{Newt}(f)$. Note that the elements of $\mathbf{reqdiv}_e(f)$ are finite, while the elements of $\mathbf{div}_e(f)$ might be infinite, see Example \ref{expl:inf}.

Write $\mathbf{reqdiv}_e(f)=(d_k)_{k=k_{\textup{min}}}^{k_{\textup{max}}}$. Define the \emph{required divisibility} with respect to an edge $e$ by $\textup{reqdiv}_e(f)=\sum_{k=k_{\textup{min}}}^{k_{\textup{max}}}d_k$ and the \emph{(total) required divisibility} by $\textup{reqdiv}(f)= \sum_e \textup{reqdiv}_e(f)$, where for parallel edges $e,e'$ we pick one, e.g. $e$, and only include $e$ in the sum, but not $e'$. This does not depend on the choice of $e$, since $\mathbf{reqdiv}_e(f)$ and $\mathbf{reqdiv}_{e'}(f)$ are the reverse of each other, hence $\textup{reqdiv}_e(f)=\textup{reqdiv}_{e'}(f)$.
\end{defi}

\begin{defi}
\label{defi:divstep}
Let $f$ be an Laurent polynomial with Newton polygon $\triangle=\textup{Newt}(f)$. Let $e$ be an edge of $\triangle$ and write $\mathbf{reqdiv}_e(f)=(d_k)_{k=k_{\textup{min}}}^{k_{\textup{max}}}$. 

If $f$ is irreducible, then $\min(\{d_k\})=0$, since otherwise $f$ would be divisible by $(1+z^{m_e})^d$ for $d=\min(\{d_k\})$. In this case, let $\overline{k}_{\textup{max}}=\min(\{k \ | | d_k=0\})$ and define the \emph{divisibility steps} of $f$ with respect to $e$ by
\[ \mathbf{reqdivstep}_e(f)=(d_k-d_{k+1})_{k=k_{\textup{min}}}^{\overline{k}_{\textup{max}}-1}. \]
We will show in Proposition \ref{prop:divstep} that $\mathbf{reqdivstep}_e(f)$ is a decreasing tuple.

If $f$ is reducible, let $f=f_1\cdots f_r$ be a factorization into irreducible components and define $\mathbf{reqdivstep}_e(f)$ as the concatenation of decreasing tuples
\[ \mathbf{reqdivstep}_e(f) = \mathbf{reqdivstep}_e(f_1) \circ \ldots \circ \mathbf{reqdivstep}_e(f_r). \]
\end{defi}

\begin{expl}
\label{expl:tomjerrydiv}
For Tom and Jerry from Example \ref{expl:tomjerrymut}, Figure \ref{fig:tomjerrydiv} shows the divisibility tuples, which are equal to the required divisibility tuples, and the divisibility steps.
\end{expl}

\begin{figure}[h!]
\centering
\begin{tikzpicture}[scale=1]
\draw (1.5,2.6) node{Tom};
\draw[black!40!red] (1.5,-2.9) node{$\mathbf{divstep}(f)=((2,1),(2),(1))$};
\draw[black!40!green] (0,2.2) -- (0,-1) node[below]{$2$};
\draw[black!40!green] (1,2.2) -- (1,-1) node[below]{$0$};
\draw[black!40!green] (2,2.2) -- (2,-1) node[below]{$0$};
\draw[black!40!green] (3,2.2) -- (3,-1) node[below]{$0$};
\draw[->,black!40!red] (0,-1.7) to[bend right=30] node[midway,below]{$2$} (1,-1.7);
\draw[black!40!green] (3.2,0) -- (-1,0) node[left]{$3$};
\draw[black!40!green] (3.2,1) -- (-1,1) node[left]{$1$};
\draw[black!40!green] (3.2,2) -- (-1,2) node[left]{$0$};
\draw[->,black!40!red] (-1.6,0) to[bend left=30] node[midway,left]{$2$} (-1.6,1);
\draw[->,black!40!red] (-1.6,1) to[bend left=30] node[midway,left]{$1$} (-1.6,2);
\draw (0,0) -- (0,2) -- (3,0) -- (0,0);
\fill (0,0) circle (1.5pt) node[below]{$1$};
\fill (1,0) circle (1.5pt) node[below]{$3$};
\fill (2,0) circle (1.5pt) node[below]{$3$};
\fill (3,0) circle (1.5pt) node[below]{$1$};
\fill (0,1) circle (1.5pt) node[left]{$2$};
\fill (1,1) circle (1.5pt) node[below]{$2$};
\fill (0,2) circle (1.5pt) node[left]{$1$};
\end{tikzpicture}
\hspace{5mm}
\begin{tikzpicture}[scale=1]
\draw (1.5,2.6) node{Jerry};
\draw[black!40!red] (1.5,-2.9) node{$\mathbf{divstep}(f)=((3),(1,1),(1))$};
\draw[black!40!green] (0,2.2) -- (0,-1) node[below]{$2$};
\draw[black!40!green] (1,2.2) -- (1,-1) node[below]{$1$};
\draw[black!40!green] (2,2.2) -- (2,-1) node[below]{$0$};
\draw[black!40!green] (3,2.2) -- (3,-1) node[below]{$0$};
\draw[->,black!40!red] (0,-1.7) to[bend right=30] node[midway,below]{$1$} (1,-1.7);
\draw[->,black!40!red] (1,-1.7) to[bend right=30] node[midway,below]{$1$} (2,-1.7);
\draw[black!40!green] (3.2,0) -- (-1,0) node[left]{$3$};
\draw[black!40!green] (3.2,1) -- (-1,1) node[left]{$0$};
\draw[black!40!green] (3.2,2) -- (-1,2) node[left]{$0$};
\draw[->,black!40!red] (-1.6,0) to[bend left=30] node[midway,left]{$3$} (-1.6,1);
\draw (0,0) -- (0,2) -- (3,0) -- (0,0);
\fill (0,0) circle (1.5pt) node[below]{$1$};
\fill (1,0) circle (1.5pt) node[below]{$3$};
\fill (2,0) circle (1.5pt) node[below]{$3$};
\fill (3,0) circle (1.5pt) node[below]{$1$};
\fill (0,1) circle (1.5pt) node[left]{$2$};
\fill (1,1) circle (1.5pt) node[below]{$3$};
\fill (0,2) circle (1.5pt) node[left]{$1$};
\end{tikzpicture}
\caption{Zero mutable Laurent polynomials for Tom (left) and Jerry (right) with (required) divisibility tuples (green) and divisibility steps (red).}
\label{fig:tomjerrydiv}
\end{figure}

\begin{expl}
Figure \ref{fig:reqdiv} shows a zero mutable Laurent polynomial with $\textup{reqdiv}(f) < \textup{div}(f)$. This is the smallest example on a rectangular triangle with this property that we could find.
\end{expl}

\begin{figure}[h!]
\centering
\begin{tikzpicture}[scale=.7]
\draw[black!40!red] (3.5,-3.6) node{$\mathbf{divstep}(f)=((3,2,2),(2,1,1,1),(1))$};
\draw[black!40!green] (0,5.2) -- (0,-1) node[below]{$5$};
\draw[black!40!green] (1,5.2) -- (1,-1) node[below]{$3$};
\draw[black!40!green] (2,5.2) -- (2,-1) node[below]{$2$};
\draw[black!40!green] (3,5.2) -- (3,-1) node[below]{$1$};
\draw[black!40!green] (4,5.2) -- (4,-1) node[below]{$0$};
\draw[black!40!green] (5,5.2) -- (5,-1) node[below]{$0$};
\draw[black!40!green] (6,5.2) -- (6,-1) node[below]{$0$};
\draw[black!40!green] (7,5.2) -- (7,-1) node[below]{$0$};
\draw[black!40!red] (0,-1.5) node[below]{$5$};
\draw[black!40!red] (1,-1.5) node[below]{$3$};
\draw[black!40!red] (2,-1.5) node[below]{$2$};
\draw[black!40!red] (3,-1.5) node[below]{$1$};
\draw[black!40!red] (4,-1.5) node[below]{$0$};
\draw[black!40!red] (5,-1.5) node[below]{$0$};
\draw[black!40!red] (6,-1.5) node[below]{$0$};
\draw[black!40!red] (7,-1.5) node[below]{$0$};
\draw[->,black!40!red] (0,-2.2) to[bend right=30] node[midway,below]{$2$} (1,-2.2);
\draw[->,black!40!red] (1,-2.2) to[bend right=30] node[midway,below]{$1$} (2,-2.2);
\draw[->,black!40!red] (2,-2.2) to[bend right=30] node[midway,below]{$1$} (3,-2.2);
\draw[->,black!40!red] (3,-2.2) to[bend right=30] node[midway,below]{$1$} (4,-2.2);
\draw[black!40!green] (7.2,0) -- (-1,0) node[left]{$7$};
\draw[black!40!green] (7.2,1) -- (-1,1) node[left]{$5$};
\draw[black!40!green] (7.2,2) -- (-1,2) node[left]{$2$};
\draw[black!40!green] (7.2,3) -- (-1,3) node[left]{$0$};
\draw[black!40!green] (7.2,4) -- (-1,4) node[left]{$0$};
\draw[black!40!green] (7.2,5) -- (-1,5) node[left]{$0$};
\draw[black!40!red] (-1.5,0) node[left]{$7$};
\draw[black!40!red] (-1.5,1) node[left]{$4$};
\draw[black!40!red] (-1.5,2) node[left]{$2$};
\draw[black!40!red] (-1.5,3) node[left]{$0$};
\draw[black!40!red] (-1.5,4) node[left]{$0$};
\draw[black!40!red] (-1.5,5) node[left]{$0$};
\draw[->,black!40!red] (-2.2,0) to[bend left=30] node[midway,left]{$3$} (-2.2,1);
\draw[->,black!40!red] (-2.2,1) to[bend left=30] node[midway,left]{$2$} (-2.2,2);
\draw[->,black!40!red] (-2.2,2) to[bend left=30] node[midway,left]{$2$} (-2.2,3);
\draw (0,0) -- (7,0) -- (0,5) -- cycle;
\fill (0,0) circle (2pt) node[below]{$1$};
\fill (1,0) circle (2pt) node[below]{$7$};
\fill (2,0) circle (2pt) node[below]{$21$};
\fill (3,0) circle (2pt) node[below]{$35$};
\fill (4,0) circle (2pt) node[below]{$35$};
\fill (5,0) circle (2pt) node[below]{$21$};
\fill (6,0) circle (2pt) node[below]{$7$};
\fill (7,0) circle (2pt) node[below]{$1$};
\fill (0,1) circle (2pt) node[below]{$5$};
\fill (1,1) circle (2pt) node[below]{$25$};
\fill (2,1) circle (2pt) node[below]{$50$};
\fill (3,1) circle (2pt) node[below]{$50$};
\fill (4,1) circle (2pt) node[below]{$25$};
\fill (5,1) circle (2pt) node[below]{$5$};
\fill (0,2) circle (2pt) node[below]{$10$};
\fill (1,2) circle (2pt) node[below]{$33$};
\fill (2,2) circle (2pt) node[below]{$37$};
\fill (3,2) circle (2pt) node[below]{$15$};
\fill (4,2) circle (2pt) node[below]{$1$};
\fill (0,3) circle (2pt) node[below]{$10$};
\fill (1,3) circle (2pt) node[below]{$19$};
\fill (2,3) circle (2pt) node[below]{$8$};
\fill (0,4) circle (2pt) node[below]{$5$};
\fill (1,4) circle (2pt) node[below]{$4$};
\fill (0,5) circle (2pt) node[below]{$1$};
\end{tikzpicture}
\caption{$f$ with $\mathbf{div}_e(f)$ (green) different from $\mathbf{reqdiv}_e(f)$ (red).}
\label{fig:reqdiv}
\end{figure}

\begin{expl}
\label{expl:inf}
For an example with $\textup{div}(f)=\infty$, see Figure \ref{fig:inf}.
\end{expl}

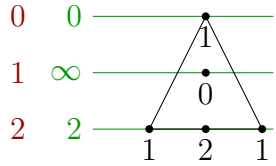
\begin{figure}[h!]
\centering
\begin{tikzpicture}[scale=.75]
\draw[black!40!green] (2.2,0) -- (-1,0) node[left]{$2$};
\draw[black!40!green] (2.2,1) -- (-1,1) node[left]{$\infty$};
\draw[black!40!green] (2.2,2) -- (-1,2) node[left]{$0$};
\draw[black!40!red] (-2,0) node[left]{$2$};
\draw[black!40!red] (-2,1) node[left]{$1$};
\draw[black!40!red] (-2,2) node[left]{$0$};
\draw (0,0) -- (2,0) -- (1,2) -- cycle;
\fill (0,0) circle (2pt) node[below]{$1$};
\fill (1,0) circle (2pt) node[below]{$2$};
\fill (2,0) circle (2pt) node[below]{$1$};
\fill (1,1) circle (2pt) node[below]{$0$};
\fill (1,2) circle (2pt) node[below]{$1$};
\end{tikzpicture}
\caption{A zero mutable Laurent polynomial $f$ with $\textup{div}(f)=\infty$.}
\label{fig:inf}
\end{figure}

\begin{expl}
Figure \ref{fig:reddiv} shows an example of a reducible zero mutable Laurent polynomial $f$ with $\textup{reddiv}(f)>0$. Let $e$ be the left edge and let $e'$ be the right edge of $\triangle=\textup{Newt}(f)$, so that $e$ and $e'$ are parallel with $m_e=m_{e'}=(0,1)$. For $e$ we have $(d_k)=(3,1,2)$ and for $e'$ we have $(d_k)=(2,1,3)$. Hence, we have $\textup{reddiv}_{\mathbb{R}(0,1)}(f)=1$ and $\mathbf{reqdiv}_e(f)=\mathbf{div}_e(f)=(2,0)$, $\mathbf{reqdiv}_{e'}(f)=\mathbf{div}_{e'}(f)=(1,0)$ and $\mathbf{reqdivstep}_e(f)=(2)$, $\mathbf{reqdivstep}_{e'}(f)=(1)$.
\end{expl}

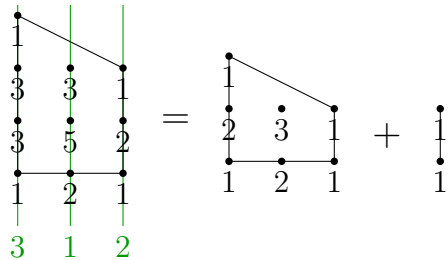
\begin{figure}[h!]
\centering
\begin{tikzpicture}[scale=.7]
\draw[black!40!green] (0,3.2) -- (0,-1) node[below]{$3$};
\draw[black!40!green] (1,3.2) -- (1,-1) node[below]{$1$};
\draw[black!40!green] (2,3.2) -- (2,-1) node[below]{$2$};
\draw (0,0) -- (2,0) -- (2,2) -- (0,3) -- cycle;
\fill (0,0) circle (2pt) node[below]{$1$};
\fill (1,0) circle (2pt) node[below]{$2$};
\fill (2,0) circle (2pt) node[below]{$1$};
\fill (0,1) circle (2pt) node[below]{$3$};
\fill (1,1) circle (2pt) node[below]{$5$};
\fill (2,1) circle (2pt) node[below]{$2$};
\fill (0,2) circle (2pt) node[below]{$3$};
\fill (1,2) circle (2pt) node[below]{$3$};
\fill (2,2) circle (2pt) node[below]{$1$};
\fill (0,3) circle (2pt) node[below]{$1$};
\draw (3,1) node{\large$=$};
\end{tikzpicture}
\begin{tikzpicture}[scale=.7]
\draw (0,0) -- (2,0) -- (2,1) -- (0,2) -- cycle;
\fill (0,0) circle (2pt) node[below]{$1$};
\fill (1,0) circle (2pt) node[below]{$2$};
\fill (2,0) circle (2pt) node[below]{$1$};
\fill (0,1) circle (2pt) node[below]{$2$};
\fill (1,1) circle (2pt) node[below]{$3$};
\fill (2,1) circle (2pt) node[below]{$1$};
\fill (0,2) circle (2pt) node[below]{$1$};
\draw (3,.5) node{\large$+$};
\draw (0,-2);
\end{tikzpicture}
\begin{tikzpicture}[scale=.7]
\draw (0,0) -- (0,1);
\fill (0,0) circle (2pt) node[below]{$1$};
\fill (0,1) circle (2pt) node[below]{$1$};
\draw (0,-2);
\end{tikzpicture}
\caption{A reducible zero mutable Laurent polynomial $f$ with $\textup{reddiv}(f)>0$.}
\label{fig:reddiv}
\end{figure}

\begin{defi}
\label{defi:mutve}
Let $f$ be a Laurent polynomial with Newton polytope $\triangle=\textup{Newt}(f)$ of dimension $\textup{dim}(\triangle)\leq 2$. Let $v$ be a vertex of $\triangle$. We define $\textup{mut}_{(\varphi,h)}(v)$ as follows. Let $f=\sum f_k$ as in Definition \ref{defi:mut}. Consider $k=\varphi(v)$. There are two cases.
\begin{enumerate}[(1)]
\item If $f_k$ is supported only on $v$ (this can only happen if $v$ is minimal or maximal with respect to $\varphi)$, then we have one of the following.
\begin{enumerate}[(a)]
\item If $k=0$, then $\textup{mut}_{(\varphi,h)}(f_k)=f_k$ and we define $\textup{mut}_{(\varphi,h)}(v)=v$.
\item Otherwise, $k>0$ (since $f_k\in\mathbb{C}$ must be divisible by $h\in\mathbb{C}[\textup{ker }\varphi_0]\setminus\mathbb{C}$) and $\textup{mut}_{(\varphi,h)}(f_k)=h^kf_k$ is supported on a line segment $L$. In this case we define $\textup{mut}_{(\varphi,h)}(v)=L$.
\end{enumerate}
\item Otherwise, $f_k$ is supported on a line segment, and $v$ is one of its endpoints. This means $v$ is minimal with respect to a linear form $\psi$ orthogonal to $\varphi$. Again, there are two cases.
\begin{enumerate}[(a)]
\item If $\textup{mut}_{(\varphi,h)}(f_k)=h^kf_k$ is supported on a single point (this can only happen if $k<0$), then define $\textup{mut}_{(\varphi,h)}(v)$ to be this point.
\item Otherwise, $\textup{mut}_{(\varphi,h)}(f_k)=h^kf_k$ is supported on a line segment $L$. Define $\textup{mut}_{(\varphi,h)}(v)$ to be the endpoint of $L$ that is minimal with respect to $\psi$.
\end{enumerate}
\end{enumerate}
Now if $e=\overline{v,w}$ is the edge of $\triangle$ connecting $v$ and $w$, then define $\textup{mut}_{(\varphi,h)}(e)$ as follows.
\begin{enumerate}[(1)]
\item If $\textup{mut}_{(\varphi,h)}(v)$ and $\textup{mut}_{(\varphi,h)}(w)$ are vertices, define $\textup{mut}_{(\varphi,h)}(e) = \overline{\textup{mut}_{(\varphi,h)}(v),\textup{mut}_{(\varphi,h)}(w)}$. If $\varphi$ is constant on $e$, then we might have $\textup{mut}_{(\varphi,h)}(v)=\textup{mut}_{(\varphi,h)}(w)$, so that this is a single vertex. Otherwise, it is an edge of $\textup{mut}_{(\varphi,h)}(\triangle)$.
\item If at least one of $\textup{mut}_{(\varphi,h)}(v)$ or $\textup{mut}_{(\varphi,h)}(w)$ is a line segment, define $\textup{mut}_{(\varphi,h)}(e)$ to be the unique edge of $\textup{mut}_{(\varphi,h)}(\triangle)$ whose vertices are contained in or equal to $\textup{mut}_{(\varphi,h)}(v)$ or $\textup{mut}_{(\varphi,h)}(w)$, respectively.
\end{enumerate}
Figure \ref{fig:mutve} explains the definitions above.
\end{defi}

\begin{figure}[h!]
\centering
\begin{tikzpicture}[scale=1.5]
\draw (0,0) -- (2,0) -- (1,2) -- (0,1) -- cycle;
\fill (0,0) circle (1pt) node[left]{$v_2$};
\fill (1,0) circle (1pt);
\fill (2,0) circle (1pt) node[right]{$v_3$};
\fill (0,1) circle (1pt) node[left]{$v_1$};
\fill (1,1) circle (1pt);
\fill (1,2) circle (1pt) node[above]{$v_4$};
\draw (0,.5) node[left]{$e_1$};
\draw (.8,0) node[below]{$e_2$};
\draw (1.5,1) node[right]{$e_3$};
\draw (.5,1.5) node[left]{$e_4$};
\draw[->] (3,1) -- (3.5,1) node[above]{$\textup{mut}_{(\braket{(1,0),\cdot},1+y)}$} -- (4.5,1);
\draw (6,0) -- (8,0) -- (8,1) -- (7,2) -- cycle;
\fill (6,0) circle (1pt);
\draw[<-] (6,0) to[bend right=10] (5,-.5) node[below]{$\textup{mut}(v_1)=\textup{mut}(v_2)=\textup{mut}(e_1)$};
\fill (7,0) circle (1pt);
\fill (8,0) circle (1pt);
\fill (7,1) circle (1pt);
\fill (8,1) circle (1pt);
\fill (7,2) circle (1pt) node[above]{$\textup{mut}(v_4)$};
\draw (6.8,0) node[below]{$\textup{mut}(e_2)$};
\draw (8,.5) node[right]{$\textup{mut}(v_3)$};
\draw (7.5,1.5) node[right]{$\textup{mut}(e_3)$};
\draw (6.5,1) node[left]{$\textup{mut}(e_4)$};
\end{tikzpicture}
\caption{A mutation explaining the definitions of $\textup{mut}_{(\varphi,h)}(v)$ and $\textup{mut}_{(\varphi,h)}(e)$.}
\label{fig:mutve}
\end{figure}
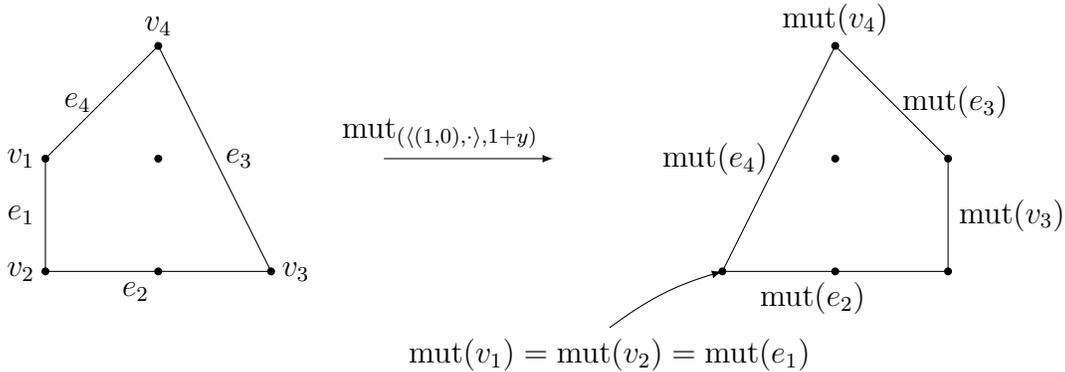

\begin{lem}
\label{lem:div}
Let $g(x)=\sum_{k\in\mathbb{Z}} g_kx^k \in \mathbb{C}[x^{\pm 1}]$ be a Laurent polynomial in one variable $x$ that is divisible by $(1+x)^d$. Then for all $a,b\in\mathbb{Z}$ and $i\in\mathbb{N}$ the Laurent polynomial $\tilde{g}(x)=\sum_k \binom{ak+b}{i} g_kx^k$ is divisible by $(1+x)^{d-i}$.
\end{lem}

\begin{proof}
Consider the Laurent polynomial
\[ g'(x) = \frac{x^{i-b}}{i!}\frac{d^i}{dx^i}x^b g(x^a) = \frac{x^{i-b}}{i!}\frac{d^i}{dx^i} \sum_{k\in\mathbb{Z}} g_kx^{ak+b} = \binom{ak+b}{i}\sum_{k\in\mathbb{Z}} g_kx^{ak} \]
Let $g(x)$ be divisible by $(1+x)^d$ and write $g(x) = (1+x)^d r(x)$. By the product rule of differentiation we have
\[ g'(x) = (1+x^a)^{d-i} r'(x^a), \]
for some $r'(x)$. Hence, $g'(x)$ is divisible by $(1+x^a)^{d-i}$. Note that $\tilde{g}(x)=g'(x^{1/a})$. Then $\tilde{g}(x)=(1+x)^{d-i}r'(x)$ is divisible by $(1+x)^{d-i}$ as claimed.
\end{proof}

\begin{prop}
\label{prop:zeromut}
Let $f$ be a zero mutable Laurent polynomial with Newton polygon $\triangle$ that admits a sequence of mutations $f\mapsto\ldots\mapsto 1$ to the trivial Laurent polynomial.
\begin{enumerate}[(1)]
\item For each edge $e$ of $\triangle$, the required divisibility tuple $\mathbf{reqdiv}_e(f)=(d_k)_{k=k_{\textup{min}}}^{k_{\textup{max}}}$ is convex in the sense that for all $k$ we have $d_k-d_{k-1}\leq d_{k+1}-d_k$.
\item The required divisibility $\textup{reqdiv}(f)$ is equal to number of lattice points of $\triangle$.
\item The Laurent polynomial $f$ is determined by the collection of tuples $(\mathbf{reqdiv}_e(f))_e$.
\end{enumerate}
\end{prop}

\begin{proof}
Let $\tilde{f}$ be a zero mutable Laurent polynomial with Newton polygon $\tilde{\triangle}=\textup{Newt}(\tilde{f})$ and let $\tilde{f} \mapsto f \mapsto \ldots \mapsto 1$ be a sequence of mutations, such that $\tilde{f}=\textup{mut}_{(\varphi,h)}(f)$. Then $f$ admits a sequence of mutation $f\mapsto\ldots\mapsto 1$ and its Newton polytope $\triangle=\textup{Newt}(f)$ is either $1$-dimensional or $2$-dimensional. If $\triangle$ is $1$-dimensional, so that it consists of a single edge $e$, we define $\mathbf{reqdiv}_e(f)=(\ell(e),0)$. If $\triangle$ is $2$-dimensional, we assume that $f$ satisfies the assertions of the proposition. In both cases, we prove that $\tilde{f}$ satisfies the assertions. Then the proposition follows by induction. Let $\tilde{e}$ be an edge of $\tilde{\triangle}=\textup{Newt}(\tilde{f})$.

If $\varphi$ is not constant on $\tilde{e}$, then we see from Definition \ref{defi:mutve} that $\tilde{e}=\textup{mut}_{(\varphi,h)}(e)$ for an edge $e$ of $\triangle$ along which $\varphi$ is not constant. We show that $\mathbf{reqdiv}_{\tilde{e}}(\tilde{f})=\mathbf{reqdiv}_e(f)$. Write $\mathbf{reqdiv}_e(f)=(d_k)_{k=k_{\textup{min}}}^{k_{\textup{max}}}$. By a transformation in the integral affine transformation group $\textup{GL}_2(\mathbb{Z})\ltimes\mathbb{Z}^2$ we can assume that $e$ is horizontal and contains the origin. Write $f=\sum_k f_{k,l}x^ky^l$. Then $d_l$ is the multiplicity of $1+x$ in $f_l=\sum_k f_{k,l}x^k$. Since $h$ must be a zero mutable Laurent polynomial, we have $h=x^{ra}y^{rb}(1+x^ay^b)^s$ for some $a,b,r,s \in \mathbb{Z}$ with $\gcd(a,b)=1$. The first factor just acts as a linear transformation, we can set $r=0$. Moreover, we have $\varphi(m) = \braket{m,n_\varphi} + \varphi_0$, where $n_\varphi$ is an integral normal vector to $(a,b)$, i.e. $n_\varphi=(tb,-ta)$ for some $t\in\mathbb{Z}$. Write $n=st$. Then the mutation is given on the variables $f$ by
$\textup{mut}_{(\varphi,h)}(x) = x(1+x^ay^b)^{nb}$, $\textup{mut}_{(\varphi,h)}(y) = y(1+x^ay^b)^{-na}$, and we have
\begin{eqnarray*}
\tilde{f} &=& \textup{mut}_{(\varphi,h)}(f) = \sum_{k,l} f_{k,l}(1+x^ay^b)^{n(kb-la)} x^ky^l \\
&=& \sum_{k,l} f_{k,l} \sum_i  \binom{n(kb-la)}{i}  x^{k+ia}y^{l+ib} \\
&=& \sum_l\sum_i x^{ia}\sum_k f_{k,l-ib} \binom{n(kb-(l-ib)a)}{i} x^ky^l.
\end{eqnarray*}
Write $\tilde{f}=\sum_l \tilde{f}_l y^l$ with $\tilde{f}_l = \sum_i \tilde{f}_{l,i}$. Then $\tilde{f}_{l,0}=f_l$, which has divisibility $d_l$ with respect to $h=1+x$. By the assumption that $f$ fulfills the assertions of the proposition, $\mathbf{reqdivstep}_e(f)$ is decreasing. Hence, $f_{l-ib}$ has divisibility at least $d_l+ib$. Then, by Lemma \ref{lem:div}, $\tilde{f}_{l,i}$ has divisibility at least $d_l+ib-i \geq d_l$. The divisibility of $\tilde{f}_l$ is the minimum of these divisibilities, hence $\tilde{d}_l\geq d_l$. It follows that $\mathbf{div}_{\tilde{e}}(\tilde{f})\geq\mathbf{div}_e(f)$. If $\mathbf{div}_{\tilde{e}}(\tilde{f})>\mathbf{div}_e(f)$, then the summation that gives $\tilde{f}$ is special. In this case, the conditions that $\mathbf{div}_{\tilde{e}}(\tilde{f})$ imposes on the coefficients of $\tilde{f}$ are linearly dependent, and we still have $\mathbf{reqdiv}_{\tilde{e}}(\tilde{f})=\mathbf{reqdiv}_e(f)$.

If $\varphi$ is constant on $\tilde{e}$, let $\triangle_k$, $k_{\textup{min}}$ and $k_{\textup{max}}$ be as in Definition \ref{defi:div} and let $\varphi_0$ be the linear part of $\varphi$. From Definition \ref{defi:mutve} we see that either $\tilde{e}=\textup{mut}_{(\varphi,h)}(e)$ with $\varphi$ constant on $e$ or $\tilde{e}=\textup{mut}_{(\varphi,h)}(v)$ for a vertex $v$ of $\triangle$. In the former case write $\mathbf{reqdiv}_e(f)=(d_k)_{k=k_{\textup{min}}}^{k_{\textup{max}}}$, in the latter case define $d_k=0$ for all $k=k_{\textup{min}},\ldots,k_{\textup{max}}$. Let $m_e$ be a primitive tangent vector of $e$ and note that $\varphi_0(m_e)=0$. By Definition \ref{defi:zeromut}, $h=(1+z^{m_e})^k z^{lm_e}$ for some $k\in\mathbb{N}$ and $l\in\mathbb{Z}$. We can assume $k=1$ by increasing the slope of $\varphi$. Then by the definition of mutation we have $\mathbf{reqdiv}_{\tilde{e}}(\tilde{f})=(\tilde{d}_k)_{k=k_{\textup{min}}}^{k_{\textup{max}}}$ with
\[ \tilde{d}_k = d_k+\varphi(\triangle_k). \]
Now we show that $\tilde{f}$ fulfills the assertions of the proposition.
\begin{enumerate}[(1)]
\item By assumption the tuples $\mathbf{reqdiv}_{e}(f)$ for edges $e$ of $\triangle$ are convex. Let $\tilde{e}$ be an edge of $\tilde{\triangle}$. If $\varphi$ is not constant along $\tilde{e}$, then we have shown above that $\mathbf{reqdiv}_{\tilde{e}}(\tilde{f})=\mathbf{reqdiv}_{e}(f)$. In particular, this is convex. If $\varphi$ is constant along $\tilde{e}$, then it follows from the formula for $\tilde{d}_k$ above that $\mathbf{reqdiv}_{\tilde{e}}(\tilde{f})$ is convex.
\item First consider the case when $\triangle$ is $2$-dimensional. Then by assumption the number of lattice points of $\triangle$ is $\textup{reqdiv}(f) = \sum_e \textup{reqdiv}_{e}(f)$, where for parallel edges we consider only one of them in the sum, see Definition \ref{defi:reqdiv}. Let $\tilde{e}$ be an edge of $\tilde{\triangle}$. If $\varphi$ is not constant along $\tilde{e}$, then $\mathbf{reqdiv}_{\tilde{e}}(\tilde{f})=\mathbf{reqdiv}_{e}(f)$, hence $\textup{reqdiv}_{\tilde{e}}(\tilde{f})=\textup{reqdiv}_{e}(f)$. If $\varphi$ is constant along $e$, then it follows from the formula for $\tilde{d}_k$ above that 
\[ \textup{reqdiv}_{\tilde{e}}(\tilde{f}) = \textup{reqdiv}_e(f) + \sum_{k=k_{\textup{min}}}^{k_{\textup{max}}} \varphi(\triangle_k). \]
There are $1$ or $2$ such edges in $\tilde{\triangle}$ and if there are $2$, then they must be parallel. In turn, the second term above only appears once in $\textup{reqdiv}(\tilde{f}) = \sum_{\tilde{e}} \textup{reqdiv}_{\tilde{e}}(\tilde{f})$. The second term is the difference between the number of lattice points of $\tilde{\triangle}$ and the number of lattice points of $\triangle$. By assumption, the first term equals the number of lattice points of $\triangle$. Hence, $\textup{reqdiv}(\tilde{f})$ equals the number of lattice points of $\tilde{\triangle}$. If $\triangle$ is $1$-dimensional, then $\textup{reqdiv}_e(f)$ is one less than the number of lattice points of $\triangle$. But we also produce an additional edge along which $\varphi$ is not constant.
\item Let $f$ be a zero mutable Laurent polynomial with Newton polygon $\triangle$. The coefficients of $f$ correspond to lattice point of $\triangle$. By Corollary \ref{cor:binomial}, the coefficients corresponding lattice points on the boundary of $\triangle$ are determined to be certain binomial coefficients. Up to a global factor this is also determined by the first element of $\mathbf{reqdiv}_e(f)$, which is $\ell(e)$. The other elements of $\mathbf{reqdiv}_e(f)$ give linear relations among the coefficients of $f$. The number of relations is equal to $\sum_e (\textup{reqdiv}_e(f)-\ell(e))$. By part (2), this number is equal to the interior lattice points of $\triangle$. By definition of $\mathbf{reqdiv}_e(f)$, the conditions are linearly independent, hence they determine the coefficients of $f$ corresponding to interior lattice points of $\triangle$. In total, this determines all coefficients of $f$. \qedhere
\end{enumerate}
\end{proof}

\begin{cor}
\label{cor:finite}
Let $\triangle$ be a lattice polygon. The number of zero mutable Laurent polynomials supported on $\triangle$ is finite.
\end{cor}

\begin{proof}
If $f$ is an irreducible zero mutable Laurent polynomial, then it admits a sequence of mutations $f\mapsto\ldots\mapsto 1$ to the trivial Laurent polynomial. By Proposition \ref{prop:zeromut}, (3), $f$ is determined by $(\mathbf{reqdiv}_e(f))_e$. By Proposition \ref{prop:zeromut}, (2), $\textup{reqdiv}(f)=\sum_e \textup{reqdiv}_e(f)$ is equal to the number of lattice points of $\triangle=\textup{Newt}(f)$. There are only finitely many choices of convex tuples $\textup{reqdiv}_e(f)$ that sum up to this finite number. Hence, the number of irreducible zero mutable Laurent polynomials supported on a lattice polygon $\triangle$ is finite. Reducible zero mutable Laurent polynomial factor as a product of finitely many irreducible zero mutable Laurent polynomials. Hence, their number is finite as well.
\end{proof}

\begin{prop}
\label{prop:divstep}
Let $f$ be a zero mutable Laurent polynomial with Newton polygon $\triangle$.
\begin{enumerate}[(1)]
\item For each edge $e$ of $\triangle$, the divisibility steps $\mathbf{divstep}_e(f)$ form a decreasing partition of the lattice length $\ell(e)$.
\item The Laurent polynomial $f$ is determined by the collection of partitions $\mathbf{divstep}_e(f)$.
\end{enumerate}
\end{prop}

\begin{proof}
Let $f$ be a zero mutable Laurent polynomial with Newton polygon $\triangle$. If $f$ is irreducible, write $\mathbf{reqdiv}_e(f)=(d_k)_{k=k_{\textup{min}}}^{k_{\textup{max}}}$. We have $\min(\{d_k\})=0$, since otherwise $f$ would be reducible. Write $\overline{k}_{\text{max}}=\min\{k \ | \ d_k=0\}$. By definition $\mathbf{divstep}_e(f)=(d_k-d_{k+1})_{k=k_{\textup{min}}}^{\overline{k}_{\textup{max}}-1}$. By part (1), $\mathbf{reqdiv}_e(f)$ is convex, hence $\mathbf{divstep}_e(f)$ is a decreasing tuple. By definition, the last element of $\mathbf{divstep}_e(f)$ is $d_{\overline{k}_{\text{max}}}>0$. Hence, all elements of $\mathbf{divstep}_e(f)$ are positive. They sum up to $d_{k_{\text{min}}}-d_{\overline{k}_{\text{max}}}=\ell(e)-0=\ell(e)$. Hence $\mathbf{divstep}_e(f)$ gives a decreasing partition of $\ell(e)$. If $f$ is irreducible, then by definition $\mathbf{divstep}_e(f)$ is a decreasing tuple of $\ell(e_1)+\ldots+\ell(e_r)=\ell(e)$. This proves (1). For (2), note that $\mathbf{divstep}_e(f)$ contains the same information as $\mathbf{reqdiv}_e(f)$. Then the statement follows from Proposition \ref{prop:zeromut}, (3).
\end{proof}

\begin{lem}
\label{lem:sumsum}
Let $f$ be irreducible and write $\mathbf{divstep}_e(f)=(s_k)_{l=1}^w$. Then $\mathbf{reqdiv}_e(f)=(d_k)_{k=1}^{w+1}$ with $d_k=\sum_{l=k}^w s_l$, in particular $d_{w+1}=0$, and $\textup{reqdiv}_e(f)=\sum_{k=1}^w\sum_{l=k}^w s_l$.
\end{lem}

\begin{proof}
This easily follows from the definitions.
\end{proof}

In view of Proposition \ref{prop:zeromut}, (2), Proposition \ref{prop:divstep} and Lemma \ref{lem:sumsum}, we are lead to the following conjecture. Some evidence will be given in \S\ref{S:classification}.

\begin{conj}
\label{conj:zmlp}
Irreducible zero mutable Laurent polynomials on $\triangle$ are in one-to-one correspondence with tuples $(\mathbf{p}_e(f))_e$ of decreasing partitions $\mathbf{p}_e(f)=(s_{e,l})_{l=1}^w$ of $\ell(e)$ such that $\sum_e \sum_{k=1}^w\sum_{l=k}^w s_{e,l}$ equals the number of lattice points of $\triangle$.
\end{conj}

\begin{defi}
\label{defi:dualtuple}
For a zero mutable $f$ we will call the tuple of partitions dual to the divisibility steps
\[ (\mathbf{a}_e)_e=(\mathbf{divstep}_e(f)^\vee)_e \]
the \emph{tuple of dual partitions} or simply the \emph{dual tuple}.
\end{defi}

% todo conj comb

\subsection{ZMLPs on rectangular triangles}											%%
\label{S:rect}

\begin{defi}
A lattice polygon $\triangle$ is a \emph{rectangular triangle} if it can be mapped to the standard triangle $\triangle(a,b)=\textup{Conv}\{(0,0),(a,0),(0,b)\}$ for some $a,b\ge1$ under a transformation in the integral affine transformation group $\textup{GL}_2(\mathbb{Z})\ltimes\mathbb{Z}^2$. Equivalently, $\triangle$ is a lattice triangle with a vertex whose primitive integral outgoing edge direction vectors have determinant $1$. We call a rectangular triangle $\triangle(a,b)$ \emph{primitive} if $\gcd(a,b)=1$.
\end{defi}

\begin{defi}
\label{defi:dualpair}
For $f \in \textup{ZMLP}(a,b)$ we omit the edge of length $1$ in the dual tuple (Definition \ref{defi:dualtuple}) and call $(\mathbf{a},\mathbf{b})$ the \emph{pair of dual partitions} or just the \emph{dual pair} of $f$.
\end{defi}

\begin{expl}
\label{expl:tomjerrydiv}
For Tom and Jerry from Example \ref{expl:tomjerrymut}, the dual pair is $(\mathbf{a},\mathbf{b})=((2,1),(1,1))$ and $(\mathbf{a},\mathbf{b})=((1,1,1),(2))$, respectively.
\end{expl}

\begin{lem}
\label{lem:exists}
If $f \in \textup{ZMLP}(a,b)$, then the dual pair $(\mathbf{a},\mathbf{b})$ satisfies
\[ \max(\mathbf{a}) \leq b, \quad \max(\mathbf{b}) \leq a, \quad \max(\mathbf{a})+\max(\mathbf{b}) \leq \max(a,b). \]
\end{lem}

\begin{proof}
Let $e_a$ be the edge of length $a$ in $\triangle=\textup{Newt}(f)$. Then $\mathbf{a}$ is dual to $\mathbf{divstep}_{e_a}(f)$, hence $\max(\mathbf{a})=\ell(\mathbf{divstep}_{e_a}(f))$, the length of the partition $\mathbf{divstep}_{e_a}(f)$. The size of $\triangle$ orthogonal to $e_a$ is equal to $b$. Hence, we have $\ell(\mathbf{divstep}_{e_a}(f))\leq b$. It follows that $\max(\mathbf{a})\leq b$. By symmetry we also get $\max(\mathbf{b})\leq a$. 

For the last inequality, assume on the contrary that we have $\max(\mathbf{a})+\max(\mathbf{b}) > \max(a,b)$. Without loss of generality $\max(\mathbf{a})\geq\max(\mathbf{b})$. Then $\max(\mathbf{a}) > \max(a,b)/2 \geq b/2$. Consider the mutation $\textup{mut}_{(\varphi,h)}$ for $\varphi=\braket{(1,0),\cdot}-\max(\mathbf{a})$ and $h=1+y$. Note that $\varphi$ is not constant along $\tilde{e}_b$. Hence $\mathbf{divstep}_{\tilde{e}_b}(\tilde{f})=\mathbf{divstep}_{e_b}(f)$, as we have shown in the proof of Proposition \ref{prop:zeromut}. The polygon $\textup{mut}_{(\varphi,h)}(\triangle(a,b))$ has height (vertical size) equal to $b-\max(\mathbf{a})$. By the assumption this is smaller than $\max(\mathbf{b})=\ell(\mathbf{divstep}_{e_b}(f))=\ell(\mathbf{divstep}_{\tilde{e}_b}(\tilde{f}))$. Hence, there is a horizontal slice of $\textup{mut}_{(\varphi,h)}(f)$ with length $0$ but divisibility $>0$, which is a contradiction. An example of this phenomenon for $(\mathbf{a},\mathbf{b})=((3,1),(3,1,1))$ is shown in Figure \ref{fig:contradiction}.
\end{proof}

\begin{figure}[h!]
\centering
\begin{tikzpicture}[scale=.8]
\draw[black!40!green] (6,0) -- (-1,0) node[left]{$5$};
\draw[black!40!green] (6,1) -- (-1,1) node[left]{$2$};
\draw[black!40!green] (6,2) -- (-1,2) node[left]{$1$};
\draw[black!40!green] (6,3) -- (-1,3) node[left]{$0$};
\draw[black!40!green] (6,4) -- (-1,4) node[left]{$0$};
\draw[black!40!green] (0,4.5) -- (0,-1) node[below]{$4$};
\draw[black!40!green] (1,4.5) -- (1,-1) node[below]{$2$};
\draw[black!40!green] (2,4.5) -- (2,-1) node[below]{$1$};
\draw[black!40!green] (3,4.5) -- (3,-1) node[below]{$0$};
\draw[black!40!green] (4,4.5) -- (4,-1) node[below]{$0$};
\draw[black!40!green] (5,4.5) -- (5,-1) node[below]{$0$};
\draw (0,0) -- (5,0) -- (0,4) -- cycle;
\draw[->] (7,2) -- (8,2);
\draw[black!40!green] (16,0) -- (9,0) node[left]{$5$};
\draw[black!40!green] (16,1) -- (9,1) node[left]{$2$};
\draw[black!40!green] (16,2) -- (9,2) node[left]{$1$};
\draw[black!40!red] (9,2) node[right]{\Large$\lightning$};
\draw[black!40!green] (10,3) -- (10,-1) node[below]{$1$};
\draw[black!40!green] (11,3) -- (11,-1) node[below]{$0$};
\draw[black!40!green] (12,3) -- (12,-1) node[below]{$0$};
\draw[black!40!green] (13,3) -- (13,-1) node[below]{$0$};
\draw[black!40!green] (14,3) -- (14,-1) node[below]{$1$};
\draw[black!40!green] (15,3) -- (15,-1) node[below]{$2$};
\draw (10,0) -- (15,0) -- (15,2) -- (10,1) -- cycle;
\end{tikzpicture}
\caption{Mutation giving a contradiction to $\max(\mathbf{a})+\max(\mathbf{b}) > \max(a,b)$. The required divisibilities are shown in green. There is no Laurent polynomial with these divisibilities, hence we cannot give the coefficients.}
\label{fig:contradiction}
\end{figure}
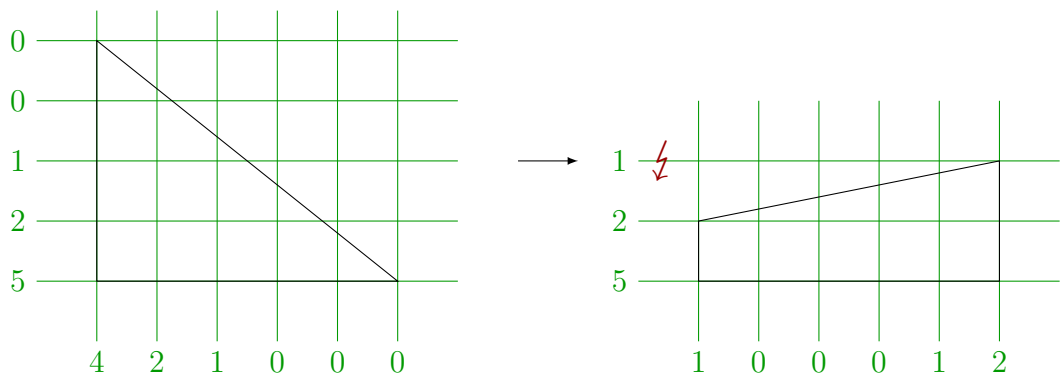

\begin{lem}
\label{lem:divint}
If $f\in\textup{ZMLP}(a,b)$ with $\gcd(a,b)=1$, then the dual pair $(\mathbf{a},\mathbf{b})=((a_i),(b_j))$ satisfies
\[ \sum a_i^2 + \sum b_j^2 = ab+1. \]
\end{lem}

\begin{proof}
Let $e_a$ be the edge in $\triangle(a,b)$ of length $a$. We have $\mathbf{divstep}_{e_a}(f)=\mathbf{a}^\vee=(s_l)_{l=1}^w$, where $s_l=|\{i \ | \ a_i\geq l\}|$ and $w=\max(\{a_i\})$. By Lemma \ref{lem:sumsum}, $\mathbf{reqdiv}_{e_a}(f)=(d_k)_{k=1}^{w+1}$, where $d_k=\sum_{l=k}^w s_l$ for $k\leq w$ and $d_{w+1}=0$. Then 
\begin{eqnarray*}
\textup{reqdiv}_{e_a}(f) &=& \sum_{k=1}^w d_k = \sum_{k=1}^w\sum_{l=k}^w |\{i \ | \ a_i\geq l\}|=\sum_{l=1}^w\sum_{k=1}^l |\{i \ | \ a_i\geq l\}| =\sum_{l=1}^w l|\{i \ | \ a_i\geq l\}| \\
&=& \sum_i \sum_{l=1}^{a_i} l = \sum_i \frac{a_i^2+a_i}{2} = \frac{1}{2}\left(\sum_i a_i^2 + a\right).
\end{eqnarray*}
Similarly for the edge of length $b$ and the edge of length $1$. In total we have
\[ \textup{reqdiv}(f) = \frac{1}{2}\left(\sum_i a_i^2 + \sum_j b_j^2 + a+b+2\right). \]
The number of lattice points $\triangle(a,b)$ is 
\[ \frac{1}{2}(ab+a+b+3). \]
This can for instance be seen from Pick's theorem, since the area of $\triangle(a,b)$ is $\frac{1}{2}ab$ and the number of boundary lattice points is $a+b$. By Proposition \ref{prop:zeromut}, (2), the number of lattice points is equal to $\textup{reqdiv}(f)$. Equating the above expressions gives the claimed formula.
\end{proof}

\begin{defi}
\label{defi:comb}
Let $\textup{ZMLP}_{\textup{comb}}(a,b)$ be the set of pairs $(\mathbf{a},\mathbf{b})$, where $\mathbf{a}$ is a decreasing partition of $a$ and $\mathbf{b}$ is a decreasing partition of $b$, such that
\[ \sum a_i^2 + \sum b_i^2 = ab+1, \quad \max(\mathbf{a}) \leq b, \quad \max(\mathbf{b}) \leq a, \quad \max(\mathbf{a})+\max(\mathbf{b}) \leq \max(a,b). \]
\end{defi}

\begin{prop}
\label{prop:contained}
If $\gcd(a,b)=1$, then $\textup{ZMLP}(a,b) \subset \textup{ZMLP}_{\textup{comb}}(a,b)$.
\end{prop}

\begin{proof}
This is the content of Lemmas \ref{lem:exists} and \ref{lem:divint}.
\end{proof}

\begin{conj}
\label{conj:comb}
If $\gcd(a,b)=1$, then $\textup{ZMLP}(a,b)=\textup{ZMLP}_{\textup{comb}}(a,b)$.
\end{conj}

This is Conjecture \ref{conj:zmlp} for rectangular triangles. In \S\ref{S:classification} we provide some evidence for Conjecture \ref{conj:comb}.

\subsection{Triangular mutations}													%%

\begin{defi}
Let $\textup{ZMLP}_\triangle(a,b)$ be the set of ZMLPs $f$ on $\triangle(a,b)$ for which there exists a sequence of mutations $f \mapsto f_k \mapsto f_{k-1} \mapsto \ldots \mapsto f_1 \mapsto f_0 = 1$ such that in each step $f_i$ is a ZMLP on a rectangular triangle, a line segment or a point.
\end{defi}

\begin{defi}
\label{defi:alphabeta}
For a Laurent polynomial $f$ on a rectangular triangle $\triangle(a,b)$ define the transposition
\[ \tau (f(x,y)) = f(y,x) \]
and the \emph{elementary triangular mutations}, see Figure \ref{fig:alphabeta},
\begin{align*}
 \alpha &= \textup{mut}_{\varphi=\braket{(0,1),\cdot}+a,h=1+y} \\
\beta &= \textup{mut}_{\varphi=\braket{(0,-1),\cdot}-b,h=1+y}
\end{align*}
Note that $\beta^{-1}=\beta$ and that $\alpha^{-1}(f)$ only exists for certain $f$.
\end{defi}

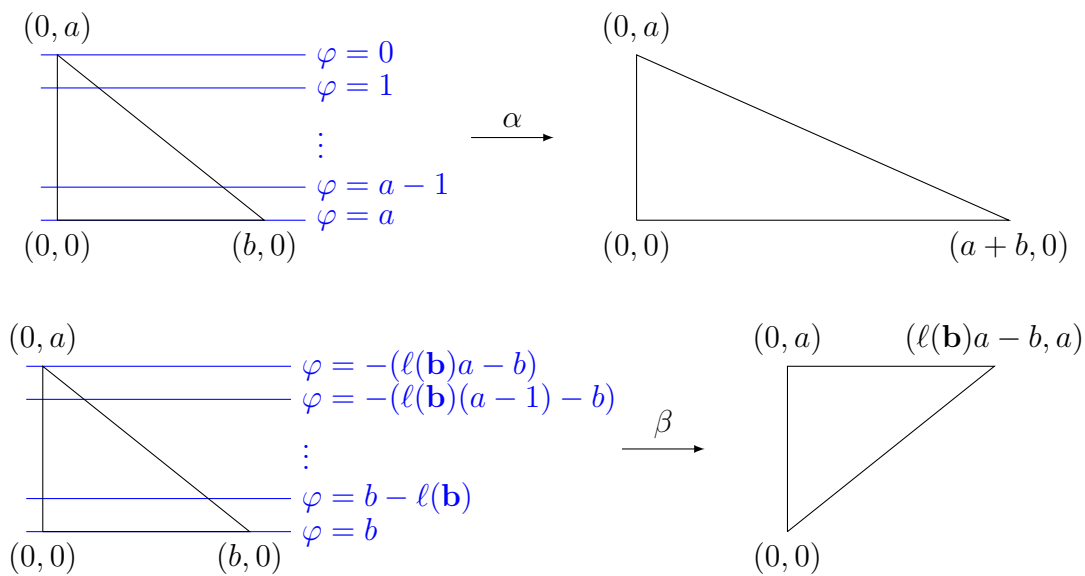
\begin{figure}[h!]
\centering
\begin{tikzpicture}[scale=1.1]
\draw[blue] (-.2,0) -- (3,0) node[right]{$\varphi=a$};
\draw[blue] (-.2,.4) -- (3,.4) node[right]{$\varphi=a-1$};
\draw[blue] (3,1) node[right]{$\vdots$};
\draw[blue] (-.2,1.6) -- (3,1.6) node[right]{$\varphi=1$};
\draw[blue] (-.2,2) -- (3,2) node[right]{$\varphi=0$};
\draw (0,0) node[below]{$(0,0)$} -- (2.5,0) node[below]{$(b,0)$} -- (0,2) node[above]{$(0,a)$} -- (0,0);
\draw[->] (5,1) -- (5.5,1) node[above]{$\alpha$} -- (6,1);
\draw (7+0,0) node[below]{$(0,0)$} -- (7+4.5,0) node[below]{$(a+b,0)$} -- (7+0,2) node[above]{$(0,a)$} -- (7+0,0);
\end{tikzpicture}
\textup{ }\\[5mm]
\begin{tikzpicture}[scale=1.1]
\draw[blue] (-.2,0) -- (3,0) node[right]{$\varphi=b$};
\draw[blue] (-.2,.4) -- (3,.4) node[right]{$\varphi=b-\ell(\mathbf{b})$};
\draw[blue] (3,1) node[right]{$\vdots$};
\draw[blue] (-.2,1.6) -- (3,1.6) node[right]{$\varphi=-(\ell(\mathbf{b})(a-1)-b)$};
\draw[blue] (-.2,2) -- (3,2) node[right]{$\varphi=-(\ell(\mathbf{b})a-b)$};
\draw (0,0) node[below]{$(0,0)$} -- (2.5,0) node[below]{$(b,0)$} -- (0,2) node[above]{$(0,a)$} -- (0,0);
\draw[->] (7,1) -- (7.5,1) node[above]{$\beta$} -- (8,1);
\draw (9+0,0) node[below]{$(0,0)$} -- (9+2.5,2) node[above]{$(\ell(\mathbf{b})a-b,a)$} -- (9+0,2) node[above]{$(0,a)$} -- (9+0,0);
\end{tikzpicture}
\caption{The elementary triangular mutations $\alpha$ and $\beta$.}
\label{fig:alphabeta}
\end{figure}

\begin{prop}
The set $\textup{ZMLP}_\triangle$ is generated from $f=1$ by repeated applications of $\tau$, $\alpha$, $\beta$ and $\alpha^{-1}$.
\end{prop}

\begin{proof}
Let $f$ be a Laurent polynomial on a rectangular triangle $\triangle(a,b)$. All mutations that transform $f$ into a Laurent polynomial on another rectangular triangle can be written as a composition of $\tau$, $\alpha$, $\beta$ and $\alpha^{-1}$.
% todo
\end{proof}

\begin{prop}
\label{prop:degree}
The action of $\tau,\alpha,\beta$ on the dual pair $(\mathbf{a},\mathbf{b})$ with $a=\sum a_i$, $b=\sum b_i$, is given by
\begin{align*}
\alpha((a_1,\ldots,a_k),(b_1,\ldots,b_l)) &= (a_1,\ldots,a_k),(a,b_1,\ldots,b_l) \\
\beta((a_1,\ldots,a_k),(b_1,\ldots,b_l)) &= (a_1,\ldots,a_k),(a-b_1,\ldots,a-b_l) \\
\tau((a_1,\ldots,a_k),(b_1,\ldots,b_l)) &= (b_1,\ldots,b_l),(a_1,\ldots,a_k)
\end{align*}
and the action on the degree (sum of partition elements) is
\begin{align*}
|\alpha(\mathbf{a},\mathbf{b})| &= (a,a+b), \\
|\beta(\mathbf{a},\mathbf{b})| &= (a,\ell(\mathbf{b})a-b),
\end{align*}
where $|\mathbf{a}|:=\sum a_i$ is the degree of $\mathbf{a}$ and $\ell(\mathbf{a})$ is the length (number of elements) of $\mathbf{a}$.
\end{prop}

\begin{proof}
This is clear from the definition.
\end{proof}

\begin{rem}
Note that by Lemma \ref{lem:exists} we have $\max(\mathbf{b})\leq a$, so $a$ is really the first element of the decreasing tuple $(a,b_1,\ldots,b_l)$ in $\alpha((a_1,\ldots,a_k),(b_1,\ldots,b_l))$.
\end{rem}

\subsection{Classification results: Tom, Jerry, Spike, Tyke and the like}								%%
\label{S:classification}

Here we compute $\textup{ZMLP}(a,b)$ for $a \leq 3$, $\textup{gcd}(a,b)=1$ and for $a \geq 1$, $b=a+1$ and show that in these cases and for small triangles with $a+b \leq 11$ we have $\textup{ZMLP}(a,b)=\textup{ZMLP}_\triangle(a,b)=\textup{ZMLP}_{\textup{comb}}(1,k)$. The results are summarized in Theorem \ref{thm:classification} and Table \ref{tab:classification}.

\begin{table}[h!]
\bgroup
\def\arraystretch{1.5}
\begin{tabular}{c|c|c|c|c}
& $\triangle(1,k)$ & $\triangle(2,k)$ & $\triangle(3,k)$ & $\triangle(k,k+1)$  \\ \hline
Tom & $(1),(\vec{1})$ & $(1,1),(\vec{2},1)$ & $\begin{matrix} (1,1,1),(\vec{3},1), & k \equiv_3 1 \\ (1,1,1),(\vec{3},2), & k \equiv_3 2 \end{matrix}$ & $(1,\ldots,1),(k,1)$ \\ \hline
Jerry & & $(2),(\vec{2},1,1,1)$ & $\begin{matrix} (3),(\vec{3},1,1,1,1), & k \equiv_3 1 \\ (3),(\vec{3},2,2,2,2), & k \equiv_3 2 \end{matrix}$ & $(k),(1,\ldots,1)$ \\ \hline
Spike & & & $\begin{matrix} (2,1),(\vec{3},2,2), & k \equiv_3 1 \\ (2,1),(\vec{3},1,1), & k \equiv_3 2 \end{matrix}$ & $\begin{matrix} (\tfrac{k}{2},\tfrac{k}{2}),(\tfrac{k}{2}+1,\tfrac{k}{2}), & k \textup{ even} \\ (\tfrac{k+1}{2},\tfrac{k-1}{2}),(\tfrac{k+1}{2},\tfrac{k+1}{2}), & k \textup{ odd} \end{matrix}$ \\ \hline
Tyke & & & $\begin{matrix} (3),(\vec{3},2,2,2,1), & k \equiv_3 1 \\ (3),(\vec{3},2,1,1,1), & k \equiv_3 2 \end{matrix}$ &
\end{tabular}
\egroup
\vspace{5mm}
\caption{Pairs of dual partitions $(\mathbf{a},\mathbf{b})$ for ZMLPs on rectangular triangles.}
\vspace{-5mm}
\label{tab:classification}
\end{table}

\begin{prop}
\label{prop:a}
For $k\geq 1$, the set $\textup{ZMLP}(1,k)$ consists of $1$ element with dual pair $\mathbf{a}=(1),\mathbf{b}=(1,\ldots1)$. Moreover, we have $\textup{ZMLP}(1,k)=\textup{ZMLP}_\triangle(1,k)=\textup{ZMLP}_{\textup{comb}}(1,k)$.
\end{prop}

\begin{proof}
The triangle $\triangle(1,k)$ has no interior lattice points. Hence, a ZMLP on $\triangle(1,k)$ is uniquely determined by Corollary \ref{cor:binomial}. This gives $f = (1+x)^k + y$, which indeed is a ZMLP by the mutations
\[ (1+x)^k + y \xrightarrow{\beta=\alpha^{-b}} 1+y \longrightarrow 1. \]
The pair of dual partitions is $(1),(1,\ldots,1)$. One easily sees that this is an element of $\textup{ZMLP}_\triangle(1,k)$ and $\textup{ZMLP}_{\textup{comb}}(1,k)$.
\end{proof}

\begin{prop}
\label{prop:b}
For $k\geq 3$ and $k\notin 2\mathbb{Z}$, the set $\textup{ZMLP}(2,k)$ consists of $2$ elements, with pair of dual partitions as in Table \ref{tab:classification}. We call them Tom and Jerry. Moreover, we have $\textup{ZMLP}(2,k)=\textup{ZMLP}_\triangle(2,k)=\textup{ZMLP}_{\textup{comb}}(2,k)$.
\end{prop}

\begin{proof}
By Corollary \ref{cor:binomial} we know that $f$ restricted to the horizontal edge of length $k$ is given by $(1+x)^k$. Consequently, $\mathbf{reqdiv}_{e_b}(f)=(k,d_1,0)$, leaving $d_1$ as the only unknown parameter. There are two possible options for $\mathbf{a}=\mathbf{reqdiv}_{e_a}(f)^\vee$. 

If $\mathbf{a}=(1,1)$, then $\mathbf{reqdiv}_{e_a}(f)=(2,0,\ldots,0)$ and Proposition \ref{prop:zeromut}, (2), gives $d_1=(k-1)/2$. Therefore, $\mathbf{b}=(2,\ldots,2,1)$. This configuration corresponds to Tom, as seen in Table \ref{tab:classification}. The unique Laurent polynomial that satisfies these divisibilities is
\[ f = (1+x)^k + 2y(1+x)^{(k-1)/2} + y^2. \]
Applying the mutation $\beta=\alpha^{-(k-1)/2}$ leads to a Laurent polynomial with pair of dual partitions given by $(1,1),(1)$. This is zero mutable, as confirmed by Proposition \ref{prop:a}. In particular, $f\in\textup{ZMLP}_\triangle(2,k)$, and one easily confirms that $(\mathbf{a},\mathbf{b})\in\textup{ZMLP}_{\textup{comb}}(2,k)$.

If $\mathbf{a}=(2)$, then $\mathbf{reqdiv}_{e_a}(f)=(1,1,0,\ldots,0)$ and Proposition \ref{prop:zeromut}, (2), gives $d_1+1=(k-1)/2$, hence $d_1=(k-3)/2$. This is Jerry from Table \ref{tab:classification}. The unique Laurent polynomial with these divisibilities is
\[ f = (1+x)^k + (2+kx)(1+x)^{(k-3)/2} + y^2. \]
The mutation $\beta=\alpha^{-(k-3)/2}$ leads to a Laurent polynomial with dual pair $(2),(1,1,1)$, and then $\tau\beta\tau$ leads to $(1),(1,1,1)$. The corresponding Laurent polynomial is zero mutable by Proposition \ref{prop:a}. Again, $f\in\textup{ZMLP}_\triangle(2,k)$ and $(\mathbf{a},\mathbf{b})\in\textup{ZMLP}_{\textup{comb}}(2,k)$.
\end{proof}

\begin{prop}
\label{prop:c}
For $k\geq 7$, the set $\textup{ZMLP}(3,k)$ consists of $4$ elements, with pair of dual partitions as in Table \ref{tab:classification}. We call them Tom, Jerry, Spike and Tyke. Moreover, we have $\textup{ZMLP}(3,k)=\textup{ZMLP}_\triangle(3,k)=\textup{ZMLP}_{\textup{comb}}(3,k)$.
\end{prop}

\begin{rem}
Spike and Tyke are dogs appearing regularly in the TV show Tom and Jerry. For $k=5$ there is no Jerry, so we only have $3$ ZMLPs.
\end{rem}

\begin{proof}
We have $\mathbf{reqdiv}_{e_b}(f)=(k,d_1,d_2,0)$ with two unknowns $d_1$ and $d_2$. By Proposition \ref{prop:zeromut}, (1), we have $k-d_1 \geq d_1-d_2 \geq d_2-0$, hence $2d_1-l \leq d_2 \leq d_1/2$. By Proposition \ref{prop:zeromut}, (2), we have $d_1+d_2=k-1-\textup{reqdiv}_{e_a}(f)$, hence $3d_1-k \leq k-1-\textup{reqdiv}_{e_a}(f) \leq 3d_1/2$ and
\[ \frac{2(k-1-\textup{reqdiv}_{e_a}(f))}{3} \leq d_1 \leq \frac{2k-1-\textup{reqdiv}_{e_a}(f)}{3}. \]
We have three options for $\mathbf{a}=\mathbf{divstep}_{e_a}(f)^\vee$.

If $\mathbf{a}=(1,1,1)$, then $\mathbf{reqdiv}_{e_a}(f)=(3,0,\ldots,0)$ and $\textup{reqdiv}_{e_a}(f)=0$. The above inequality gives $d_1=(2k-2)/3$ if $k \equiv_3 1$ and $d_1=(2k-1)/3$ if $k\equiv_3 2$. From $d_1+d_2=k-1$ we get $d_2=(k-1)/3$ if $k \equiv_3 1$ and $d_2=(k-2)/3$ if $k \equiv_3 2$. The dual partition of $\mathbf{reqdiv}_{e_b}(f)=(k,d_1,d_2,0)$ is $\mathbf{b}=(3,\ldots,3,1)$ if $a \equiv_3 1$ and $\mathbf{b}=(3,\ldots,3,2)$ if $k \equiv_3 2$. This is Tom from Table \ref{tab:classification}. The corresponding Laurent polynomial is indeed zero mutable, since the mutation $\beta=\alpha^{(k-1)/3}$ resp. $\beta=\alpha^{(k-2)/3}$ leads to $(1,1,1),(1)$ resp. $(1,1,1),(2)$, which is zero mutable by Proposition \ref{prop:a} resp. Jerry from Proposition \ref{prop:b}.

If $\mathbf{a}=(2,1)$, then $\mathbf{reqdiv}_{e_a}(f)=(3,1,0,\ldots,0)$ and $\textup{reqdiv}_{e_a}(f)=1$. The above inequality gives $d_1=(2k-2)/3$ if $k \equiv_3 1$ and $d_1=(2k-4)/3$ if $k\equiv_3 2$. From $d_1+d_2=k-2$ we get $d_2=(k-4)/3$ if $k \equiv_3 1$ and $d_2=(k-2)/3$ if $k \equiv_3 2$. The dual partition of $\mathbf{reqdiv}_{e_b}(f)=(k,d_1,d_2,0)$ is $\mathbf{b}=(3,\ldots,3,2,2)$ if $a \equiv_3 1$ and $\mathbf{b}=(3,\ldots,3,1,1)$ if $k \equiv_3 2$. This is Spike from Table \ref{tab:classification}. The corresponding Laurent polynomial is indeed zero mutable, since the mutation $\beta=\alpha^{(k-1)/3}$ resp. $\beta=\alpha^{(k-2)/3}$ leads to $(2,1),(2,2)$ resp. $(2,1),(1,1)$. In the first case we apply $\beta$ again to obtain $(2,1),(1,1)$ as well. The Laurent polynomial with dual pair $(2,1),(1,1)$ is zero mutable by Proposition \ref{prop:b}.

If $\mathbf{a}=(3)$, then $\mathbf{reqdiv}_{e_a}(f)=(3,2,1,0,\ldots,0)$ and $\textup{reqdiv}_{e_a}(f)=3$. Hence, we have $d_1+d_2=k-4$ with $(2k-8)/3 \leq d_1 \leq (2k-4)/3$. For $k \equiv_3 1$ we have the two possibilities $d_1=(2k-5)/3,d_2=(k-7)/3$ and $d_1=(2k-8)/3,d_2=(k-4)/3$. They correspond to $\mathbf{b}=(3,\ldots,3,2,2,2,1)$ and $\mathbf{b}=(3,\ldots,3,1,1,1,1)$, respectively, hence to Tyke and Jerry from Table \ref{tab:classification}. Similarly, for $k \equiv_3 2$ we have the two possibilities $d_1=(2k-4)/3,d_2=(k-8)/3$ and $d_1=(2k-7)/3,d_2=(k-5)/3$. They correspond to $\mathbf{b}=(3,\ldots,3,2,2,2,2)$ and $\mathbf{b}=(3,\ldots,3,2,1,1,1)$, respectively, hence to Jerry and Tyke from Table \ref{tab:classification}. By applying $\beta$ repeatedly we can reduce all cases to either $(3),(1,1,1,1)$ or $(3),(2,1,1,1)$. If we now apply $\tau\beta\tau$, we get $(1),(1,1,1,1)$ resp. $(2),(2,1,1,1)$. In the first case we are done by Proposition \ref{prop:a}. In the second case we apply $\beta$ again to obtain $(2),(1,1,1)$. This is Tom from Proposition \ref{prop:b}, hence zero mutable.

In each case, we have only used the transposition $\tau$ and the triangular mutations $\alpha$ and $\beta$. Hence, all Laurent polynomials above are elements of $\textup{ZMLP}_\triangle(3,b)$. One easily sees from the pairs of dual partitions $(\mathbf{a},\mathbf{b})$ in Table \ref{tab:classification} that they are also elements of $\textup{ZMLP}_{\textup{comb}}(3,b)$. This completes the proof.
\end{proof}

\begin{prop}
\label{prop:d}
For $k\geq 3$, the set $\textup{ZMLP}(k,k+1)$ consists of $3$ elements, with pair of dual partitions as in Table \ref{tab:classification}. We call them Tom, Jerry and Spike. Moreover, we have $\textup{ZMLP}(k,k+1)=\textup{ZMLP}_\triangle(k,k+1)=\textup{ZMLP}_{\textup{comb}}(k,k+1)$.
\end{prop}

\begin{proof}
Consider the case $\mathbf{a}=(1,\ldots,1)$ and write $\mathbf{reqdiv}_{e_b}(f)=(k+1,d_1,\ldots,d_{k-1},0)$. Then $\textup{reqdiv}_{e_a}(f)=0$ and Proposition \ref{prop:zeromut}, (2), gives $\sum_{i=1}^{k-1}d_i = k(k-1)/2$. Since $d_i-d_{i+1}\geq d_{i+1}-d_{i+2}$ the $d_i$ must be convex. This shows $d_i=k-i$. This gives $\mathbf{divstep}_{e_b}(f)=(2,\ldots,1)$ and $\mathbf{b}=(k,1)$. Similarly, for $\mathbf{b}=(1,\ldots,1)$ one can show $\mathbf{a}=(k)$.

Now consider $\mathbf{a}=(2,1,\ldots,1)$, such that $\textup{reqdiv}_{e_a}(f)=1$. Then we have $\sum_{i=1}^{k-1}d_i = k(k-1)/2-1$. But decreasing only one of the $d_i$ breaks the convexity. Hence, there cannot exist a ZMLP with $\mathbf{a}=(2,1,\ldots,1)$. One can similarly argue for most of the other choices of $\mathbf{a}$. The only thing that does not break the convexity will be to (i) lower all $d_i$ at the same time or (ii) increase the slope with which the $d_i$ fall. In (i) we don't change $\ell(\mathbf{divstep}_{e_b}(f))=\max(\mathbf{b})$ but we increase $\ell(\mathbf{divstep}_{e_a}(f))=\max(\mathbf{a})$. This will violate $\max(\mathbf{a})+\max(\mathbf{b})\leq \max(a,b)$ from Lemma \ref{lem:exists}. The case (ii) where the $d_i$ fall with slope $2$ is exactly Spike in Table \ref{tab:classification}. For any higher slope $\rho$, the divisibilities of the other edge $e_a$ will fall with slope $2/\rho$, which again violates convexity.

One easily checks that all cases are elements of $\textup{ZMLP}_{\textup{comb}}(k,k+1)$. We show that they are elements of $\textup{ZMLP}_\triangle(k,k+1)$. For Tom, application of $\tau\alpha^{-1}$ gives $(1),(1,\ldots,1)$. For Jerry, application of $\tau\beta\tau$ gives $(1),(1,\ldots,1)$. For Spike, by repeated application of $\tau\beta$ we get smaller and smaller cases of Spike until we finally arrive at $(2,1),(2,2)$, then at $(1,1),(2,1)$ and then at $(1),(1,1)$.
\end{proof}

\begin{rem}
Similarly, but with more combinatorial efforts, one should be able to find all ZMLPs for other families of rectangular triangles, leading to the numbers in Table \ref{tab:large}.
\end{rem}

\begin{prop}
For $a+b \leq 11$ we have $\textup{ZMLP}_\triangle(a,b)=\textup{ZMLP}(a,b)=\textup{ZMLP}_{\textup{comb}}(a,b)$.
\end{prop}

\begin{proof}
By definition of $\textup{ZMLP}_\triangle(a,b)$ and Proposition \ref{prop:contained} we have 
\[ \textup{ZMLP}_\triangle(a,b) \subset \textup{ZMLP}(a,b) \subset \textup{ZMLP}_{\textup{comb}}(a,b). \]
There are finitely many cases $a+b\leq 11$ and for each case $\textup{ZMLP}_{\textup{comb}}(a,b)$ is finite. We have verified by direct computation (with a computer) that in all cases $a+b\leq 11$ we have $\textup{ZMLP}_{\textup{comb}}(a,b)\subset\textup{ZMLP}_\triangle(a,b)$ by showing that in each case either $\beta$ decreases $a+b$ or $\alpha^{-1}$ exists and hence decreases $a+b$.
\end{proof}

\begin{expl}
\label{expl:nontriangular}
For $(a,b)=(5,7)$, we have $\textup{ZMLP}_\triangle(a,b)\subsetneq\textup{ZMLP}(a,b)$. Let $f$ be the ZMLP corresponding to $(\mathbf{a},\mathbf{b})=(4,1),(3,3,1)$. Then $f\in\textup{ZMLP}(a,b)\setminus\textup{ZMLP}_\triangle(a,b)$, since any sequence of mutations $f\mapsto\ldots\mapsto 1$ involves a non-triangular mutation. One such sequence is shown in Figure \ref{fig:57}. Note that application of $\beta$ and $\tau\beta\tau$ does not decrease $a+b$.
\end{expl}

\begin{figure}[h!]
\centering
\begin{tikzpicture}[scale=.7]
\draw (0,0) -- (7,0) -- (0,5) -- cycle;
\fill (0,0) circle (2pt) node[below]{\tiny$1$};
\fill (1,0) circle (2pt) node[below]{\tiny$7$};
\fill (2,0) circle (2pt) node[below]{\tiny$21$};
\fill (3,0) circle (2pt) node[below]{\tiny$35$};
\fill (4,0) circle (2pt) node[below]{\tiny$35$};
\fill (5,0) circle (2pt) node[below]{\tiny$21$};
\fill (6,0) circle (2pt) node[below]{\tiny$7$};
\fill (7,0) circle (2pt) node[below]{\tiny$1$};
\fill (0,1) circle (2pt) node[below]{\tiny$5$};
\fill (1,1) circle (2pt) node[below]{\tiny$25$};
\fill (2,1) circle (2pt) node[below]{\tiny$50$};
\fill (3,1) circle (2pt) node[below]{\tiny$50$};
\fill (4,1) circle (2pt) node[below]{\tiny$25$};
\fill (5,1) circle (2pt) node[below]{\tiny$5$};
\fill (0,2) circle (2pt) node[below]{\tiny$10$};
\fill (1,2) circle (2pt) node[below]{\tiny$33$};
\fill (2,2) circle (2pt) node[below]{\tiny$37$};
\fill (3,2) circle (2pt) node[below]{\tiny$15$};
\fill (4,2) circle (2pt) node[below]{\tiny$1$};
\fill (0,3) circle (2pt) node[below]{\tiny$10$};
\fill (1,3) circle (2pt) node[below]{\tiny$19$};
\fill (2,3) circle (2pt) node[below]{\tiny$8$};
\fill (0,4) circle (2pt) node[below]{\tiny$5$};
\fill (1,4) circle (2pt) node[below]{\tiny$4$};
\fill (0,5) circle (2pt) node[below]{\tiny$1$};
\draw[->] (7,2.5) -- (8,2.5);
\draw (9+0,0) -- (9+1,0) -- (9+4,5) -- (9+0,5) -- cycle;
\fill (9+0,0) circle (2pt) node[below]{\tiny$1$};
\fill (9+1,0) circle (2pt) node[below]{\tiny$1$};
\fill (9+0,1) circle (2pt) node[below]{\tiny$5$};
\fill (9+1,1) circle (2pt) node[below]{\tiny$5$};
\fill (9+0,2) circle (2pt) node[below]{\tiny$10$};
\fill (9+1,2) circle (2pt) node[below]{\tiny$13$};
\fill (9+2,2) circle (2pt) node[below]{\tiny$1$};
\fill (9+0,3) circle (2pt) node[below]{\tiny$10$};
\fill (9+1,3) circle (2pt) node[below]{\tiny$19$};
\fill (9+2,3) circle (2pt) node[below]{\tiny$8$};
\fill (9+0,4) circle (2pt) node[below]{\tiny$5$};
\fill (9+1,4) circle (2pt) node[below]{\tiny$14$};
\fill (9+2,4) circle (2pt) node[below]{\tiny$13$};
\fill (9+3,4) circle (2pt) node[below]{\tiny$4$};
\fill (9+0,5) circle (2pt) node[below]{\tiny$1$};
\fill (9+1,5) circle (2pt) node[below]{\tiny$4$};
\fill (9+2,5) circle (2pt) node[below]{\tiny$6$};
\fill (9+3,5) circle (2pt) node[below]{\tiny$4$};
\fill (9+4,5) circle (2pt) node[below]{\tiny$1$};
\draw[->] (13,4) -- (14,4);
\draw (15+0,5) -- (15+0,4) -- (15+1,3) -- (15+4,5) -- cycle;
\fill (15+0,5) circle (2pt) node[below]{\tiny$1$};
\fill (15+1,5) circle (2pt) node[below]{\tiny$4$};
\fill (15+2,5) circle (2pt) node[below]{\tiny$6$};
\fill (15+3,5) circle (2pt) node[below]{\tiny$4$};
\fill (15+4,5) circle (2pt) node[below]{\tiny$1$};
\fill (15+0,4) circle (2pt) node[below]{\tiny$1$};
\fill (15+1,4) circle (2pt) node[below]{\tiny$2$};
\fill (15+2,4) circle (2pt) node[below]{\tiny$1$};
\fill (15+1,3) circle (2pt) node[below]{\tiny$1$};
\draw[->] (19,4) -- (20,4);
\draw (21,5) -- (21,4) -- (22,3) -- cycle;
\fill (21,5) circle (2pt) node[below]{\tiny$1$};
\fill (21,4) circle (2pt) node[below]{\tiny$1$};
\fill (22,3) circle (2pt) node[below]{\tiny$1$};
\draw[->] (21,3) -- (21,2);
\draw (21,1) -- (21,0);
\fill (21,1) circle (2pt) node[below]{\tiny$1$};
\fill (21,0) circle (2pt) node[below]{\tiny$1$};
\draw[->] (21.25,.5) -- (21.75,.5);
\fill (22,.5) circle (2pt) node[below]{\tiny$1$};
\end{tikzpicture}
\caption{A zero mutable Laurent polynomial with non-triangular mutations.}
\label{fig:57}
\end{figure}
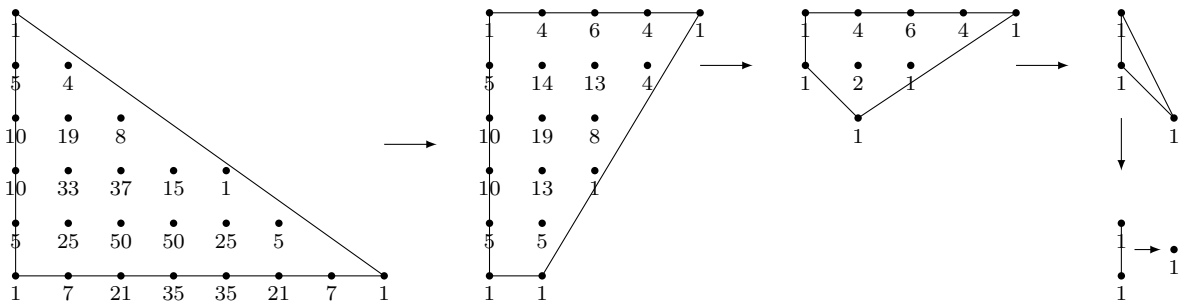

\subsection{Large triangles}													%%
\label{S:large}

Recall from Definition \ref{defi:comb} that we defined a set $\textup{ZMLP}_{\textup{comb}}(a,b)$ using combinatorial conditions fulfilled by ZMLPs, so that $\textup{ZMLP}(a,b)\subset\textup{ZMLP}_{\textup{comb}}(a,b)$. Given $a,b\in\mathbb{N}$, the cardinality $|\textup{ZMLP}_{\textup{comb}}(a,b)|$ can be easily calculated, e.g. by a computer, and gives an upper bound for $|\textup{ZMLP}(a,b)|$. This leads us to the following conjecture and Table \ref{tab:large}. By Conjecture \ref{conj:comb}, we expect $|\textup{ZMLP}(a,b)|$ to agree with this upper bound.

\begin{table}[h!]
\bgroup
\def\arraystretch{1.25}
\begin{tabular}{c|c}
a & $|\textup{ZMLP}_{\textup{comb}}(a,b)|, b\gg 0$  \\ \hline
1 & $1$ \\ \hline
2 & $2$ \\ \hline
3 & $4$ \\ \hline
4 & $8$ \\ \hline
5 & $11$ \\ \hline
6 & $15$ \\ \hline
7 & $\begin{matrix} 28, & b \equiv 1,2,5,6 \textup{ mod }7 \\ 31, & b \equiv 3,4 \textup{ mod }7 \end{matrix}$
\end{tabular}
\hspace{1cm}
\begin{tabular}{c|c}
k & $|\textup{ZMLP}_{\textup{comb}}(a,a+k)|, a\gg 0$  \\ \hline
1 & $3$ \\ \hline
2 & $3$ \\ \hline
3 & $6$ \\ \hline
4 & $15$
\end{tabular}
\egroup
\vspace{5mm}
\caption{Number of zero mutable Laurent polynomials for large rectangular triangles}
\label{tab:large}
\end{table}

Analyzing Table \ref{tab:large}, we have the following conjecture.

\begin{conj}
\begin{enumerate}[(1)]
\item For each $a\in\mathbb{N}$ there exists $b_0\in\mathbb{N}$ such that for all $b\geq b_0$ the number $|\textup{ZMLP}_{\textup{comb}}(a,b)|$ depends only on $b$ modulo $a$.
\item For each $k\in\mathbb{N}$ there exists $a_0\in\mathbb{N}$ such that for all $a\geq a_0$, $|\textup{ZMLP}_{\textup{comb}}(a,a+k)|$ depends only on $a$ modulo $k$ .
\end{enumerate}
\end{conj}

\section{Toric degenerations and log smoothings}									%%%
\label{S:log}

In this section we construct a degeneration of an affine Gorenstein toric variety $X_0$ into a union of toric varieties $\cup X_i$ via a central subdivision of the dual cone. Then we describe how a zero mutable Laurent polynomial defines a log structure on $\cup X_i$. The log singular locus will be determined by the dual partitions of the divisibility steps (Definition \ref{defi:dualtuple}). Then we use \cite{CR} to show that a compatible collection of divisorial extractions $Y_i\rightarrow X_i$ glues to a toroidal crossing space $\cup Y_i$ with smooth log structure over the standard log point and use the results of \cite{FFR} to obtain a smoothing of $X_0$.

\subsection{Toric degenerations}												%%
\label{S:toricdeg}

Let $M\simeq\mathbb{Z}^n$ be a lattice, let $N=\textup{Hom}(M,\mathbb{Z})$ be the dual lattice and let $M_{\mathbb{R}}=M\otimes_{\mathbb{Z}}\mathbb{R}$ and $N_{\mathbb{R}}=N\otimes_{\mathbb{Z}}\mathbb{R}$ be the vector spaces spanned by $M$ and $N$. Let $\triangle \subset M_{\mathbb{R}}$ be a full-dimensional lattice polytope. Let $\sigma \subset M_{\mathbb{R}} \times \mathbb{R}$ be the cone over $\triangle$,
\[ \sigma := C(\triangle \times \{1\}) := \{\overline{m}=(m,r) \in M_{\mathbb{R}}\times\mathbb{R} \ | \ m \in r\triangle \} \]
and let $\sigma^\vee \subset N_{\mathbb{R}}\times\mathbb{R}$ be the dual cone,
\[ \sigma^\vee = \{\overline{n} \in N_{\mathbb{R}}\times\mathbb{R} \ | \ \overline{n}(\overline{m}) \geq 0 \ \forall \overline{m}\in \sigma\}. \]
Let $X_0$ be the affine toric variety defined by $\sigma$. By definition,
\[ X_0 = \textup{Spec }\mathbb{C}[\sigma^\vee \cap (N\times\mathbb{Z})]. \]
Since $\sigma$ is a cone over a polytope $\triangle$ placed at height $1$, the variety $X_0$ has Gorenstein index $1$, hence, it is Gorenstein. Since $\sigma$ is contained in $M\times\mathbb{R}_{\geq 0}$, its dual $\sigma^\vee$ contains the ray $\rho_0=\mathbb{R}(0,1) \subset N_{\mathbb{R}}\times\mathbb{R}$ in its interior.

\begin{con}
\label{con:toricdeg}
For each facet $\omega_i$ of $\sigma^\vee$, let $\sigma_i^\vee \subset N_{\mathbb{R}}\times\mathbb{R}$ be the cone generated by $\rho_0$ and the rays of $\omega_i$. The cones form the \emph{central subdivision} of $\sigma^\vee$. Let $\varphi : \sigma^\vee \rightarrow \mathbb{R}$ be an integral piecewise linear function whose domains of linearity are the $\sigma_i$ and that takes values $0$ on $\rho_0$ and $\geq 1$ on the primitive generators of $\sigma^\vee$. Let $P_\varphi$ be the upper convex hull of $\varphi$,
\[ P_\varphi = \{(\overline{n},h) \in \sigma^\vee\times\mathbb{R} \ | \ h\geq\varphi(\overline{n})\}. \]
Define a family
\[ \mathcal{X} = \textup{Spec }\mathbb{C}[P_\varphi \cap (N\times\mathbb{Z}^2)] \rightarrow \mathbb{A}^1 = \textup{Spec }\mathbb{C}[\mathbb{N}]. \]
by projection onto the coordinate $z^{(0,\ldots,0,1)}$. The general fiber is isomorphic to $X_0$, since the general fiber of the map of (unbounded) polytopes $P_\varphi \rightarrow \mathbb{R}^1$ defined by projection onto the last component is isomorphic to $\sigma^\vee$. The central fiber $\pi^{-1}(0)$ is a union of toric varieties $\cup X_i$, where $X_i = \textup{Spec }\mathbb{C}[\sigma_i^\vee \cap (N\times\mathbb{Z})]$, since the $\sigma_i^\vee$ form the ``bottom components'' of $P_\varphi$.
\end{con}

\begin{expl}
Although we are interested in the case $n=2$, we describe the case $n=1$ in full generality, since in this case $P_\varphi$ is $3$-dimensional and can be depicted. Consider an interval $\triangle = [0,n] \subset \mathbb{R}$. Then $\sigma = C(\triangle\times\{1\}) \subset\mathbb{R}^2$ is generated by $(0,1)$ and $(n,1)$. The dual cone $\sigma^\vee\subset\mathbb{R}^2$ is generated by $(1,0)$ and $(-1,n)$. The central subdivision is given by the cone $\sigma_1$ generated by $(1,0),(0,1)$ and the cone $\sigma_2$ generated by $(1,0),(-1,n)$. The piecewise linear function $\varphi$ is given by $\varphi|_{\sigma_1}=x$ and $\varphi|_{\sigma_2}=-x$. The upper convex hull $P_\varphi\subset\mathbb{R}^3$ is generated by $(1,0,1)$, $(-1,n,1)$ and $(0,1,0)$. The monoid $P_\varphi\cap\mathbb{Z}^3$ is generated by $(1,0,1)$, $(-1,n,1)$, $(0,1,0)$ and $(0,0,1)$. Then we have
\[ \mathcal{X} = \textup{Spec }\mathbb{C}[z^{(1,0,1)},z^{(-1,n,1)},z^{(0,1,0)},z^{(0,0,1)}] = \textup{Spec }\mathbb{C}[x,y,z,t]/(xy-t^2z^n). \]
For $t\neq 0$ this is isomorphic to the surface $A_{n-1}$ defined by $xy=z^n$. For $t=0$ it is the union of two copies of $\mathbb{A}^2$.
\end{expl}

\begin{figure}[h!]
\centering
\begin{tikzpicture}[scale=1.1]
\draw (-1,1) node{$\sigma=$};
\draw (0,2.2) -- (0,0) -- (3,1.5);
\draw (0,1) node[below]{\tiny$(0,1)$} -- (1,1);
\draw[dashed] (1,1) -- (2,1) node[below]{\tiny$(n,1)$};
\fill (0,1) circle (1pt);
\fill (1,1) circle (1pt);
\fill (2,1) circle (1pt);
\draw[<->] (3.5,1) -- (4.5,1);
\draw (5.4,1) node{$\sigma^\vee=$};
\draw (5.9,2.2) -- (7,0) -- (8.2,0);
\draw (7,0) -- (7,2.2);
\draw (7.5,1) node{\tiny$\varphi=x$};
\draw (6.7,1) node{\tiny$-x$};
\fill (7,0) circle (1pt);
\fill (8,0) circle (1pt) node[below]{\tiny$(1,0)$};
\fill (6,2) circle (1pt) node[below]{\tiny$(-1,n)$};
\fill (7,1) circle (1pt);
\fill (8.5,1) node{$\leadsto$};
\end{tikzpicture}
\begin{tikzpicture}[scale=1.3]
\rotateRPY{-90}{10}{0}
\begin{scope}[RPY]
\draw (-1,0,1) node{$P_\varphi=$};
\draw (0,0,0) -- (1,0,1) -- (1.5,0,1.5);
\draw (0,0,0) -- (-1,2,1) -- (-1.2,2.4,1.2);
\draw (0,0,0) -- (0,2,0) -- (0,3,0);
\fill (1,0,1) circle (1pt) node[below]{\tiny$(1,0,1)$};
\fill (-1,2,1) circle (1pt) node[below]{\tiny$(-1,n,1)$};
\fill (0,1.5,0) circle (1pt) node[below]{\tiny$(0,1,0)$};
\fill (0,0,1) circle (1pt) node[below]{\tiny$(0,0,1)$};
\end{scope}
\end{tikzpicture}
\caption{Cones defining a toric degeneration of an $A_{n-1}$-surface.}
\label{fig:interval}
\end{figure}
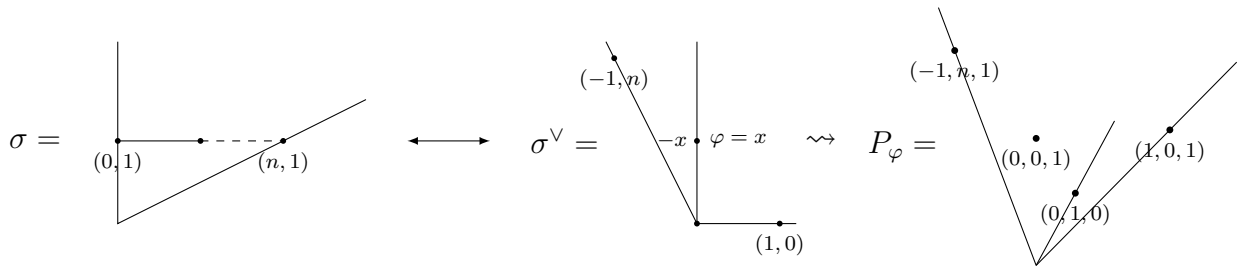

\begin{defi}
\label{defi:type}
A singularity is of type $\tfrac{1}{n}(a,b,c)$ if locally it is isomorphic to the cyclic quotient singularity $\mathbb{A}^3/\mathbb{Z}_n$, where $\mathbb{Z}_n=\mathbb{Z}/n\mathbb{Z}$ acts by $\zeta_n \cdot (a,b,c)\mapsto (\zeta_n^ax,\zeta_n^by,\zeta_n^cz)$.
\end{defi}

\begin{expl}
\label{expl:toricdeg}
Consider a rectangular triangle $\triangle(a,b) = \textup{Conv}\{(0,0),(0,a),(b,0)\}\subset\mathbb{R}^2$. The cone over $\triangle(a,b)$ (placed at height $1$) is $\sigma=\braket{(0,0,1),(0,a,1),(b,0,1)}$ and the dual cone is $\sigma^\vee=\braket{(1,0,0),(0,1,0),(-a,-b,ab)}$. The central subdivision of $\sigma^\vee$ is given by the cones $\sigma_1=\braket{(1,0,0),(0,1,0),(0,0,1)}$, $\sigma_2=\braket{(0,1,0),(-a,-b,ab-a-b),(0,0,1)}$ and $\sigma_3=\braket{(1,0,0),(-a,-b,ab-a-b),(0,0,1)}$. The piecewise linear function $\varphi$ is given by $\varphi|_{\sigma_1}=x+y$, $\varphi|_{\sigma_2}=y-(b-1)x$ and $\varphi|_{\sigma_3}=x-(a-1)y$. The upper convex hull $P_\varphi$ is generated by $(1,0,0,1)$, $(0,1,0,1)$, $(-a,-b,ab,ab-a-b)$, $(0,0,1,0)$, and $(0,0,0,1)$. The family 
\[ \mathcal{X}=\textup{Spec }\mathbb{C}[P_\varphi\cap\mathbb{Z}^4]\rightarrow\mathbb{A}^1 \]
is given by projection to $t=z^{(0,0,0,1)}$. For $t\neq 0$ the fiber is isomorphic to the affine cone over $\mathbb{P}(1,a,b)$. The fiber over $t=0$ has three components, $X_1=X_{\sigma_1}=\mathbb{A}^3$, $X_2=X_{\sigma_2}=\tfrac{1}{a}(1,b)\times\mathbb{A}^1$ and $X_3=X_{\sigma_3}=\tfrac{1}{b}(1,a)\times\mathbb{A}^1$. To see the latter, note that $\sigma_3$ is isomorphic via $\left(\begin{smallmatrix}1&0&0\\0&-1&0\\0&a&1\end{smallmatrix}\right)$ to the cone $\braket{(1,0,0),(-a,b,0),(0,0,1)}$, which gives $X_{\sigma_3}=\tfrac{1}{b}(1,a,0)$.
\end{expl}

\begin{expl}
\label{expl:toricdegtomjerry}
Consider the triangle $\triangle(2,3) = \textup{Conv}\{(0,0),(0,2),(3,0)\}\subset\mathbb{R}^2$. The cone $\sigma$ is generated by $(0,0,1),(0,2,1),(3,0,1)$. The dual cone $\sigma^\vee$ is generated by $(1,0,0)$, $(0,1,0)$, $(-2,-3,6)$. This is isomorphic via $\left(\begin{smallmatrix}1&0&0\\0&1&0\\1&1&1\end{smallmatrix}\right)$ to the cone generated by $(1,0,1)$, $(0,1,1)$, $(-2,-3,1)$, the cone over the triangle (placed at height $1$) $\textup{Conv}\{(1,0),(0,1),(-2,-3)\}$, which is the polar dual $\triangle^\circ$ of $\triangle$. This is always the case when $\triangle$ is a reflexive polytope. The central subdivision is shown in Figure \ref{fig:tomjerrypoly}. As in Example \ref{expl:toricdeg}, the induced toric degeneration $\cup X_i$ has irreducible components $X_1=\mathbb{A}^3$, $X_2=\tfrac{1}{2}(1,1)\times\mathbb{A}^1$ and $X_3=\tfrac{1}{3}(1,1)\times\mathbb{A}^1$.
\end{expl}

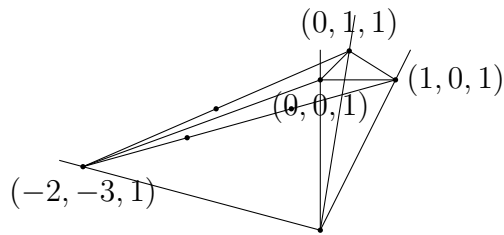
\begin{figure}[h!]
\centering
\begin{tikzpicture}[scale=1]
\rotateRPY{-90}{0}{0}
\begin{scope}[RPY]
\draw (0,0,1) -- (1,0,1);
\draw (0,0,1) -- (0,1,1);
\draw (0,0,1) -- (-2,-3,1);
\draw (1,0,1) -- (0,1,1) -- (-2,-3,1) -- (1,0,1);
\draw (0,0,-1) -- (1.2,0,1.4);
\draw (0,0,-1) -- (0,1.2,1.4);
\draw (0,0,-1) -- (-2.2,-3.3,1.2);
\draw (0,0,-1) -- (0,0,1.4);
\fill (0,0,1) circle (1pt) node[below]{$(0,0,1)$};
\fill (1,0,1) circle (1pt) node[right]{$(1,0,1)$};
\fill (0,1,1) circle (1pt) node[above]{$(0,1,1)$};
\fill (-2,-3,1) circle (1pt) node[below]{$(-2,-3,1)$};
\fill (-1,-1,1) circle (1pt);% node[below]{$(-1,-1,1)$};
\fill (-1,-2,1) circle (1pt);% node[below]{$(-1,-2,1)$};
\fill (0,-1,1) circle (1pt);% node[below]{$(0,-1,1)$};
\fill (0,0,-1) circle (1pt);
\end{scope}
\end{tikzpicture}
\caption{Central subdivision of $\sigma^\vee$ for the affine cone over $\mathbb{P}(1,2,3)$.}
\label{fig:tomjerrypoly}
\end{figure}

\subsection{Log structures}									%%
\label{S:log1}

We recall the notion of log structures. For more details see \cite{GSdata}, \S3.

\begin{defi}
A \textup{pre-log structure} on an algebraic space $X$ is a morphism 
\[ \alpha_X : \mathcal{M}_X\rightarrow\mathcal{O}_X \]
of sheaves of (multiplicative) monoids. It is a \emph{log structure} if $\alpha_X : \alpha_X^{-1}(\mathcal{O}_X^\times)\iso\mathcal{O}_X^\times$. The \textit{ghost sheaf} $\overline{\mathcal{M}}_X$ of $\alpha_X$ is the sheaf of monoids defined as the quotient
\[ 1 \rightarrow \mathcal{O}_X^\times \overset{\alpha_X^{-1}}{\longrightarrow} \mathcal{M}_X \rightarrow \overline{\mathcal{M}}_X \longrightarrow 0. \]
\end{defi}

\begin{expl}
Let $X$ be a scheme and $D\subset X$ a closed subset of pure codimension $1$. Let $j : X\setminus D \hookrightarrow X$ be the inclusion and consider
\[ \alpha_{(X,D)} : \mathcal{M}_{(X,D)} := (j_\star\mathcal{O}_{X\setminus D}^\times)\cap\mathcal{O}_X \hookrightarrow \mathcal{O}_X. \]
This is called the \textit{divisorial log structure induced by $D \subset X$}. Here $\mathcal{M}_{(X,D)}$ is the sheaf of functions that are invertible away from $D$. The ghost sheaf $\overline{\mathcal{M}}_{(X,D)}$ measures the order of vanishing of functions along $D$. Its rank at a point $x \in X$ equals the codimension of the smallest stratum of $D$ containing $x$.
\end{expl}

\begin{defi}
Let $\alpha : \mathcal{P}_X \rightarrow \mathcal{O}_X$ be a pre-log structure on $X$. Then the \textit{log structure associated to $\alpha$} is given by the sheaf of monoids
\[ \mathcal{M}_X := \bigslant{\mathcal{P}_X\oplus\mathcal{O}_X^\times}{\left\{(p,\alpha(p)^{-1}) \ | \ p\in\alpha^{-1}(\mathcal{O}_X^\times)\right\}} \]
and $\alpha_X : \mathcal{M}_X \rightarrow \mathcal{O}_X$ is given by $\alpha_X(p,f):=\alpha(p)\cdot f$.
\end{defi}

\begin{defi}
Let $f : X \rightarrow Y$ be a morphism of schemes and let $Y$ carry a log structure $\alpha_Y : \mathcal{M}_Y \rightarrow \mathcal{O}_Y$. The \textit{pullback log structure} of the log structure $\alpha_Y$ along $f$ is the log structure associated to the pre-log structure
\[ \alpha : f^{-1}(\mathcal{M}_Y) \overset{\alpha_Y}{\longrightarrow} f^{-1}(\mathcal{O}_Y) \overset{f^\star}{\longrightarrow} \mathcal{O}_X. \]
\end{defi}

\begin{expl}
\label{expl:ptN}
Consider an affine toric variety $X=\textup{Spec }\mathbb{C}[P]$ with the divisorial log structure induced by $\{0\}\subset X$. Let $\textup{Spec }\mathbb{C}$ be a point with equipped with the pullback log structure along the map to the origin $\textup{Spec }\mathbb{C}\hookrightarrow\mathbb{A}^1$. We denote this points by $\textup{pt}_P$. The log structure can be given explicitly by 
\[ \alpha_X : \mathcal{M}_X = \mathbb{C}^\times \oplus P, (x,p) \mapsto \begin{cases} x & p=0 \\ 0 & p \neq 0 \end{cases}. \]
In the special case $P=\mathbb{N}$ with $X=\mathbb{A}^1$ we call $\textup{pt}_{\mathbb{N}}$ the \textit{standard log point}.
\end{expl}

\begin{rem}
Consider $\cup X_i$ with the pullback of the divisorial log structure by $\cup X_i \subset \mathcal{X}$. This remembers some information of how $\cup X_i$ sits inside $\mathcal{X}$. Hence, it is possible to recover $\mathcal{X}$ and hence $X_0$ from $\cup X_i$. We will define a perturbation of this log structure on $\cup X_i$, such that we recover $X$ instead of $X_0$.
\end{rem}

\begin{defi}
Let $X$ be a scheme. A log structure $\mathcal{M}_X$ on $X$ is \textit{fine} if there is an \'etale open cover $\{U_i\}$ of $X$ such that on $U_i$ there is a finitely generated integral monoid $P_i$ and a pre-log structure $\underline{P_i}\rightarrow\mathcal{O}_{U_i}$ (\textit{chart}) whose associated log structure is isomorphic to the pullback log structure of $\mathcal{M}_X$. Here $\underline{P_i}$ is the constant sheaf of $P_i$. A fine log structure $\mathcal{M}_X$ is \textit{saturated} if $\overline{\mathcal{M}}_{X,x}$ is saturated for all geometric points $x\in X$.
\end{defi}

\begin{defi}
A morphism $f:X\rightarrow Y$ of fine log schemes is \textit{log smooth} if \'etale locally on $X$ and $Y$ it fits into a commutative diagram
\begin{center}
\hspace{0cm}
\begin{xy}
  \xymatrix{
	X \ar[r]\ar[d]							& \textup{Spec }\mathbb{Z}[P] \ar[d] \\
	Y \ar[r] 							& \textup{Spec }\mathbb{Z}[P']
  }
\end{xy} \\
\textup{ }
\end{center}
with the following properties:
\begin{compactenum}[(1)]
\item The horizontal maps induce charts $\underline{P}\rightarrow\mathcal{O}_X$ and $\underline{P}'\rightarrow\mathcal{O}_Y$ for the log structures on $X$ and $Y$, respectively;
\item The induced morphism $X\rightarrow Y\times_{\textup{Spec }\mathbb{Z}[P']} \textup{Spec }\mathbb{Z}[P]$ is a smooth morphism of schemes;
\item The right-hand vertical arrow is induced by a monoid homomorphism $P'\rightarrow P$ with $\textup{ker}(P'^{\textup{gp}}\rightarrow P^{\textup{gp}})$ and the torsion part of $\textup{coker}(P'^{\textup{gp}}\rightarrow P^{\textup{gp}})$ finite groups. Note that if $P$ and $P'$ are toric monoids, then $P^{\textup{gp}}$ and $P'^{\textup{gp}}$ are torsion free.
\end{compactenum}
\end{defi}

\subsection{Log structures with prescribed singular locus}							%%
\label{S:log2}

From the short exact sequence defining $\overline{\mathcal{M}}_X$ we see that the set of isomorphism classes of log structures with given ghost sheaf should be a subset of $\textup{Ext}^1(\overline{\mathcal{M}}_X^{\textup{gp}},\mathcal{O}_X^\times)$. By \cite{GSdata}, Corollary 3.12, this is indeed the case. More precisely, we have the following.

\begin{defi}
A \emph{polyhedral complex} $\mathcal{P}$ is a set of (possibly empty or non-compact) polytopes $\sigma\subset\mathbb{R}^n$ such that for each $\sigma_1,\sigma_2 \in \mathcal{P}$ we have $\sigma_1\cap\sigma_2\in\mathcal{P}$. A \emph{fan} is a polyhedral complex such that all $\sigma\in\mathcal{P}$ are cones.
\end{defi}

\begin{expl}
The central subdivision from Construction \ref{con:toricdeg} yields a polyhedral complex.
\end{expl}

\begin{defi}
For a polyhedral complex $\mathcal{P}$, let $\mathcal{P}^{[k]}$ be the subset of elements of codimension $k$. For $\omega\in\mathcal{P}^{[1]}$ and $\tau\in\mathcal{P}^{[2]}$ let
$d_\omega$ be a primitive integral normal vector of $\omega$ and define
\[ \epsilon_\tau(\omega) \begin{cases} 0 & \textup{if }\tau\not\subseteq\omega \\ \pm 1 & \textup{if }\tau\subseteq\omega \end{cases} \]
where the sign is chosen compatibly with some orientation on $\tau$ and the orientations on $\omega$ induced by the choices of $d_\omega$.
\end{defi}

\begin{prop}
\label{prop:joint}
Let $X=\cup X_i$ be the central fiber of the toric degeneration of an affine Gorenstein toric variety $X_0$ constructed in Construction \ref{con:toricdeg}. Let $U \subset X$ be an \'etale open set. Then the set of isomorphism classes of log structures on $X$ that are log smooth over the standard log point and such that the ghost sheaf has rank $r+1$ along a codimension $r$ stratum of $\cup X_i$ is in bijection with sections $(f_\omega) \in \Gamma(U,\oplus_{\omega\in\mathcal{P}^{[1]}} \mathcal{O}_{X_\omega}^\times)$ such that for every codimension $2$ face $\tau\in\mathcal{P}^{[1]}$ of the central subdivision $\mathcal{P}$ we have
\[ \prod_{\omega\in\mathcal{P}^{[1]}} d_\omega \otimes f_\omega^{\epsilon_\tau(\omega)}|_{X_\tau} = 1 \in M \otimes_{\mathbb{Z}}\Gamma(U,\mathcal{O}_{X_\tau}^\times). \]
This is called the \emph{joint compatibility condition}.
\end{prop}

\begin{proof}
The statement follows from \cite{GSdata}, Proposition 3.20 and Theorem 3.22, as follows. By \cite{GSdata}, Example 3.17, the central subdivision $\mathcal{P}$ defines what is called a \emph{ghost type} $X^g$. Having this ghost type is equivalent to the condition on the rank of the ghost sheaf stated in the proposition. By \cite{GSdata}, Proposition 3.20, the set of isomorphism classes of smooth log structures on $X$ with the given ghost type is isomorphic to the sections of a sheaf $\mathcal{LS}_{X^g}$. By \cite{GSdata}, Theorem 3.22, $\mathcal{LS}_{X^g}$ is isomorphic to a subsheaf of $\oplus_{\omega\in\mathcal{P}^{[1]}} \mathcal{O}_{X_\omega}^\times$ whose sections satisfy the joint compatibility condition. Note that we work in the cone picture, while \cite{GSdata} work in the fan picture, i.e., with the dual subdivision, so dimension is replaced with codimension. Moreover, we consider the affine case, where $\mathcal{P}$ has only one vertex. In general, one would have to consider an affine cover of $X$ labelled by the vertices of $\mathcal{P}$.
\end{proof}

\subsection{Log structures from zero mutable Laurent polynomials}						%%
\label{S:log3}

We we describe how a zero mutable Laurent polynomial gives sections $f_\omega$ as in Proposition \ref{prop:joint}.

\begin{con}
\label{con:singlog}
Let $f$ be a zero mutable Laurent polynomial with $2$-dimensional support $\triangle=\textup{Newt}(f)$. Let $(\mathbf{divstep}_e(f)^\vee)_e$ be the tuple of partitions dual to the divisibility steps (Definition \ref{defi:div}). Let $X_0$ be the affine Gorenstein toric variety defined by the cone $\sigma$ over $\triangle$ and let $\mathcal{X}\rightarrow\mathbb{A}^1$ be the toric degeneration defined by the central subdivision $\mathcal{P}$ of $\sigma^\vee$, as in Construction \ref{con:toricdeg}, with central fiber $\cup X_i$. Let $e$ be an edge of $\triangle$. The cone over $e$ gives a facet of $\sigma$. This is dual to an internal $2$-dimensional cone $\omega$ of $\mathcal{P}$, hence to a divisor $X_\omega=X_i\cap X_j\simeq\mathbb{A}^2$ of $\cup X_i$. The cone $\omega$ is generated by the internal point $(0,0,1)$ and a generator of $\sigma^\vee$. Let $u=z^{(0,0,1)}$ and $z$ be the corresponding monomials. Then $\Gamma(X_\omega,\mathcal{O}_{X_\omega}^\times)\simeq\mathbb{C}[u,z]$. Let $\mathbf{divstep}_e(f)^\vee=(d_1,\ldots,d_r)$. Let $\lambda_1,\ldots,\lambda_r\in\mathbb{C}$, be sufficiently general, e.g. irrational. Define the \emph{wall function}
\[ f_\omega = \prod_{i=1}^r (u^{d_i}-\lambda_i z) \in \mathbb{C}[u,z] \simeq \Gamma(X_\omega,\mathcal{O}_{X_\omega}^\times). \]
The degree of $f_\omega$ equals the lattice length $\ell(e)$ of $e$, since $\mathbf{divstep}_e(f)^\vee$ is a partition of $\ell(e)$ by Proposition \ref{prop:divstep}. Locally at the origin, $f_\omega$ has $r$ branches that intersect the line $X_{\rho_0}\simeq\mathbb{A}^1$ with multiplicities $d_1,\ldots,d_r$. The joint compatibility condition for $\tau=\rho_0=\cup_{\omega\in\mathcal{P}^{[1]}} \omega$ translates to the closing condition for the boundary of the polygon $\triangle$, hence it is satisfied.
\end{con}

\begin{expl}
Consider Tom and Jerry from Example \ref{expl:tomjerrydiv}, the zero mutable Laurent polynomials on $\triangle(2,3)$ with tuple of dual partitions $(1),(1,1),(2,1)$ and $(1),(2),(1,1,1)$, respectively. By Example \ref{expl:toricdegtomjerry}, the central subdivision of $\sigma^\vee$ gives a toric degeneration whose central fiber $\cup X_i$ has components $X_1=\mathbb{A}^3$, $X_2=A_1\times\mathbb{A}^1$ and $X_3=A_2\times\mathbb{A}^1$. The internal $2$-dimensional cones $\omega$ of $\mathcal{P}$ correspond to intersections $X_i \cap X_j$. The corresponding wall functions $f_\omega=f_{ij}$ are, for Tom,
\[ f_{23} = u-a_1 z, \qquad f_{13} = (u-b_1z)(u-b_2z), \qquad f_{12} = (u^2-c_1z)(u-c_2z), \]
and for Jerry,
\[ f_{23} = u-a_1 z, \qquad f_{13} = u^2-b_1z, \qquad f_{12} = (u-c_1z)(u-c_2z)(u-c_3z). \]
A sufficiently general choice for the parameters is $a_1=b_1=c_1=0$, $b_2=c_2=-1$, $c_3=1$.
\end{expl}

Note that Proposition \ref{prop:joint} gives a smooth log structure away from the vanishing loci $C_\omega=C_{ij}$ of the sections $f_\omega$. See Figure \ref{fig:intro2} for a picture of $\cup X_i$ and the log singular loci $\cup C_{ij}$ for Tom and Jerry. We explain in the next subsection how this can be extended to a smooth log structure on the whole space. 

\subsection{Divisorial extractions and log resolutions}											%%
\label{S:divex}

\begin{defi}
\label{defi:divex}
Consider a projective birational morphism $\sigma : Y \rightarrow X$ such that
\begin{enumerate}[(1)]
\item $X$ and $Y$ are quasiprojective $\mathbb{Q}$-factorial algebraic varieties,
\item there exists a unique prime divisor $E \subset Y$ such that $\Gamma=\sigma(E)$ has $\textup{codim}_X \Gamma\geq 2$,
\item $\sigma$ is an isomorphism outside $E$,
\item $-K_Y$ is $\sigma$-ample and the relative Picard number is $\rho(Y/X)=1$.
\end{enumerate}
We call $\sigma : Y \rightarrow X$ a \emph{divisorial contraction} of $E\subset Y$ and a \emph{divisorial extraction} of $\Gamma\subset X$.
\end{defi}

Divisorial contractions, together with flips and flops, form the elementary links of birational maps between Fano fiber spaces in the Sarkisov program \cite{Cor}. There are classifications for divisorial extractions of points \cite{Cor}\cite{Kaw}, of smooth curves \cite{Tzi}, and if $Y$ has only Gorenstein singularities \cite{Cut}. Divisorial extractions of singular curves in $\mathbb{A}^3$ and of curves in singular threefolds have been studied in some cases, but so far there is no complete classification.

\begin{defi}
\label{defi:compatible}
Let $\cup X_i$ be reducible. A collection of divisorial extractions $\sigma_i : Y_i \rightarrow X_i$ is \emph{compatible} if for each $i,j$ we have $\sigma_i^{-1}(X_i\cap X_j)\simeq\sigma_j^{-1}(X_i\cap X_j)$ and if $\sigma_i^{-1}(X_i\cap X_j)$ contains an isolated threefold singularity $P$ of $Y_i$, then $\sigma_j^{-1}(X_i\cap X_j)$ contains an isolated threefold singularity $Q$ of $Y_j$ such that $Y_i$ \'etale locally at $P$ is isomorphic to $Y_j$ \'etale locally at $Q$.
\end{defi}

Let $X_0$ be an affine Gorenstein toric $3$-fold defined by the cone over a lattice polygon $\triangle$ (at height $1$). Let $\cup X_i$ be the central fiber of the toric degeneration of $X_0$ from Construction \ref{con:toricdeg}. Let $\cup C_{ij}\subset X_i\cap X_j$ be the vanishing locus of wall functions $(f_\omega)$ satisfying the joint compatibility condition (Proposition \ref{prop:joint}). Let $Y_i \rightarrow X_i$ be a compatible collection of divisorial extractions of $C=\cup C_{ij}\subset \cup X_i$. By compatibility, we can glue the $Y_i$ to obtain a variety $Y=\cup Y_i$. This is a Gorenstein toroidal crossing space (\cite{SS}, Definition 7.1), in particular a generically toroidal crossing space (\cite{CR}, Definition 2.1). We expect that $\mathcal{LS}_Y=f^\star\mathcal{LS}_X(-C)$, i.e. there is a unique log structure on $Y$ over the standard log point that on $Y\setminus f^{-1}(C)=X\setminus C$ is the given one on $X\setminus C$, and that $\cup Y_i$ is log smooth over the standard log point away from a single point $P$, where \'etale locally it is isomorphic to $(xy=0)\subset\tfrac{1}{r}(1,-1,a,-a)$. Moreover, we expect that the result of \cite{FFR}, Theorem 1.7, generalizes to this setting, so that $\cup Y_i$ can be smoothed to an orbifold with terminal singularities $X$.

\section{Constructing compatible divisorial extractions}											%%%

Let $\triangle$ be a lattice polygon, let $X_0$ be the affine Gorenstein toric surface defined by the cone over $\triangle$ (placed at height $1$) and let $\cup X_i$ be the central fiber of the toric degeneration defined by the central subdivision of the dual cone $\sigma^\vee$, see Construction \ref{con:toricdeg}. Let $\cup C_{ij}$ be a union of curves $C_{ij}\subset D_{ij}=X_i\cap X_j$ and consider the log structure on $\cup X_i$ with singular locus $\cup C_{ij}$, see Construction \ref{con:singlog}. By Conjecture \ref{conj:log}, a compatible collection of divisorial extractions $Y_i\rightarrow X_i$ of $\cup C_{ij}$ is expected to induce a smoothing of $X_0$. In this section we consider rectangular triangles $\triangle(a,b)$. We describe two collections of divisorial extractions $Y_i\rightarrow X_i$ that we expect to be compatible and partially prove this for the families called Tom, Jerry, Spike and Tyke in Table \ref{tab:classification}. Precisely, we prove Theorem \ref{thm:main}.

\subsection{Rectangular triangles: reduction to $\mathbb{A}^3$}								%%
\label{S:rect}

\begin{con}
\label{con:procedure}
Consider a rectangular triangle $\triangle(a,b)$. By Example \ref{expl:toricdeg}, the toric degeneration $\cup X_i$ from Construction \ref{con:toricdeg} has $3$ components $X_1=\mathbb{A}^3$ and $X_2=\tfrac{1}{a}(1,b)\times\mathbb{A}^1$ and $X_3=\tfrac{1}{b}(1,a)\times\mathbb{A}^1$. All divisors $D_{ij}=X_i\cap X_j$ are isomorphic to $\mathbb{A}^2$. The log singular locus consists is a union of (possibly reducible) curves $C_{12}\subset D_{12}$ and $C_{13}\subset D_{13}$ and $C_{23}\subset D_{23}$. Write $C_{\mathbf{a}}$ for the curve with intersection multiplicities determined by a partition $\mathbf{a}$. To construct log crepant log resolutions, we consider two different collections of divisorial extractions:
\begin{enumerate}[(1)]
\item 
\begin{enumerate}[(a)]
\item The blow up $Y_2\rightarrow X_2$ of a line $C_{(1)}\subset D_{23}\subset X_2=\tfrac{1}{a}(1,b)\times\mathbb{A}^1$.
\item The blow up $Y\rightarrow X_1$ of $C_{\mathbf{a}}\subset D_{12}\subset X_1=\mathbb{A}^3$.
\item A divisorial extraction $Z\rightarrow Y$ of the proper transform $\tilde{C}_{\mathbf{b}}\subset\tilde{D}_{13}\subset Y$.
\end{enumerate}
\item 
\begin{enumerate}[(a)]
\item The blow up $Y_3\rightarrow X_3$ of a line $C_{(1)}\subset D_{23}\subset X_3=\tfrac{1}{b}(1,a)\times\mathbb{A}^1$.
\item The blow up $Y\rightarrow X_1$ of $C_{\mathbf{b}}\subset D_{13}\subset X_1=\mathbb{A}^3$.
\item A divisorial extraction $Z\rightarrow Y$ of the proper transform $\tilde{C}_{\mathbf{a}}\subset\tilde{D}_{12}\subset Y$.
\end{enumerate}
\end{enumerate}
The construction of the divisorial extraction $Z\rightarrow Y$ will depend on the case considered and will be made precise for each case in the following.
\end{con}

\begin{prop}
\label{prop:stepa}
Consider step (a) in Construction \ref{con:procedure}.
\begin{enumerate}[(1)]
\item Let $Y_2\rightarrow X_2=\tfrac{1}{a}(1,b)\times\mathbb{A}^1$ be the blow up of a line $C_{(1)}\subset D_{23} \subset X_2$. Then $\tilde{D}_{13}\simeq A_{a-1}$ and $Y_2$ contains an isolated terminal cyclic quotient singularity of type $\tfrac{1}{a}(1,-1,b)$.
\item Let $Y_3\rightarrow X_3=\tfrac{1}{b}(1,a)\times\mathbb{A}^1$ be the blow up of a line $C_{(1)}\subset D_{23} \subset X_3$. Then $\tilde{D}_{12}\simeq A_{b-1}$ and $Y_3$ contains an isolated terminal cyclic quotient singularity of type $\tfrac{1}{b}(1,-1,a)$.
\end{enumerate}
\end{prop}

\begin{proof}
We prove (1), then (2) follows by symmetry. The variety $\tfrac{1}{a}(1,b)\times\mathbb{A}^1$ is naturally a toric variety and the divisor $D_{23}\simeq\mathbb{A}^2$ is a toric divisor. Without loss of generality, we can assume that the line $C_{(1)}$ is a toric line that is different from the singular locus of $\tfrac{1}{a}(1,b)\times\mathbb{A}^1$. Then the blow up of $C_{(1)}$ can be performed as a toric blow up. The cone of $\tfrac{1}{a}(1,b)\times\mathbb{A}^1 = \tfrac{1}{a}(1,b,0)$ has ray generators $(1,0,0),(0,1,0),(0,-b,a)$. The facet spanned by $(1,0,0),(0,-b,a)$ corresponds to $C_{(1)}$. The blow up along $C_{(1)}$ corresponds to the star subdivision of the cone obtained by adding a ray generated by $(1,0,0)+(0,-b,a)=(1,-b,a)$. The resulting fan has a cone $\braket{(1,0,0),(0,1,0),(1,-b,a)}$. By \cite{CLS}, Theorem 11.4.21, this cone defines the terminal cyclic quotient singularity $X_{\sigma_3}=\tfrac{1}{a}(1,-1,b)$. The divisor $\tilde{D}_{13}$ corresponds to the ray $(0,1,0)$. The projection of $\braket{(1,0,0),(0,1,0),(1,-b,a)}$ to the normal space of this ray is isomorphic to the cone $\braket{(1,0),(1,a)}$. Hence, $\tilde{D}_{13}\simeq \tfrac{1}{a}(1,-1)=A_{a-1}$.
\end{proof}

\begin{prop}
\label{prop:divisor}
Consider steps (b) and (c) in Construction \ref{con:procedure}.
\begin{enumerate}[(1)]
\item Let $Y\rightarrow X_1=\mathbb{A}^3$ be the blow up of a curve $C_{\mathbf{a}}\subset D_{12}\subset X_1=\mathbb{A}^3$ and let $Z\rightarrow Y$ be a divisorial extraction of $\tilde{C}_{\mathbf{b}}\subset\tilde{D}_{13}\subset Y$. Then the proper transform of $\tilde{D}_{13}$ in $Z$ is isomorphic to $A_{a-1}=\tfrac{1}{a}(1,-1)$.
\item Let $Y\rightarrow X_1=\mathbb{A}^3$ be the blow up of a curve $C_{\mathbf{b}}\subset D_{13}\subset X_1=\mathbb{A}^3$ and let $Z\rightarrow Y$ be a divisorial extraction of $\tilde{C}_{\mathbf{a}}\subset\tilde{D}_{12}\subset Y$. Then the proper transform of $\tilde{D}_{12}$ in $Z$ is isomorphic to $A_{b-1}=\tfrac{1}{b}(1,-1)$.
\end{enumerate}
\end{prop}

\begin{proof}
We prove (1), then (2) follows by symmetry. Choose coordinates on $\mathbb{A}^3$ such that $D_{12}=\{x=0\}$ and $C_{\mathbf{a}}=\{x=f(y,z)=0\}\subset\mathbb{A}^3_{x,y,z}$. Let $Y\rightarrow\mathbb{A}^3$ be the blow up of $C_{\mathbf{a}}\subset D_{12}$. Then $Y=\{xU=f(y,z)V\}\subset\mathbb{A}^3_{x,y,z}\times\mathbb{P}^1_{U,V}$. On the affine chart $U=1$ this is smooth and on $V=1$ it is given by $Y=\{xu=f(y,z)\}\subset\mathbb{A}^4_{x,y,z,u}$. Since $C_{\mathbf{a}}$ has intersection multiplicity $a$ with $D_{12}\cap D_{13}=\{x=y=0\}$ at the origin, we have $f(0,z)=z^n$ and the proper transform of $D_{13}=\{y=0\}$ is $\tilde{D}_{13}=\{xu=z^a\}\subset\mathbb{A}^3_{x,z,u}$. Hence, $\tilde{D}_{13}\simeq A_{a-1}$.

If $Z\rightarrow Y$ is a divisorial extraction of a curve $\tilde{C}_{\mathbf{b}}$ that is contained in a surface $\tilde{D}_{13}$, then the proper transform of $\tilde{D}_{13}$ in $Z$ is isomorphic to $\tilde{D}_{13}$, since $Z\rightarrow Y$ is an isomorphism away from $\tilde{C}_{\mathbf{b}}$ and in the tangent directions of $\tilde{C}_{\mathbf{b}}$, it only affects to normal directions to $\tilde{C}_{\mathbf{b}}$. Hence, the proper transform of $\tilde{D}_{13}$ is isomorphic to $A_{a-1}$.
\end{proof}

By Propositions \ref{prop:stepa} and \ref{prop:divisor}, to obtain a compatible collection of divisorial extractions, we have to show the following.

\begin{conj}
\label{conj:divex}
If $(\mathbf{a},\mathbf{b})$ is the pair of dual partitions (Definition \ref{defi:dualpair}) for some zero mutable Laurent polynomial $f\in\textup{ZMLP}(a,b)$, then we have one of the following:
\begin{enumerate}[(1)]
\item There exists a divisorial extraction $Z\rightarrow \mathbb{A}^3$ as in Construction \ref{con:procedure}, (1), such that $Z$ has an isolated singularity of type $\tfrac{1}{a}(1,-1,b)$.
\item There exists a divisorial extraction $Z\rightarrow \mathbb{A}^3$ as in Construction \ref{con:procedure}, (2), such that $Z$ has an isolated singularity of type $\tfrac{1}{b}(1,-1,a)$.
\end{enumerate}
\end{conj}

\begin{prop}
\label{prop:maindivex}
Conjecture \ref{conj:main} follows from Conjecture \ref{conj:divex}.
\end{prop}

\begin{proof}
This is the content of Propositions \ref{prop:stepa} and \ref{prop:divisor}.
\end{proof}

In the rest of this chapter, we prove Conjecture \ref{conj:divex}, hence Conjecture \ref{conj:main}, for the cases called Tom, Jerry and Tyke in Table \ref{tab:classification}.

\subsection{$A_n$-triangles}													%%
\label{S:An}

Consider an $A_{n-1}$ triangle $\triangle(1,n)$. By Proposition \ref{prop:a}, there is only one ZMLP on $\triangle(1,n)$. Its pair of dual partitions is $(1),(1,\ldots,1)$, hence $C_{\mathbf{a}}=C_{(1)}$ is a line and $C_{\mathbf{b}}=C_{(1,\ldots,1)}$ is a plane $n$-fold intersection of lines.

\begin{prop}
\label{prop:An}
For $(a,b)=(1,n)$ both (1) and (2) in Conjecture \ref{conj:divex} are true:
\begin{enumerate}[(1)]
\item Let $Z\rightarrow \mathbb{A}^3$ be the composition of the blow up $Y\rightarrow\mathbb{A}^3$ of $C_{(1)}\subset D_{12}\simeq\mathbb{A}^2\subset\mathbb{A}^3$ and the blow up $Z \rightarrow Y$ of $\tilde{C}_{(1,\ldots,1)}\subset\tilde{D}_{13}\subset Y$. Then $Z$ is smooth.
\item Let $Y\rightarrow \mathbb{A}^3$ be the blow up of $C_{(1,\ldots,1)}\subset D_{13}\simeq\mathbb{A}^2\subset\mathbb{A}^3$. Then there exists a divisorial extraction $Z \rightarrow Y$ of $\tilde{C}_{(1)}\subset\tilde{D}_{12}\subset Y$ such that $Z$ has an isolated singularity of type $\tfrac{1}{n}(1,-1,1)\simeq \tfrac{1}{n}(1,-1,-1)$.
\end{enumerate}
\end{prop}

\begin{proof}
For (1), write $C_{(1,\ldots,1)}=\cup_{i=1}^n L_i$. Let $Y \rightarrow \mathbb{A}^3$ be the blow up of a line $C_{(1)}\subset D_{12}$. The proper transforms $\tilde{L_i}$ of the $n$ lines $L_i$ become disjoint in $Y$. Hence, the blow up $Z\rightarrow Y$ of $\tilde{C}_{(1,\ldots,1)}=\cup_{i=1}^n \tilde{L_i}$ is isomorphic to the successive blow up of the lines $\tilde{L}_1,\ldots,\tilde{L}_n$. In each step we blow up a smooth variety along a smooth subvariety. Hence, the result $Z$ is smooth.

Now consider (2). After a change of coordinates we can assume that $C_{(1,\ldots,1)}=\cup_{i=1}^n L_i$ is given by $\{x=\prod_{i=1}^n (z + \lambda_i y)=0\}\subset\mathbb{A}^3_{x,y,z}$, with $\lambda_i\neq 0$, and that the line $C_{(1)}$ is given by $\{y=z=0\}\subset\mathbb{A}^3_{x,y,z}$. The blow up $Y\rightarrow\mathbb{A}^3$ at $C_{(1,\ldots,1)}$ is given by $\{xU=V\prod_{i=1}^n (z + \lambda_i y)\}\subset\mathbb{A}^3\times\mathbb{P}_{U,V}^1$. On the affine chart $U=1$ this is smooth. On the chart $V=1$ it is given by
\[ Y = \left\{xu = \prod_{i=1}^n (z + \lambda_i y)\right\} \subset \mathbb{A}^4. \]
The proper transform of $D_{12}=\{y=0\}$ is given by $\tilde{D}_{12}=\{xu=z^n\}\simeq A_{n-1}$. The proper transform of the line $C_{(1)}$ is $\tilde{C}_{(1)}=\{u=y=z=0\}\subset\tilde{D}_{12}$, the $x$-axis. We explicitly write down a divisorial extraction $Z\rightarrow Y$ of $\tilde{C}_{(1)}\subset Y$ as follows. Define a morphism $\mathbb{A}^3_{u,y,z}\rightarrow Y$ by the corresponding homomorphism of coordinate rings 
\begin{eqnarray*}
\bigslant{\mathbb{C}[u,y,z,x]}{(xu-\prod_{i=1}^n(z+\lambda_i y))} &\rightarrow& \mathbb{C}[u,y,z] \\
(u,y,z,x) &\mapsto& \left(u^n,yu,zu,\prod_{i=1}^n(z+\lambda_i y)\right).
\end{eqnarray*}
The morphism eliminates $x=\prod_{i=1}^n(z+\lambda_i y)$ and factors through $\mathbb{A}^3\rightarrow\mathbb{A}^3/\mathbb{Z}_n=\tfrac{1}{n}(1,-1,-1)$, since $u^n$, $yu$ and $zu$ are invariant under the $\mathbb{Z}_n$-action with weights $(1,-1,-1)$. Hence, we get a morphism $Z=\tfrac{1}{n}(1,-1,-1)\rightarrow Y$. This is an isomorphism away from $E=\{u=0\}\subset Z$ and the image of $E$ is given by $\tilde{C}_{(1)}=\{u=y=z=0\}$. Hence, $Z$ is a divisorial extraction of $\tilde{C}_{(1)}$.
\end{proof}

\begin{rem}
Note that the morphism $\frac{1}{n}(1,-1,-1)\rightarrow Y\subset\mathbb{A}^4$ in the proof above is the restriction of the toric weighted blow up of $\mathbb{A}_{x,u,y,z}^4$ with weights $(0,n,1,1)$. The corresponding subdivision of the standard cone $\braket{e_1,e_2,e_3,e_4}$ adds a ray $(0,n,1,1)$ and by \cite{CLS}, Theorem 11.4.21, the cone $\braket{e_1,e_3,e_4,(0,n,1,1)}$ corresponds to a singularity $\frac{1}{n}(0,1,-1,-1)=\frac{1}{n}(1,-1,-1)\times\mathbb{A}_x^1$. Setting $x=\prod_{i=1}^n (z + \lambda_i y)$ gives a singularity $\frac{1}{n}(1,-1,-1)$.
\end{rem}

\begin{rem}
For $n=2$ we can see this torically. In this case, $C_{(1,1)}$ is the union of two intersecting lines. This is a complete intersection defined by $x=yz=0$ in suitable coordinates. The blow up $Y \rightarrow X$ of $C_{(1,1)}$ is given by $\{xU=yzV\}\subset\mathbb{A}_{x,y,z}^3 \times \mathbb{P}_{U,V}^1$. In the affine chart $V=1$ we have $\{xu=yz\}\subset\mathbb{A}^4$, which is the affine toric variety defined by the cone with generators $(0,0,0),(1,0,1),(0,1,1),(1,1,1)$. The blow up $Z \rightarrow Y$ of the line $\tilde{C}_{(1)}$ corresponds to the subdivision by $(0,1,1)+(1,1,1)=(1,2,2)$. The resulting fan contains a cone generated by $(0,0,1),(1,0,1),(1,1,2)$. This is isomorphic via the transformation $\left(\begin{smallmatrix}1&0&0\\0&1&0\\-1&0&1\end{smallmatrix}\right)$ to the cone generated by $(0,0,1),(1,0,0),(0,2,1)$. Hence, by \cite{CLS}, Theorem 11.4.21, $Z$ has an isolated singularity of type $\tfrac{1}{2}(1,-1,-1)=\tfrac{1}{2}(1,1,1)$.
\end{rem}

\subsection{Induction step: the geometry of elementary triangular mutations}					%%

Recall the elementary triangular mutations $\alpha$ and $\beta$ from Definition \ref{defi:alphabeta}.

\begin{prop}
\label{prop:induction}\
\begin{enumerate}[(1)]
\item If Conjecture \ref{conj:divex}, (1), holds for $(\mathbf{a},\mathbf{b})$, then it holds for $\alpha(\mathbf{a},\mathbf{b})$ and $\beta(\mathbf{a},\mathbf{b})$. 
\item If Conjecture \ref{conj:divex}, (2), holds for $(\mathbf{a},\mathbf{b})$, then it holds for $\tau\alpha\tau(\mathbf{a},\mathbf{b})$ and $\tau\beta\tau(\mathbf{a},\mathbf{b})$.
\end{enumerate}
\end{prop}

\begin{proof}
We prove (1), then (2) follows by symmetry. By Proposition \ref{prop:degree} we have $\alpha(a,b)=(a,a+b)$ and $\beta(a,b)=(a,\ell a-b)$ for some $\ell\in\mathbb{N}$. Note that $\tfrac{1}{a}(1,-1,a+b)\simeq\tfrac{1}{a}(1,-1,b)$ and $\tfrac{1}{a}(1,-1,\ell a-b)\simeq\tfrac{1}{a}(1,-1,-b)\simeq\tfrac{1}{a}(1,-1,b)$. Hence, we have to show that application of $\alpha$ and $\beta$ does not change the type of the isolated singularity produced by divisorial extraction. Write $\mathbf{a}=(a_1,\ldots,a_k)$, $\mathbf{b}=(b_1,\ldots,b_l)$.

By Proposition \ref{prop:degree}, $\alpha(\mathbf{a},\mathbf{b})=(\mathbf{a},\tilde{\mathbf{b}})$ with $\tilde{\mathbf{b}}=(a,b_1,\ldots,b_l)$. Hence, $C_{\tilde{\mathbf{b}}}=C_{(a)}\cup C_{\mathbf{b}}$, where as usual we write $C_{(a)}$ for a curve that intersects $X_1\cap X_2\cap X_3\simeq\mathbb{A}^1$ in a single point with multiplicity $a$. Let $Y\rightarrow\mathbb{A}^3$ be the blow up of $C_{\mathbf{a}}\subset D_{12}$. Then the proper transforms $\tilde{C}_{(a)}$ and $\tilde{C}_{\mathbf{b}}$ are disjoint. If we blow up $\tilde{C}_{(a)}$ and perform the divisorial extraction of $\tilde{C}_{\mathbf{b}}$ that satisfies Conjecture \ref{conj:divex}, (1), then the resulting divisorial extraction $Z\rightarrow Y$ of $\tilde{C}_{\tilde{\mathbf{b}}}$ also satisfies Conjecture \ref{conj:divex}, (1).

By Proposition \ref{prop:degree}, $\beta(\mathbf{a},\mathbf{b})=(\mathbf{a},\tilde{\mathbf{b}})$ with $\tilde{\mathbf{b}}=(a-b_1,\ldots,a-b_l)$. By Proposition \ref{prop:divisor}, $\tilde{D}_{13}\simeq A_{a-1}$. Write $\tilde{D}_{13}=\{xu=z^a\}\subset\mathbb{A}^3_{x,z,u}$. An element $b_i$ of $\mathbf{b}$ corresponds to a curve $C_{(b_i)}$ whose proper transform under this blow up is given by $\tilde{C}_{(b_i)}=\{x=z^{b_i}\}\subset\tilde{D}_{13}$. Similarly, $a-b_i$ corresponds to $\tilde{C}_{(a-b_i)}=\{x=z^{a-b_i}\}=\{u=z^{b_i}\}\subset\tilde{D}_{13}$. Hence, $\tilde{C}_{\mathbf{b}}$ and $\tilde{C}_{\tilde{\mathbf{b}}}$ are isomorphic under exchanging the $x$- and $u$-coordinate. The divisorial extraction of $\tilde{C}_{\mathbf{b}}$ that satisfies Conjecture \ref{conj:divex}, (1), can be composed with this isomorphism to obtain a divisorial extraction $Z\rightarrow Y$ of $\tilde{C}_{\tilde{\mathbf{b}}}$ that satisfies Conjecture \ref{conj:divex}, (1).

This proves the statement.
\end{proof}

\begin{rem}
From Propositions \ref{prop:An} and \ref{prop:induction} it follows that if $(\mathbf{a},\mathbf{b})\in\textup{ZMLP}_{\textup{comb}}(a,b)$ can be obtained from $(1),(1,\ldots,1)$ via only $\alpha,\beta$ or only $\tau\alpha\tau,\tau\beta\tau$, then it satisfies Conjecture \ref{conj:divex}. However, for most $(\mathbf{a},\mathbf{b})\in\textup{ZMLP}_{\textup{comb}}(a,b)$ this is not the case, and we have to explicitly construct divisorial extractions $Z\rightarrow\mathbb{A}^3$ satisfying Conjecture \ref{conj:divex} in these cases.
\end{rem}

\subsection{Tom}														%%
\label{S:tom}

\begin{prop}
\label{prop:tom1}
For case Tom with $(a,b)=(k,k+1)$, i.e., $\mathbf{a}=(1,\ldots,1)$ and $\mathbf{b}=(k,1)$, Conjecture \ref{conj:divex}, (1) is true: Let $Z\rightarrow \mathbb{A}^3$ be the composition of the blow up $Y\rightarrow\mathbb{A}^3$ of $C_{(1,\ldots,1)}\subset D_{12}\simeq\mathbb{A}^2\subset\mathbb{A}^3$ and the blow up $Z \rightarrow Y$ of $\tilde{C}_{(k,1)}\subset\tilde{D}_{13}\subset Y$. Then $\tilde{D}_{13}\simeq A_{k-1}$ and $Z$ has an isolated singularity of type $\tfrac{1}{k}(1,-1,k+1)=\tfrac{1}{k}(1,-1,1)\simeq\tfrac{1}{k}(1,-1,-1)$.
\end{prop}

\begin{proof}
Note that $(\mathbf{a},\mathbf{b})=\alpha((1,\ldots,1),(1))$. Then the statement follows from Propositions \ref{prop:An} and \ref{prop:induction}. Let us summarize the reason again. After the blow up $Y\rightarrow\mathbb{A}^3$ of $C_{(1,\ldots,1)}$, the proper transform of $C_{(k,1)}$ is a disjoint union $\tilde{C}_{(k,1)}=\tilde{C}_{(k)}\cup \tilde{C}_{(1)}$. The curve $\tilde{C}_{(k)}$ is disjoint from the singularities of $Y$. Hence, its blow up does not produce any singularities. The blow up of $\tilde{C}_{(1)}$ produces an isolated singularity of type $\tfrac{1}{k}(1,-1,-1)$ as shown in the proof of Proposition \ref{prop:An}, (2).
\end{proof}

\begin{prop}
\label{prop:tom2}
For case Tom with $(a,b)=(2,k)$ or $(a,b)=(3,k)$, Conjecture \ref{conj:divex}, (1) is true. More generally:
\begin{enumerate}[(a)]
\item Let $(a,b)=(a,ka+1)$, $\mathbf{a}=(1,\ldots,1)$ and $\mathbf{b}=(a,\ldots,a,1)$. Let $Y\rightarrow\mathbb{A}^3$ be the blow up of $C_{\mathbf{a}}\subset D_{12}\subset \mathbb{A}^3$ and let $Z\rightarrow Y$ be the blow up of $\tilde{C}_{\mathbf{b}}\subset\tilde{D}_{13}$. Then $Z$ has an isolated singularity of type $\tfrac{1}{a}(1,-1,ka+1)=\tfrac{1}{a}(1,-1,1)\simeq\tfrac{1}{a}(1,-1,-1)$.
\item Let $(a,b)=(a,ka+a-1)$, $\mathbf{a}=(1,\ldots,1)$ and $\mathbf{b}=(a,\ldots,a,a-1)$. Let $Y\rightarrow\mathbb{A}^3$ be the blow up of $C_{\mathbf{a}}\subset D_{12}\subset \mathbb{A}^3$ and let $Z\rightarrow Y$ be the blow up of $\tilde{C}_{\mathbf{b}}\subset\tilde{D}_{13}$. Then $Z$ has an isolated singularity of type $\tfrac{1}{a}(1,-1,ka+a-1)=\tfrac{1}{a}(1,-1,-1)$.
\end{enumerate}
\end{prop}

\begin{proof}
For (a), note that $(1,\ldots,1),(a,\ldots,a,1)=\alpha\circ\ldots\circ\alpha((1,\ldots,1),1)$. Then the statement follows from Proposition \ref{prop:An}, (2), and Proposition \ref{prop:induction}, (1). 

For (b), note that $(1,\ldots,1),(a,\ldots,a,a-1)=\beta((1,\ldots,1),(a,\ldots,a,1))$. Then the statement follows from (a) and Proposition \ref{prop:induction}, (1).
\end{proof}

\subsection{Divisorial extractions of singular curves in smooth $3$-folds}		%%

Note that for the cases called Jerry and Tyke in Table \ref{tab:classification} we have $\mathbf{a}=(a)$. The blow up $Y\rightarrow\mathbb{A}^3$ of the smooth curve $C_{(a)}\subset D_{12}$ is smooth and by Proposition \ref{prop:divisor} we have $\tilde{D}_{13}\simeq A_{a-1}$. Hence, we are led to consider divisorial extractions of singular curves in $\mathbb{A}^3$. This case has been studied by Ducat \cite{Ducat} and he provided an algorithm that proves the existence of divisorial extractions for an infinite family of singular curves:

\begin{prop}[\cite{Ducat}, Theorem 0.3]
\label{prop:ducat}
Let $m,k\geq 1$, and define a sequence of integers by the recurrence relation
\[ Q_0 = 0, Q_1 = 1, Q_{j+1}=mQ_k-Q_{j-1} \ \forall j \geq 1. \]
Let $a=Q_k+Q_{k-1}$ and $r=Q_k+Q_{k+1}$. Let $C\subset A_{r-1}$ be a curve with $m+2$ branches meeting the singular point of $A_{r-1}$ with multiplicity $Q_k$ each.
\footnote{If $A_{r-1}=\{xy=z^r\}\subset\mathbb{A}^3$, then such a branch can be given by $\{x=z^{Q_k}\}\subset A_{r-1}$.}
Then there exists a divisorial extraction $Z\rightarrow\mathbb{A}^3$ of $C$ such that $Z$ has an isolated singularity of type $\tfrac{1}{r}(1,-1,a)$.
\end{prop}

\begin{cor}
\label{cor:ducat2}
Let $C\subset A_m$ be an $m+2$-fold intersection of lines. Then there exists a divisorial extraction $Z\rightarrow\mathbb{A}^3$ of $C$ such that $Z$ has an isolated singularity of type $\tfrac{1}{m+1}(1,-1,-1)$.
\end{cor}

\begin{proof}
This is the case $k=1$ of Proposition \ref{prop:ducat}.
\end{proof}

\begin{prop}[\cite{Ducat}, \S3.4.1]
\label{prop:ducat2}
Let $C\subset A_2$ be a curve with $4$ branches that meet the singular point of $A_2$ with multiplicities $(1,1,1,2)$. Then there exists a divisorial extraction $Z\rightarrow\mathbb{A}^3$ of $C$ such that $Z$ has an isolated singularity of type $\tfrac{1}{3}(1,-1,-1)$.
\end{prop}

\begin{cor}
\label{cor:ducat}
Let $C\subset A_2$ be a curve with $4$ branches that meet the singular point of $A_2$ with multiplicities $(1,1,1,1)$ or $(1,1,1,2)$ or $(1,2,2,2)$ or $(2,2,2,2)$. Then there exists a divisorial extraction $Z\rightarrow\mathbb{A}^3$ of $C$ such that $Z$ has an isolated singularity of type $\tfrac{1}{3}(1,-1,-1)$.
\end{cor}

\begin{proof}
The case $(1,1,1,1)$ is Corollary \ref{prop:ducat2} with $m=2$. The case $(1,1,1,2)$ is Proposition \ref{prop:ducat2}. The cases $(1,2,2,2)$ and $(2,2,2,2)$ follow from the cases $(1,1,1,2)$ and $(1,1,1,1)$, since $\{x=z^1\}$ and $\{x=z^2\}$ are isomorphic in $A_2=\{xy=z^3\}$ via exchanging $x$ and $y$.
\end{proof}

\begin{expl}
Consider a triple of intersection of lines $C\subset A_1\subset Y=\mathbb{A}^3$. This is the case $m=1$ of Corollary \ref{cor:ducat2}. This example was first described in \cite{PR}. A geometric interpretation of the divisorial extraction $Z\rightarrow\mathbb{A}^3$ was given by Hironaka as follows, see \cite{Ducat}, Remark 3.5.

Let $Y'$ be the blow up of $Y=\mathbb{A}^3$ at the singular point $P$ of the $A_{a-1}$ containing $C$ followed by the blowup of the proper transform of $C$. The exceptional locus has two components $E$ and $E'$ dominating $P$ and $C$ respectively. Assuming the tangent directions of the branches of $C$ at $P$ are distinct then $E$ is a del Pezzo surface of degree $6$. Consider the three $-1$-curves of $E$ that don't lie in the intersection $E \cap E'$. They have normal bundle $\mathcal{O}_{Y'}(-1)\oplus\mathcal{O}_{Y'}(-1)$ so we can flop them. After the flop $E$ becomes a plane $\hat{E} \cong \mathbb{P}^2$ with normal bundle $\mathcal{O}_{\hat{Y}}(-2)$, so we can contract it to a point $\hat{P}$ to get $Z$ with a $\tfrac{1}{2}$-quotient singularity. The ordinary blowup $\textup{Bl}_{C}Y$ along $C$ is the midpoint of this flop and we end up with the following diagram.
\begin{equation*}
\begin{xy}
\xymatrix{
& E \cup E' \subset Y' \ar[dl]\ar[dr]\ar@{--}[rr]^{\textup{flop}} & & \hat{E} \cup \hat{E}' \subset \hat{Y} \ar[dl]\ar[dr] & \\
P \subset C \subset Y & & \textup{Bl}_{C}Y & & \hat{P} \subset \hat{E}' \subset Z
}
\end{xy}
\end{equation*}

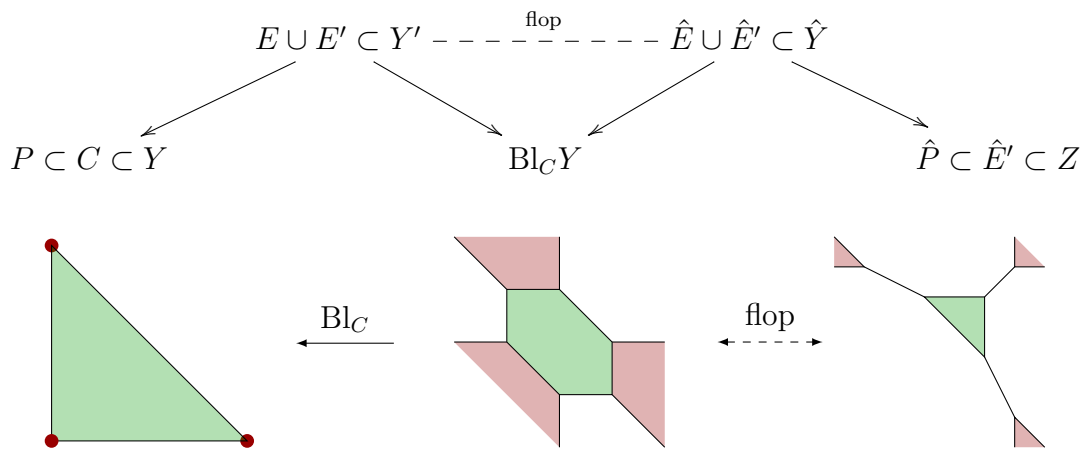
\begin{figure}[h!]
\centering
\begin{tikzpicture}[scale=1.3]
\fill[black!40!red] (2,0) circle (2pt);
\fill[black!40!red] (0,2) circle (2pt);
\fill[black!40!green,opacity=.3] (0,0) -- (2,0) -- (0,2);
\draw (0,0) -- (2,0) -- (0,2) -- cycle;
\fill[black!40!red] (0,0) circle (2pt);
\draw[<-] (2.5,1) -- (3,1) node[above]{$\textup{Bl}_{C}$} -- (3.5,1);
\draw (4,1);
\end{tikzpicture}
\begin{tikzpicture}[scale=.7]
\fill[black!40!red,opacity=.3] (5,3) -- (6,2) -- (7,2) -- (7,3);
\fill[black!40!red,opacity=.3] (9,1) -- (8,1) -- (8,0) -- (9,-1);
\fill[black!40!red,opacity=.3] (7,-1) -- (7,0) -- (6,1) -- (5,1);
\fill[black!40!green,opacity=.3] (6,1) -- (7,0) -- (8,0) -- (8,1) -- (7,2) -- (6,2);
\draw (6,1) -- (7,0) -- (8,0) -- (8,1) -- (7,2) -- (6,2) -- cycle;
\draw (6,1) -- (5,1);
\draw (7,0) -- (7,-1);
\draw (8,0) -- (9,-1);
\draw (8,1) -- (9,1);
\draw (7,2) -- (7,3);
\draw (6,2) -- (5,3);
\draw[<->,dashed] (10,1) -- (11,1) node[above]{flop} -- (12,1);
\end{tikzpicture}
\begin{tikzpicture}[scale=.4]
\fill[black!40!red,opacity=.3] (12,3) -- (13,3) -- (12,4);
\fill[black!40!red,opacity=.3] (18,-3) -- (18,-2) -- (19,-3);
\fill[black!40!red,opacity=.3] (18,4) -- (18,3) -- (19,3);
\fill[black!40!green,opacity=.3] (15,2) -- (17,0) -- (17,2);
\draw (15,2) -- (17,0) -- (17,2) -- cycle;
\draw (15,2) -- (13,3);
\draw (12,3) -- (13,3) -- (12,4);
\draw (17,0) -- (18,-2);
\draw (18,-3) -- (18,-2) -- (19,-3);
\draw (17,2) -- (18,3);
\draw (18,4) -- (18,3) -- (19,3);
\end{tikzpicture}
\caption{Exceptional divisors of the blow up of a point $P$, followed by the blow up of the curve $C$ and a flop. The divisors $E$, $\hat{E}$ dominating $P$ are shown in green and the divisors $E$, $\hat{E}'$ dominating $C$ in red.}
\label{fig:hironaka}
\end{figure}
\end{expl}

We give a similar geometric interpretation for the other cases of Corollary \ref{cor:ducat2}. The construction is based on an idea by Alessio Corti and was worked out together with him.

\begin{con}
Consider an $m+2$-fold intersection of lines $C\subset A_m\subset Y=\mathbb{A}^3$. Write $A_m=\{xy=z^{m+1}\}$. Let $L_1=\{y=z=0\}$ be the $x$-axis and $L_2=\{x=z=0\}$ the $y$-axis. The irreducible components of $C$ are of the form $C_i=\{x=\lambda_iz\}\subset A_m$, for some $\lambda_i\neq 0$. They intersect each other transversely, intersect the $x$-axis with multiplicity $1$, and the $y$-axis with multiplicity $m$, respectively. Let $Y'$ be the blow up of $Y=\mathbb{A}^3$ at the singular point of $A_m$, i.e. the origin. Let $E=\mathbb{P}^2$ be the exceptional divisor. Then $E$ intersects the proper transform of $A_m$ in the proper transforms $\tilde{L}_1$ and $\tilde{L}_2$. The blow up reduces the intersection multiplicity by $1$, hence the proper transforms $\tilde{C}_i$ of the irreducible components of $C$ are disjoint. They are disjoint from $\tilde{L}_1$ and intersect $\tilde{L}_2$ with multiplicity $m-1$. The normal bundle of $L_1$ in $Y'$ is $N_{L_1/Y'}\simeq\mathcal{O}_{\mathbb{P}^1}(-1) \oplus \mathcal{O}_{\mathbb{P}^1}(1)$. Let $\tilde{\tilde{Y}}$ be the $(m-1)$-weighted blow up of $\tilde{Y}$ along the line $L_1$, such that the exceptional divisor is $\tilde{E}=\mathbb{P}(\mathcal{O}(-1)\oplus\mathcal{O}(m-1))$, which is isomorphic to the Hirzebruch surface $\mathbb{F}_m$. The proper transforms of $\tilde{C}_i$ intersect $\tilde{E}$ in $m+2$ distinct points $P_i$. Through each $m+1$ of the $m+2$ points $P_i$, there is a section $S_i$ of $E'\simeq\mathbb{F}_m$. The situation is summarized in Figure \ref{fig:jerryexc}, which should be compared with the left picture of Figure \Ref{fig:hironaka}.

\begin{figure}[h!]
\begin{tikzpicture}[scale=2.5]
\draw (0,0) -- (1,0) -- (0,1) -- (0,0);
\draw (0,0) -- (1,0) -- (3,-1) -- (0,-1) -- (0,0);
\draw[blue] (0,1) -- (0,0) to[bend right=10] (3,-1);
\draw[black!40!green] (.55,-.4) -- (.6,-.31) to[bend left=90] (1.2,-.565) to[bend right=90] (1.8,-.76) -- (1.82,-.7);
\fill[red] (.6,-.31) circle (.7pt);
\fill[red] (1.2,-.565) circle (.7pt);
\fill[red] (1.8,-.76) circle (.7pt);
\fill[red] (2.4,-.9) circle (.7pt);
\end{tikzpicture}
\caption{The excpetional locus $E \cup \tilde{E} = \mathbb{P}^2 \cup \mathbb{F}_{a-1}$ in $\tilde{\tilde{Y}}$, showing the proper transforms of $L_1,L_2$ (blue), the points $P_i$ of intersection with $C_{\mathbf{b}}$ (red) and the sections $S_i$ (green).}
\label{fig:jerryexc}
\end{figure}
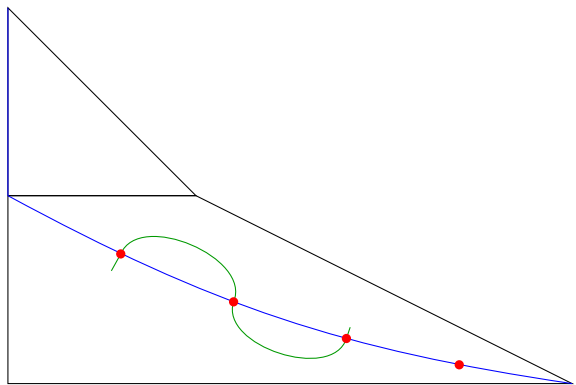

Let $Y'$ be the blow up of $\tilde{\tilde{Y}}$ along the disjoint curves $\tilde{C}_i$. The total transform $E'$ of $\tilde{E}$ is isomorphic to the blow up of $\mathbb{F}_m$ in the $m+2$ points $P_i$. The proper transforms $S'_i$ of the sections $S_i$ have normal bundle $N_{S'_i/Y'}=\mathcal{O}(-1)\oplus\mathcal{O}(-1)$, hence they are floppable.

Let $\hat{Y}$ be the flop of $Y'$ along the $S'_i$. After the flop, the divisor $E'$ becomes $\hat{E}\cong\mathbb{F}_m$ and the exceptional lines over the points $P_i$ become sections in $\hat{E}\cong\mathbb{F}_m$. This is similar to the case $m=1$ (see Figure \ref{fig:hironaka}), where the exceptional $\mathbb{P}^2$ after the blow up at three points and the flop of three lines became a $\mathbb{P}^2$ again. Here it comes from a symmetry between the points $P_i$ and the sections $S_i$, respectively the exceptionals over $P_i$ and the proper transforms of $S_i$.

The flop is given by a blow up and a contraction $Y' \leftarrow \tilde{Y}' \rightarrow \hat{Y}$. We have $\tilde{S}'_i \simeq \mathbb{P}^1 \times\mathbb{P}^1 \subset \tilde{Y}'$ mapping to $S'_i\simeq\mathbb{P}^1$ in $Y'$ and to $\hat{S}_i \simeq \mathbb{P}^1$ in $\hat{Y}$. Let $F_i$ be the fiber trough $P_i$ in $\tilde{\tilde{Y}}$ and write $F'_i$, $\tilde{F}'_i$, $\hat{F}_i$ for the respective proper transforms. Note that $\tilde{F}'_i \cdot \tilde{S}'_j = \delta_{ij}=\big\{\begin{smallmatrix}1,&i=j\\0,&i\neq j\end{smallmatrix}$, since for $i\neq j$ the curves $F_i$ and $S_j$ meet transversely in a point $P_i$ that gets blown up.

We compute the normal bundle $N_{\hat{E}/\hat{Y}}$. The Picard group of $\mathbb{F}_m$ is generated by the class of a fiber $f$ and the class of the special section $e$, with $f^2=0,e^2=-m,f\cdot e=1$. 

We have
\[ N_{\hat{E}/\hat{Y}} \cdot f = (\hat{E}\cdot\hat{F}_i)_{\hat{Y}}=(\tilde{E}'\cdot \tilde{F}'_i)_{\tilde{Y}'_i} \]
and
\[ -1 = N_{\tilde{E}/\tilde{\tilde{Y}}}\cdot f = (\tilde{E}\cdot F_i)_{\tilde{\tilde{Y}}} = (E'\cdot F'_i)_{Y'} = ((\tilde{E}'+\sum \tilde{S}_j)\cdot \tilde{F}'_i)_{\tilde{Y}'} = (\tilde{E}'\cdot\tilde{F}'_i)_{\tilde{Y}'}+1, \]
hence $N_{\hat{E}/\hat{Y}} \cdot f = -2$. Moreover, we have
\[ N_{\hat{E}/\hat{Y}} \cdot e = N_{\tilde{E}/\tilde{\tilde{Y}}} \cdot e = m-1. \]
This shows $N_{\hat{E}/\hat{Y}}\simeq\mathcal{O}_{\mathbb{F}_m}(-(m+1)f-2e)$.

We show that the union $E\cup \hat{E} = \mathbb{P}^2 \cup \mathbb{F}_m$ can be contracted, and that the contraction has an isolated singularity of type $\tfrac{1}{m+1}(1,-1,-1)$. To do this we show that $E=\mathbb{P}^2$ and $\hat{E}=\mathbb{F}_m$ with the given normal bundles can obtained as toric divisors in a toric variety. Then we show that the toric blow down of $E\cup \hat{E}$ describes a toric variety with the claimed singularity.

Consider the fan with rays generated by
\[ \rho_0=(-1,-1,-1),\rho_1=(1,0,0),\rho_2=(0,1,0),\rho_3=(1,1,m+1),\rho'=(1,1,m),\rho=(1,1,m-1). \]
This can be obtained by a weighted blow up of $\mathbb{P}^3$ at a torus fixed point followed by the blow up along two torus fixed lines. We took a complete fan, i.e. a compact variety, to be able to use intersection theory. The divisor corresponding to the ray $\rho=(1,1,m-1)$ is isomorphic to $\mathbb{P}^2$. To see this, note that the normal fan to $(1,1,m-1)$ is generated by 
\[ (1,0,0)+\mathbb{R}(1,1,m-1), \quad (0,1,0)+\mathbb{R}(1,1,m-1), \quad (-1,-1,0)+\mathbb{R}(1,1,m-1), \]
hence isomorphic to the fan generated by $(1,0),(0,1),(-1,-1)$, which is the fan of $\mathbb{P}^2$. Similarly, the divisor corresponding to the ray $\rho'=(1,1,m)$ is isomorphic to $\mathbb{F}_m$, since modulo $\mathbb{R}(1,1,m)$ we have
\[ \rho_2=(0,1,0), \quad \rho_3=(0,0,1), \quad \rho=(0,0,-1), \quad \rho_1=(0,-1,-m). \]
We have $N_{D'/X} \cdot f = D'^2\cdot D_2$ and $N_{D'/X} \cdot e = D'^2\cdot D$. Comparing the components of the rays we get the equations
\begin{eqnarray*}
-D_0 + D_1 + D_3 + D' + D &=& 0, \\
-D_0 + D_2 + D_3 + D' + D &=& 0, \\
-D_0 + (m+1) D_3 + mD' + (m-1)D &=& 0.
\end{eqnarray*}
From the first equation we get $D'=D_0-D_1-D_3-D$, hence
\[ N_{D'/X} \cdot f = D'D_2(D_0-D_1-D_3-D) = 0-0-1-1 = -2. \]
We have $D'D_2D_0=0$ and $D'D_2D_1=0$, since the corresponding rays don't span a cone in the fan, and $D'D_2D_3=D'D_2D=1$, because the corresponding rays span a standard cone (with determinant $1$). From the third equation minus $m-1$ times the second equation we get $D'=(m-1)D_2-2D_3-(m-2)D_0$, hence
\[ N_{D'/X} \cdot e = D'D((m-1)D_2-2D_3-(m-2)D_0) = (m-1)-0-0 = m-1. \]
We have $D'DD_2=1$, since the corresponding rays span a standard cone, $D'DD_3=0$, since the rays lie on a plane, and $D'DD_0=0$, since the corresponding rays don't span a cone in the fan. As above, $N_{D'/X} \cdot f=-2$ and $N_{D'/X} \cdot e=m-1$ imply $N_{D'/X}=\mathcal{O}_{D'}(-(m+1)f-2e)$.

Now we can forget the cones $\rho$ and $\rho'$, which corresponds to a contraction of $D\cup D'$. The resulting fan contains a cone spanned by $(1,0,0),(0,1,0),(1,1,m+1)$. By \cite{CLS}, Theorem 11.4.21, this gives a singularity of type $\tfrac{1}{m+1}(1,-1,-1)$, as claimed.
\end{con}

\subsection{Jerry and Tyke}								%%
\label{S:jerrytyke}

Recall the collections of divisorial extractions from Construction \ref{con:procedure} and recall Conjecture \ref{conj:divex}.

\begin{prop}
\label{prop:jerrytyke}
For the cases Jerry and Tyke in Table \ref{tab:classification}, Conjecture \ref{conj:divex}, (1), is true.
\end{prop}

\begin{proof}
Let $(\mathbf{a},\mathbf{b})$ be a pair corresponding to a case called Jerry or Tyke. From Table \ref{tab:classification} we see that $(\mathbf{a},\mathbf{b})$ can be obtained via repeated application of $\alpha$ and $\beta$ (see Proposition \ref{prop:degree}) from $(3),(1,1,1,2)$ or $(a),(1,\ldots,1)$, where the second tuple has length $a+1$. We use Proposition \ref{prop:induction}, (1), to reduce to these cases.

Let $Y \rightarrow \mathbb{A}^3$ be the blow up of the smooth curve $C_{(a)}\subset D_{12}$. Then $Y$ is smooth. By Proposition \ref{prop:divisor}, $\tilde{D}_{13}\simeq A_{a-1}$. Then $\tilde{C}_{\mathbf{b}}\subset A_{a-1}\subset Y=\mathbb{A}^3$ has $a+1$ branches intersecting the singular point of $A_{a-1}$ with multiplicities $(1,\ldots,1)$, or we have $a=3$ and $(1,1,1,2)$. For the first case, we get a divisorial extraction $Z\rightarrow Y$ as claimed from Corollary \ref{cor:ducat2}, for the second case we get it from Proposition \ref{prop:ducat2}.
\end{proof}

%%%% Literature

%\newgeometry{left=2cm, right=2cm, top=2.5cm, bottom=2cm}

\end{document}